\documentclass[a4paper,12pt]{amsart} 
\usepackage{amssymb, amsmath, amsthm, color}
\usepackage{tikz-cd}
\usepackage[colorlinks=true, breaklinks=true, linkcolor=black, citecolor=blue, urlcolor=red]{hyperref} 
\usepackage[english]{babel}
\definecolor{cadmiumgreen}{rgb}{0.0, 0.42, 0.24}

\numberwithin{equation}{section}
\newtheorem{lettertheo}{Theorem}

\newtheorem{theorem}{Theorem}[section]

\newtheorem{lemma}[theorem]{Lemma}
\newtheorem{corollary}[theorem]{Corollary}
\newtheorem{proposition}[theorem]{Proposition}

\theoremstyle{definition} 
\newtheorem{definition}[theorem]{Definition}
\newtheorem{notation}[theorem]{Notation}
\newtheorem{remark}[theorem]{Remark}
\newtheorem{example}[theorem]{Example}
\newcommand{\tbb}{\textcolor{black}}

\newcommand{\act}{\curvearrowright}

\DeclareMathOperator{\ad}{ad}
\newcommand{\al}{\alpha}

\DeclareMathOperator{\Aut}{Aut}
\newcommand{\cC}{\mathcal C}

\newcommand{\fC}{\mathfrak C}
\newcommand{\cCF}{\mathcal{CF}}
\DeclareMathOperator{\Coc}{Coc}
\newcommand{\cD}{\mathcal D}

\DeclareMathOperator{\End}{End}

\newcommand{\cF}{\mathcal F}
\DeclareMathOperator{\Fix}{Fix}
\DeclareMathOperator{\Frac}{Frac}
\newcommand{\fs}{\{0,1\}^{(\N)} }
\newcommand{\is}{\{0,1\}^{\N} }

\newcommand{\cG}{\mathcal G}

\DeclareMathOperator{\Gr}{Gr}

\newcommand{\ga}{\gamma}
\newcommand{\Ga}{\Gamma}

\DeclareMathOperator{\Hilb}{Hilb}

\DeclareMathOperator{\Hom}{Hom}
\DeclareMathOperator{\Homeo}{Homeo}
\newcommand{\cI}{\mathcal I}
\DeclareMathOperator{\id}{id}
\DeclareMathOperator{\Inn}{Inn}
\DeclareMathOperator{\Isom}{Isom}

\newcommand{\cK}{\mathcal K}

\newcommand{\La}{\Lambda}

\DeclareMathOperator{\Leaf}{Leaf}
\DeclareMathOperator{\Leb}{Leb}

\newcommand{\N}{\mathbf{N}}
\newcommand{\cN}{\mathcal{N}}
\newcommand{\NCV}{N_{H(\fC)}(V)}

\DeclareMathOperator{\Out}{Out}
\newcommand{\ot}{\otimes}
\newcommand{\ov}{\overline}

\newcommand{\Q}{\mathbf{Q}}
\newcommand{\Qt}{{\mathbf{Q}_2}}

\newcommand{\cR}{\mathcal{R}}
\newcommand{\cS}{\mathcal S}

\DeclareMathOperator{\Stab}{Stab}

\DeclareMathOperator{\supp}{supp}
\newcommand{\fT}{\mathfrak T}
\newcommand{\ti}{\tilde}

\newcommand{\cU}{\mathcal U}

\newcommand{\varep}{\varepsilon}

\newcommand{\Z}{\mathbf{Z}}

\begin{document}

\title[Classification of fraction groups II]{Classification of Thompson related groups arising from Jones technology II}
\thanks{
AB is supported by the Australian Research Council Grant DP200100067 and a University of New South Wales Sydney start-up grant.}
\author{Arnaud Brothier}
\address{Arnaud Brothier\\ School of Mathematics and Statistics, University of New South Wales, Sydney NSW 2052, Australia}
\email{arnaud.brothier@gmail.com\endgraf
\url{https://sites.google.com/site/arnaudbrothier/}}
\maketitle

\begin{abstract}
In this second article, we continue to study classes of groups constructed from a functorial method due to Vaughan Jones. 
A key observation of the author shows that these groups have remarkable diagrammatic properties that can be used to deduce their properties.
Given any group and two of its endomorphisms, we construct a semidirect product.
In our first article dedicated to this construction, we classify up to isomorphism all these semidirect products when one of the endomorphisms is trivial and described their automorphism group.

In this article we focus on the case where both endomorphisms are automorphisms. 
The situation is rather different and we obtain semidirect products where the largest Richard Thompson's group $V$ is acting on some discrete analogues of loop groups.
Note that these semidirect products appear naturally in recent constructions of quantum field theories. 
Moreover, they have been previously studied by Tanushevski and can be constructed via the framework of cloning systems of Witzel-Zaremsky. 
In particular, they provide examples of groups with various finiteness properties and possible counterexamples of a conjecture of Lehnert on co-context-free groups.

We provide a partial classification of these semidirect products and describe explicitly their automorphism group.
Moreover, we prove that groups studied in the first and second articles are never isomorphic to each other nor admit nice embeddings between them.
We end the article with an appendix comparing Jones technology with Witzel-Zaremsky's cloning systems and with Tanushevski's construction.
As in the first article, all the results presented were possible to achieve via a surprising rigidity phenomena on isomorphisms between these groups.
\end{abstract}


\section*{Introduction}

\subsection*{Jones technology}

Jones subfactor theory has been intimately linked with chiral conformal field theory (in short CFT) for more than two decades. 
Longo and Rehren proved that any CFT produces a subfactor \cite{Longo_Rehren_95_Nets_sf}.
Conversely, certain subfactors produce CFT but there is no general construction \cite{Bischoff17}.
This led Vaughan Jones to ask the fundamental question: 
``Does every subfactor have something to do with a conformal field theory?''
Over several decades, Jones tried to answer this question by looking for a general construction that would associate a CFT to \textit{any} subfactor.
In his last attempt Jones constructed from a subfactor a discretisation of a CFT where Thompson group $T$ was replacing the usual conformal group invariance \cite{Jones17-Thompson}.
The idea was to take a limit and obtain a honest classical CFT.
However, discontinuity issues appeared and the CFT goal stayed out of touch \cite{Jones16-Thompson}.

Jones realised that the technology he developed to construct a discrete CFT was very useful and practical for constructing group actions \cite{Brothier-19-survey}.
Given a nice category $\cC$ (a category with a left or right calculus of fractions, e.g.~a commutative cancellative monoid) there is a well-known process for constructing a so-called fraction group $G_\cC$ (in fact a fraction groupoid) from this category \cite[Section 1]{GabrielZisman67}.
Jones discovered that given \textit{any} functor $\Phi:\cC\to\cD$ ending in \textit{any} category $\cD$ one can construct very explicitly an action (called a \textit{Jones action}) $\pi_\Phi: G_\cC\act K_\Phi$ of the fraction group $G_\cC$ of the source category $\cC$ on a set $K_\Phi$ (sometime $K_\Phi$ will be rather an object in a category) constructed from the functor $\Phi.$
Moreover, properties of the target category $\cD$ are reflected in the properties of the Jones action $\pi_\Phi.$
For instance, a functor $\Phi:\cC\to\Hilb$ ending in the category of Hilbert spaces with isometries for morphisms provides a unitary representation $\pi_\Phi$ of the fraction group $G_\cC$.
We say that $\pi_\Phi$ is a \textit{Jones representation}.
A functor $\Phi:\cC\to\Gr$ ending in the category of groups provides an action by group automorphisms $\pi_\Phi:G_\cC\act K_\Phi$ on a group $K_\Phi.$

A key example that we will be considering in this article is given by the category $\cC=\cF$ of (rooted, finite, ordered, binary) forests whose fraction group is Richard Thompson's group $F$: the group of piecewise linear homeomorphisms of the unit interval having finitely many breakpoints at dyadic rationals and slopes power of two \cite{Cannon-Floyd-Parry96}.
Moreover, considering larger categories of forests equipped with permutations of their leaves one can build the larger Thompson groups $T$ and $V$ as fraction groups that are the analogues of $F$ but acting by homeomorphisms on the unit torus and on the Cantor space respectively. 
Recall that we have the chain of inclusions $F\subset T\subset V$.

\subsection*{Applications of Jones technology}

Jones technology for producing group actions is recent (less than ten years) but has already provided a number of interesting applications and new connections between fields of mathematics.
Functors $\Phi:\cF\to \cD$ from the category of forests going to a Jones planar algebra $\cD$ produce discrete CFT \cite{Jones17-Thompson,Jones16-Thompson}.
These are physical field theories sometime called Thompson field theory that are now studied in their own right \cite{Jones18-Hamiltonian,Osborne-Stiegemann19,Stiegemann19-thesis}.
This led to fruitful developments in operator algebraic quantum field theory \cite{Brot-Stottmeister-M19,Brot-Stottmeister-Phys}. 
Those functors produce, among others, new and interesting unitary representations of the Thompson groups using the inner product structure given by the quantum trace \cite{Jones19Irred, ABC19}.
For instance, Jones provided a one parameter family of irreducible unitary representations of $F$ that are pairwise non-isomorphic and where the parameter of the family is the celebrated index set of subfactors $\{4\cos(\pi/n):\ n\geq 3\}\cup [4,\infty)$ \cite{Jones19Irred}. 

Functors $\Phi:\cF\to \cD$ ending in the category of Conway tangles provides a way to constructs knots and links from Thompson group elements \cite{Jones17-Thompson}.
This is a fascinating and deep connection between Thompson group $F$, knot theory and quantum algebras producing among other new knots invariants \cite{Jones19-thomp-knot,Ren18-Thomp, Aiello20}.
This connection led to the discovery of the \textit{Jones subgroup} $\Vec F\subset F$ that is defined using orientation of knots and satisfies a number of remarkable properties \cite{Golan-Sapir17,Golan-Sapir21}.

Functors $\Phi:\cF\to\Hilb$ produce unitary representations of Thompson group $F$.
In particular, equipping $\Hilb$ with the direct sum for monoidal structure and considering monoidal functors Jones and the author obtained various continuous paths of unitary representations thus giving deformations \cite{Brot-Jones18-2}.
These functors provided connections between the Thompson group and the Cuntz algebra continuing previous works of Nekrashevych but also connections between the Thompson group and Connes non-commutative tori and more generally Connes-Landi spheres \cite{Cuntz77,Nekrashevych04,Connes-Landi01}.
Using monoidal functors $\Phi:\cF\to\Hilb$ where $\Hilb$ is now equipped with its classical monoidal structure Jones and the author obtained matrix coefficients vanishing at infinity for Thompson groups $F,T,V$ and reproved some results on their analytical properties such as absence of Kazhdan property (T) and the Haagerup property giving new proofs of theorems due to Reznikoff, Ghys-Sergiescu, Navas and Farley \cite{Reznikoff01,Ghys-Sergiescu87,Navas02,Farley03-H,Brot-Jones18-1}.

\subsection*{Constructing fraction groups}

Consider now a functor $\Phi:\cF\to\Gr$ from the category of forests to the category of groups.
The Jones action is then an action by group automorphism of Thompson group $F$ on a group $K_\Phi.$
Consider the semidirect product $K_\Phi\rtimes F$ which is a group.
The author made the key observation that $K_\Phi\rtimes F$ is in fact the fraction group of a certain category $\cC_\Phi$  \cite{Brothier19WP}.  
Morphisms of $\cC_\Phi$ can be diagrammatically described as forests whose leaves are labelled by elements of a group.
In particular, one can reapply Jones technology to the category $\cC_\Phi$ and construct unitary representations of the group $K_\Phi\rtimes F$.
This method was successfully applied to obtain new families of permutational (i.e.~non-standard) wreath products $K_\Phi\rtimes V$ (note that the larger Thompson group $V$ is acting) having the Haagerup property \cite{Brothier19WP}.
These examples are the first of their kind providing new lights on the comprehension of non-standard wreath products and complementing previous deep and beautiful general studies on wreath products and their analytical properties \cite{Cornulier06,CornulierStalderValette12,Cornulier18}. 

\subsection*{Motivations}

This leads to the main motivation of this present article: constructing and describing classes of groups for which Jones technology can be applied. 
Hence, we want to find groups $G$ equal to fraction groups $G_\cC$ of categories $\cC$. 
Moreover, we want to have a category $\cC$ with a simple structure in order to apply Jones technology efficiently.
From the key observation made above: natural candidates of fraction groups are given by semidirect products $G=K_\Phi\rtimes F$ constructed from a functor $\Phi:\cF\to\Gr.$
A first place to start is to consider the simplest possible functors: \textit{monoidal covariant} functors $\Phi:\cF\to\Gr$ (where the monoidal structure of $\Gr$ is given by direct sums).
These functors are in one to one correspondence with triples $(\Ga,\al_0,\al_1)$ where $\Ga$ is a group and $\al_0,\al_1\in\End(\Ga)$ are endomorphisms.
Moreover, using the fact that $\Phi$ is monoidal, we can extend the Jones action $\pi_\Phi:F\act K_\Phi$ into an action $\pi_\Phi:V\act K_\Phi$ of the larger Thompson group $V$.
Hence, we obtain a semidirect product $K_\Phi\rtimes V$. 
This larger semidirect product is again the fraction group of a category $\cC_\Phi$ of labelled forest.
The category $\cC_\Phi$ has a very simple structure because $\Phi$ is monoidal. 
It is then rather easy to construct functors $\Psi:\cC_\Phi\to \Hilb$ and thus, using Jones technology, unitary representations of the fraction group $K_\Phi\rtimes V$. 
Such functors $\Psi$ are in one to one correspondence with a class of triples $(R,H,\sigma)$ where $H$ is a Hilbert space, $R:H\to H\otimes H$ is an isometry and $\sigma:\Ga\to \cU(H)$ is a unitary representation so that $(R,\sigma)$ satisfies a compatibility condition, see \cite[Proposition 4.1]{Brothier19WP}.
This article studies the class of groups obtained from triples $(\Ga,\al_0,\al_1)$ where $\Ga$ is a group and $\al_0,\al_1$ are \textit{automorphisms} of $\Ga.$ 
We will see that the fraction group obtained is isomorphic to a semidirect product $L\Ga\rtimes V$. 
The group $L\Ga$ is equal to the group of continuous maps from the Cantor space to $\Ga$ (where $\Ga$ is equipped with the discrete topology). 
The action $V\act L\Ga$ is the one obtained from the classical action of $V$ on the Cantor space and with a twist depending on $\al_0$ and $\al_1.$

The second motivation of this work comes from quantum field theory.
As we briefly explained, Jones technology appeared from the desire of constructing CFT. 
It provides a different field theory with Thompson group $T$ replacing the spatial diffeomorphism group that can be studied in its own right: Thompson field theory. 
It resembles CFT but in a discrete way where the space-time, equal to a circle in CFT, is better described in this case by the Cantor space.
Stottmeister and the author have built field theories using Jones technology and ideas from loop quantum gravity \cite{Brot-Stottmeister-M19,Brot-Stottmeister-Phys}.
We found that the global group of symmetries together with the gauge group generate exactly fraction groups associated to triples $(\Ga,\id_\Ga,\id_\Ga)$. 
This motivates to better understand those groups as they will help understanding the physical model coming from a field theory. 

The third motivation is to produce Thompson-like groups satisfying remarkable properties and connecting this approach with pre-existing works.
The three Thompson groups $F,T,V$ satisfy very unusual behaviours.
There exists a number of families of groups, such as ours, that resemble the Thompson groups and follow some of their remarkable properties.
The class of groups considered in this article previously appeared in other studies.  
They were described and constructed in a different manner and used to answer other questions and problems than ours.
This is a great motivation for our work since it directly applies to those additional frameworks and creates potential future interplays between these different point of views.
Here is a brief presentation of those studies.
The interested reader can read the appendix of this article where we analyse and compare them more deeply.

Tanushevski considered certain categories $\cC$ of labelled forests and some associated semidirect products that are fraction groups: $K\rtimes F, K\rtimes T$ and $K\rtimes V$ \cite{Tanushevski16,Tanushevski17}.
He focused its study on $K\rtimes F$ and proved remarkable and beautiful results. 
For instance and constructed groups with certain finiteness properties, exhibited concrete finite presentations and described explicitly the lattice of normal subgroups of $K\rtimes F.$ 
The class of categories $\cC$ considered by Tanushevski is exactly the class of categories $\cC_\Phi$ obtained from monoidal covariant functors $\Phi:\cF\to\Gr$ that we are studying here. 
Hence, the fraction group $K\rtimes F$ of Tanushevski is the same semidirect product $K_\Phi\rtimes F$ associated to our functor $\Phi.$
Tanushevski and our results are interestingly disjoint and use different techniques. However, they both rely on very similar formalisms.

Using Zappa-Szep products (i.e.~bicrossed products) Brin constructed a braided version $BV$ of Thompson group \cite{Brin07-BraidedThompson,Brin06-BV2}.
Note that independently Dehornoy constructed the same group and other related groups using another approach \cite{Dehornoy06}.
The construction of Brin consists in producing a monoid of ``braided forests'' equals to a Zappa-Szep product between the monoid of finitely supported forests and the group $B_\infty$ of finitary braids over countably many strands.
By abstracting this process, Witzel and Zaremsky defined a general theory of \textit{cloning systems} for constructing groups that are Thompson-like groups and where $B_\infty$ is now replaced by a direct limit of groups \cite{Witzel-Zaremsky18,Zaremsky18-clone}. 
They were particularly motivated in studying finiteness properties of groups in which they have been very successful.
We will see in the appendix that their construction of groups and the one presented in this article are in essence different but share certain similarities.
In particular, the two classes of groups constructed using one or the other methods share a large common subclass.
The two methods have their own advantages/disadvantages and naturally complement each other. 
One can reinterpret results obtained using cloning system in the framework of this article and vice-versa.
For example, Witzel and Zaremsky proved that a number of fraction groups obtained from functors $\Phi:\cF\to\Gr$ are of type $F_n$ for some $n\geq 1$. 
From the functorial perspective developed in the formalism of Jones there are additional candidates for being of type $F_n$ that perhaps have appeared less natural to look at from the cloning system perspective.
Cloning systems have been used to answer other questions not just involving finiteness properties.
For instance, using cloning systems Ishida constructed new families of groups that are (left- or bi-) orderable \cite{Ishida17}.
Note that his results can easily be reinterpreted in our framework and again natural candidates of orderable groups can be given using Jones machinery.
Another exciting study concerns co-context-free groups (in short $co\cCF$ groups).
A conjecture of Lehnert modified using a theorem of Bleak, Matucci and Neunh\"{o}ffer states that any $co\cCF$ group is finitely generated and embeds in Thompson group $V$ \cite{Lehnert08-thesis,BleakMatucciNeunhoffer16}.
A class of possible counterexamples to this conjecture has been recently constructed using cloning systems in \cite{BZFGHM18}. 
They exactly correspond to groups obtained from triples $(\Ga,\id_\Ga,\al_1)$ in our framework where $\Ga$ is finite, $\id_\Ga$ is the identity automorphism of $\Ga$ and $\al_1$ is any automorphism of $\Ga.$ 
We provide a separate study of this class in Section \ref{sec:coCF}. 
We hope that the constructions and techniques presented in this article would provide a useful point of view for experts in cloning systems and would lead to future fruitful developments. For example, future studies regarding the following questions already addressed: classifications, computations of automorphism groups, descriptions of normal subgroups, analytical properties, constructions of deformations, connections with operator algebras, finiteness properties, orderability and co-word problem.

\subsection*{Details and main results}

We now enter into the details of this article and compare our present analysis with our previous one \cite{Brothier20-1}.
Consider the category of forests $\cF$ whose fraction group is Thompson group $F$.
We consider covariant monoidal functors $\Phi:\cF\to\Gr$ that are in one to one correspondence with triples $(\Ga,\al_0,\al_1)$ where $\Ga$ is a group and $\al_0,\al_1\in \End(\Ga).$
Such a functor produces a Jones action $F\act K$ on a group $K$ which extends to an action $V\act K$.
We consider the semidirect product obtained $K\rtimes V$. 
Our goal is to provide a clear description of $K\rtimes V$, decide how much $K\rtimes V$ depends on the triple $(\Ga,\al_0,\al_1)$ and describe the automorphism group of $K\rtimes V$.
%


In the first article, we studied semidirect products $K\rtimes V$ arising from triples $(\Ga,\al_0,\al_1)$ where one of the endomorphisms is trivial \cite{Brothier20-1}.
We proved that one could always assume that the nontrivial endomorphism was an automorphism and showed that $K\rtimes V$ was isomorphic to a restricted twisted permutational wreath product $\oplus_\Qt \Ga\rtimes V$ where $\Qt$ was the set of dyadic rationals in $[0,1)$ and $V\act \Qt$ was the restriction of the classical action of $V$ on the unit interval. 
The twist is induced by the automorphism, say $\al_0$, with formula:
$$v\cdot a(vx) = \al_0^{\log_2(v'(x))}(a(x)), \ v\in V, a\in\oplus_{\Qt}\Ga, x\in \Qt.$$
In particular, if $\al_0=\id_\Ga$, then the wreath product is untwisted and we say in that case that the fraction group $K\rtimes V$ is untwisted. 
We recover the class of groups previously considered by the author and mentioned earlier for which we proved that they have the Haagerup property when $\Ga$ has it (as a discrete group).
We provided a complete classification of these groups up to isomorphism and described their automorphism group in the untwisted case \cite{Brothier20-1}.
Note that all this analysis was possible to achieve because there are very few isomorphisms between groups in this class.
Indeed, an isomorphism $\theta: \oplus_\Qt \Ga\rtimes V\to \oplus_\Qt \ti\Ga\rtimes V$ will always decompose as follows:
$$\theta(av) = \kappa(a) \cdot c_v \cdot (\varphi v\varphi^{-1}), \ a\in \oplus_\Qt\Ga, v\in V$$
where $\kappa:\oplus_\Qt\Ga\to\oplus_\Qt\ti\Ga$ is an isomorphism, $c:V\to \oplus_{\Qt}\ti\Ga$ is a cocycle and $\varphi:\Q_2\to\Q_2$ is a map that extends to a homeomorphism of the Cantor space $\fC=\{0,1\}^{ \N }$ where $\Q_2$ corresponds to finitely supported sequences inside $\fC$.
But more important, $$\supp(\kappa(a)) = \varphi(\supp(\kappa(a))) \text{ for } a\in\oplus_\Qt\Ga$$ where $\supp$ denotes the support. 
We say that the morphism is \textit{spatial}.
The proof regarding spatiality was rather easy in the first article and the main difficult resided in describing all automorphisms of an untwisted fraction group.

In this present article, we consider semidirect products $K\rtimes V$ built from triples $(\Ga,\al_0,\al_1)$ where both of the endomorphisms are nontrivial. 
The situation is much more complex and cannot be reduced to automorphisms as in the previous article.
Although, using a trick due to Tanushevski one can always reduce to the case where $g\mapsto(\al_0(g),\al_1(g))$ is injective, see Proposition \ref{prop:Tanushevski}.
We quickly specialise to $\al_0,\al_1$ being automorphisms. 
We prove that $K\rtimes V$ is then isomorphic to $L\Ga\rtimes V$ where $L\Ga$ is the group of maps $a:\Qt\to\Ga$ that are \textit{locally constant} in the sense that there exists a standard dyadic partition $(I_1,\cdots,I_n)$ of $\Qt$ such that $a$ is constant on each interval $I_k, 1\leq k\leq n$.
For technical reasons, it will be convenient to embed $\Qt$ inside the Cantor space $\fC:=\{0,1\}^\N$ and extend elements of $L\Ga$ into functions from $\fC$ to $\Ga.$
Note that $L\Ga$ understood as a subgroup of the product $\prod_\fC\Ga$ is equal to the group of all continuous functions from $\fC$ to $\Ga$ where $\Ga$ is equipped with the discrete topology.
We call $L\Ga$ the \textit{discrete loop group} of $\Ga$ by analogy of loop groups in the Lie group context: the loop group of a Lie group $G$ is the group of all smooth maps from the circle to $G$ equipped with the pointwise multiplication.
The Jones action $V\act L\Ga$ is spatial and twisted by the automorphisms $\al_0,\al_1.$
The formula is 
$$(v\cdot a)(vx) := \tau_{v,x}(a(x)), \ v\in V, a\in L\Ga, x\in\Qt$$
where $(v,x)\mapsto \tau_{v,x}$ is a map valued in the subgroup of $\Aut(\Ga)$ generated by $\al_0,\al_1$, see Section \eqref{sec:loop}.
Note that this action extends to the full product $\prod_\Qt\Ga$ and thus the group $K\rtimes V$ is isomorphic to a certain subgroup of the unrestricted twisted permutational wreath product $\prod_\Qt\Ga\rtimes V$.

An important particular case is given by triples $(\Ga,\id_\Ga,\id_\Ga)$ producing $L\Ga\rtimes V$ where the Jones action $V\act L\Ga$ is untwisted: $(v\cdot a)(vx) =a(x), v\in V, a\in L\Ga, x\in\Qt.$
We say that the fraction group or the semidirect product $L\Ga\rtimes V$ is untwisted.
Elements of $L\Ga$ can be seen as a colouring of finitely many regions of the circle that are moved around by $V$.
As mentionned earlier, these groups previously appeared in physical models of Stottmeister and the author inspired by the work on Jones in field theory and by loop quantum gravity. \cite{Brot-Stottmeister-M19,Brot-Stottmeister-Phys}.
The physical space is the circle approximated by dyadic rationals. The gauge group is nothing else than the discrete loop group $L\Ga$ and the spatial group is Thompson group $T$ (the one in between $F$ and $V$) acting geometrically by rotations and local scale transformations.
The spatial symmetry group normalises $L\Ga$ and together they generate $L\Ga\rtimes T$.

\tbb{Similarly to the previous article, we find the surprising fact} that isomorphisms between two semidirect products $\theta:L\Ga\rtimes V\to L\ti \Ga\rtimes V$ are (up to multiplying by a centrally valued morphism) spatial in a similar sense to what was previously described:
$$\theta(av) = \kappa(a) \cdot c_v \cdot (\varphi v\varphi^{-1}) \text{ and } \supp(\theta(a))=\varphi(\supp(a)), \ a\in L\Ga, v\in V$$
but with the important difference that $\varphi\in \Homeo(\fC)$ is no longer stabilizing $\Qt$ in general, see Theorem \ref{theo:support} and Remark \ref{rem:support} for more precise statements.
This makes the analysis much harder especially in the twisted case.
The proof of this rigidity phenomena on isomorphism is more difficult than in the single nontrivial endomorphism case of the previous article. 
It is the most technical proof of the present article.

Using the structure of isomorphisms we establish a partial classification of the class of groups considered.

\begin{lettertheo}\label{lettertheo:A}(Corollary \ref{cor:isom})
Consider two groups $\Ga,\ti\Ga$ and two pairs of automorphisms $\al_0,\al_1\in\Aut(\Ga), \ti\al_0,\ti\al_1\in\Aut(\ti\Ga).$
Denote by $G=L\Ga\rtimes V$ and $\ti G=L\ti\Ga\rtimes V$ the associated twisted fraction groups.
The following assertions are true:
\begin{enumerate}
\item If $G\simeq \ti G$, then $\Ga\simeq\ti\Ga.$ 
\item Assume that $\ti\al_0,\ti\al_1$ are inner automorphisms. We have that $G\simeq\ti G$ if and only if $\Ga\simeq\ti\Ga$ and $\al_0,\al_1$ are inner automorphisms.
\end{enumerate}
\end{lettertheo}

Note that we obtain a complete classification of the subclass of fraction groups generated by triples $(\Ga,\al_0,\al_1)$ where the automorphisms are inner. In particular, two triples $(\Ga,\id_\Ga,\id_\Ga)$ and $(\ti\Ga,\id_{\ti\Ga},\id_{\ti\Ga})$ have their associated fraction groups isomorphic if and only if $\Ga\simeq\ti\Ga.$
We could not provide a thinner classification of the general case because of the existence of exotic homeomorphisms of the Cantor space normalising $V$ that are not stabilising $\Qt$ nor sending it to the other copy of dyadic rationals inside the Cantor space, see Remark \ref{rem:exotic-isom}.


Given a triple $(\Ga,\al_0,\al_1)$ we write $G(\Ga,\al_0,\al_1)$ for the group previously denoted by $K\rtimes V$.
We have been able to conduct deep analysis on the class of groups $G(\Ga,\al_0,\varep_\Ga)$ studied in the first article that are built from triples $(\Ga,\al_0,\varep_\Ga)$ where $\varep_\Ga: g\in\Ga\mapsto e_\Ga$ is the trivial endomorphism.
As mentioned earlier we proved that, if $\Ga$ has the Haagerup property and $\al_0$ is injective, then $G(\Ga,\al_0,\varep_\Ga)$ has the Haagerup property \cite{Brothier19WP}.
We have been trying to prove a similar theorem for the class of groups $G(\Ga,\id_\Ga,\id_\Ga)$ or more generally $G(\Ga,\al_0,\al_1)$ for $\al_0,\al_1\in\Aut(\Ga)$ but so far have been unable to do it.
After various failed attempts we started wondering if given $G(\Ga,\id_\Ga,\id_\Ga)$ one can find another group $\ti\Ga$ and an automorphism $\ti\al_0\in\Aut(\ti\Ga)$ satisfying $G(\Ga,\id_\Ga,\id_\Ga)\simeq G(\ti\Ga,\ti\al_0,\varep_{\ti\Ga})$ or satisfying that $G(\Ga,\id_\Ga,\id_\Ga)$ embeds inside $G(\ti\Ga,\ti\al_0,\varep_{\ti\Ga})$ in a nice way.
This would have permitted to deduce properties of $G(\Ga,\id_\Ga,\id_\Ga)$ from the study of $G(\ti\Ga,\ti\al_0,\varep_{\ti\Ga})$.
The following theorem proves that these hypothetical isomorphisms and nice embeddings never exist except in trivial situations:

\begin{lettertheo}(Theorem \ref{theo:noisom})
Consider two groups $\Ga,\ti\Ga$, an endomorphism $\al\in\End(\Ga)$ and two injective endomorphisms $\ti\al_0,\ti\al_1\in \End(\ti\Ga)$.
Consider the fraction groups $G=K\rtimes V$ and $\ti G=\ti K\rtimes V$ built via the triples $(\Ga,\al,\varep_\Ga)$ and $(\ti\Ga,\ti\al_0,\ti\al_1)$ respectively.

\begin{enumerate}
\item If $K$ or $\ti K$ is nontrivial, then there are no isomorphisms between $K\rtimes V$ and $\ti K\rtimes V$.
\item If $\ti K $ is nontrivial, then there are no injective $V$-equivariant morphisms from $\ti K$ to $K$.
\item If $K$ is nontrivial, then there are no $V$-equivariant morphisms from $K$ to $\ti K$.
\end{enumerate}
\end{lettertheo}

The proof of this theorem is not difficult. 
It relies on the fact that isomorphisms between fraction groups $\theta:K\rtimes V\to \ti K\rtimes V$ send $K$ onto $\ti K$ and then compare relative commutants $\{v\in V:\ va=va\}$ for some fixed $a$ in  $K$ and in $\ti K.$
Unfortunately, this theorem does not decide if the class of $co\cCF$ groups constructed in \cite{BZFGHM18} can be embedded in $V$ or not since we are considering isomorphism or $V$-equivariant morphisms.
However, our general analysis provides a new description of  these groups that we give in detail in Section \ref{sec:coCF}. 
Moreover, using Theorem A we provide in Corollary \ref{cor:coCF} a rather complete classification up to isomorphism of this class of $co\cCF$ groups.

The last section of the article is devoted to describe the automorphism group of $G=G(\Ga,\id_\Ga,\id_\Ga)\simeq L\Ga\rtimes V$ the group induced by a group $\Ga$ and the identity automorphism. 
We limited our analysis to this specific case. 
Treating the general case with $\al_0,\al_1$ being any automorphisms is not only more technical but will require to better understand homeomorphisms of the Cantor space $\fC$. 
In particular, the homeomorphisms normalising $V$ but mixing in a nontrivial way the classes of $\fC/V.$
We leave this for future study.
Note that since $\al_0=\al_1=\id_\Ga$ we have that the semidirect product $L\Ga\rtimes V$ is untwisted
and thus the Jones action $V\act L\Ga$ is purely spatial:
$$(v\cdot a)(vx) = a(x) , \ v\in V, a\in L\Ga, x\in\Qt.$$
We start by exhibiting four different kinds of automorphisms that we call \textit{elementary}.
They are the following:
The group $\NCV=\{\varphi\in\Homeo(\fC):\ \varphi V\varphi^{-1}\}$ acts spatially on $G$:
$$\varphi\cdot(av) :=  (a\circ\varphi^{-1})\cdot (\varphi v\varphi^{-1}),\ \varphi\in\NCV, a\in L\Ga,v\in V.$$
The group $\Aut(\Ga)$ acts diagonally on $G$:
$$\beta(av) = \ov\beta(a) v, \ \ov\beta(a)(x):=\beta(a(x)), \ \beta\in\Aut(\Ga), a\in L\Ga, v\in V, x\in\Qt.$$
This provides an action $$A:\NCV\times\Aut(\Ga)\to\Aut(G), (\varphi,\beta)\mapsto A_{\varphi,\beta}.$$
The normaliser subgroup $N(G):=\{f\in\prod_\Qt\Ga:\ fGf^{-1}=G\}$ acts by adjoint action: $\ad:N(G)\act G$.
The formula giving the last class of elementary automorphisms is less obvious than the three previous ones.
Using slopes, for any central element $\zeta\in Z(\Ga)$ we construct a cocycle 
$$v\in V\mapsto s(\zeta)_v \text{ where } s(\zeta)_v(x):= \zeta^{\log_2(v'(v^{-1}x)}, x\in\Qt.$$
This induces an action of $Z(\Ga)$ on $G$ :
$$F_\zeta(av) := a\cdot s(\zeta)_v\cdot v, \ \zeta\in Z(\Ga), a\in L\Ga, v\in V.$$
We prove that any automorphism of $G$ is a product of these four elementary ones.
Moreover, we provide an explicit description of $\Aut(G)$ as a quotient of a semidirect product.

\begin{lettertheo}(Theorem \ref{theo:Aut(G)loop})
Let $\Ga$ be a group and $G:=L\Ga\rtimes V$ the associated untwisted fraction group. 
The formula
$$(\varphi,\beta)\cdot (\zeta, f):=(\beta(\zeta)^{k_\varphi} , \ov\beta(f)^\varphi\cdot \zeta^{\ga_\varphi})$$
defines an action by automorphisms of $\NCV\times\Aut(\Ga)$ on $Z(\Ga)\times N(G)/Z(\Ga),$
where $k_\varphi$ is a constant and $\ga_\varphi:\Qt\to\Z$ is a map (see Section \eqref{sec:decompoAut}) for all $$\varphi\in\NCV, \beta\in\Aut(\Ga), \zeta\in Z(\Ga), f\in N(G)/Z(\Ga).$$

The following map
\begin{align*}
\Xi: & ( Z(\Ga)\times N(G)/Z(\Ga)) \rtimes (\NCV\times \Aut(\Ga))\to \Aut(G)\\
& (\zeta , f , \varphi , \beta)\mapsto F_\zeta\circ\ad(f)\circ A_{\varphi,\beta} 
\end{align*}
is a surjective group morphism with kernel
$$\ker(\Xi)=\{ (e_{Z(\Ga)} , \ov g , \id_V , \ad(g^{-1}) ):\ g\in \Ga\}$$
where $\ov g\in N(G)/Z(\Ga)$ is the class of the constant map equal to $g$ everywhere and $\ad(g^{-1}):h\mapsto g^{-1}hg$ is the inner automorphism of $\Ga$ associated to $g\in \Ga$.
\end{lettertheo}

\subsection*{Sketch of the proof of the main technical result}

We end this introduction by giving a sketch of the proof of Theorem \ref{theo:support}. This theorem implies Theorem A and B and is a key result for proving Theorem C.
Consider two triples $(\Ga,\al_0,\al_1), (\ti\Ga,\ti\al_0,\ti\al_1)$ with associated groups $G=L\Ga\rtimes V$ and $\ti G=L\ti\Ga\rtimes V$.
Assume there exists an isomorphism $\theta:G\to \ti G.$
\tbb{Using a result from the first article we know that $\theta(L\Ga)=L\ti\Ga$. Moreover, using the fact that any automorphism of $V$ is induced by conjugation by a homeomorphism }of the Cantor space $\fC$ we deduce that $\theta$ can be decomposed as follows:
$$\theta(av) =\kappa(a)\cdot c_v \cdot \ad_\varphi(v),\ a\in L\Ga, v\in V$$
where $\kappa:L\Ga\to L\ti\Ga$ is an isomorphism, $V\ni v\mapsto c_v\in L\ti\Ga$ is a cocycle and $\ad_\varphi(v):=\varphi v \varphi^{-1}, v\in V$ where $\varphi$ is a homeomorphism of $\fC$ normalizing the Thompson group $V$.
From this we want to show that up to multiplying $\theta$ by a morphism $\zeta:G\to Z(\ti G)$ from $G$ to the center $Z(\ti G)$ we have that $\supp(\kappa(a)) =\varphi(\supp(a))$ for all $a\in L\Ga.$
We prove that if $a\in L\Ga$ and $x$ is not in the support of $a$, then $\kappa(a)(\varphi(x))$ must belong to the center of $\ti\Ga$. 
By describing certain centralizer subgroups we deduce that $\kappa(a)$ is constant outside of $\varphi(\supp(a))$ taking a value that is not only central but also invariant under the automorphisms $\ti\al_0$ and $\ti\al_1.$
\tbb{Now, constant maps valued in the center} of $\ti\Ga$ and invariant under $\ti\al_0,\ti\al_1$ are elements of $Z(\ti G).$
\tbb{Hence, we find that} if $a\in L\Ga$, then $\kappa(a) = \kappa^0(a)\cdot \zeta(a)$ where $\kappa^0(a)$ is supported in $\varphi(\supp(a))$ and $\zeta(a)\in Z(\ti G).$
It is then easy to conclude that $\kappa^0:L\Ga\to L\ti\Ga$ is an isomorphism satisfying $\supp(\kappa^0(a))=\varphi(\supp(a))$ for all $a\in L\Ga.$

\section*{Acknowledgement}
We thank the anonymous referee for helping us improving the quality of this article.

\section{Preliminaries}\label{sec:preliminary}

We will follow the notations of the previous article  \cite{Brothier20-1} that we refer to for more details along with \cite{Brothier19WP}.

\subsection{Rational points and slopes}\label{sec:slopes}
We write $\fC=\is$ for the Cantor space equal to all infinite sequences in $0,1$ equipped with the usual product topology. 
We write $\fs\subset \fC$ for the subset of finitely supported sequences that we may identify with finite words in $0,1$.
Recall that 
$$S:x=(x_n)_{n\in\N}\mapsto \sum_{n\in\N}\frac{x_n}{2^n}$$ 
defines a surjection from $\fC$ onto $[0,1]$ where each dyadic rationals of $(0,1)$ has exactly two preimages and all the other points have only one preimage.
Write $\Q_2$ the dyadic rationals of $[0,1)$ that we identify with the finitely supported sequences of $\fC$, hence $\Q_2\subset\fC$ corresponds to the inclusion $\fs\subset\is.$
Note that $\fC$ contains a copy of the dyadic rationals of $(0,1]$: the set of all sequences having finitely many $0$.
A standard dyadic interval (in short sdi) is a subset of $\fC$ of the form 
$$I=\{m_I\cdot y:\ y\in\is\}$$ 
where $m_I$ is a finite word: $I$ is the set of all the sequences with the prefix $m_I.$
We will often refer to $m_I$ as the \textit{word associated to} $I$.
Note that $I$ is an open and closed subset of $\fC$.
It is mapped by $S$ to an interval of the form $S(I)=[\frac{a}{2^b}, \frac{a+1}{2^b}]$ with $a,b$ natural numbers justifying the terminology.
For technical reasons we will often identify $I$ with the half-open interval $\dot S(I):=[\frac{a}{2^b}, \frac{a+1}{2^b})$ and also with $\dot S(I)\cap \Q_2.$
A standard dyadic partition (in short sdp) is a finite partition of $\fC$ made of sdi.
Similarly, we identify a sdp with the corresponding partitions of $[0,1)$ and $[0,1)\cap \Q_2$.

Consider two sdp $(I_k:\ 1\leq k\leq n)$ and $(J_k:\ 1\leq k\leq n)$ having the same number of sdi $n\geq 1.$
Consider the map $v:\fC\to\fC$ defined as 
$$v(m_{I_k}\cdot x) = m_{J_k}\cdot x, \ 1\leq k\leq n, x\in\is.$$
This a homeomorphism of $\fC$ and the set of all such $v$ (for all choices of two sdp with the same number of sdi) forms a group for the composition called Thompson group $V$.
Note that if we restrict to \textit{ordered} sdp, meaning that $\sup(I_k)=\inf(I_{k+1})$ for all $1\leq k\leq n-1$, then the set of all maps as above between two ordered sdp with the same number of sdi is a group isomorphic to Thompson group $F$.
Fix $v\in V$.
We say that a sdp $(I_k:\ 1\leq k\leq n)$ (resp. a sdi $I$) is adapted to $v$ if there exists a family of finite words $(m_k:\ 1\leq k\leq n)$ such that $v(m_{I_k}\cdot x) = m_k\cdot x$ for all $1\leq k\leq n$ and $x\in\is$ (resp. a finite word $m$ such that $v(m_I\cdot x) = m\cdot x$ for all $x\in\is$).

The group $V$ stabilises the subset $\Q_2\subset\fC$.
Consider the complement inside $\fC$ of all sequences having finitely many $0$'s.
This set is in bijection with $[0,1)$ via the map $S$ of above.
We have that this subset is stabilised by the action $V$. 
This provides a piecewise linear action of $V$ on $[0,1)$. Each element $v\in V$ has finitely many discontinuous points when acting on $[0,1)$ all appearing at dyadic rationals. This is the classical Thompson group action of $V$ on the unit interval $[0,1)$.

If $v\in V$ and $x\in \fC$, then there exists some finite words $m,w$ and an infinite sequence $y$ satisfying that $x=m\cdot y$ and that $v(m\cdot z)=w\cdot z$ for all infinite sequence $z.$
We say that the \textit{slope} or the \textit{derivation} of $v$ at $x\in\fC, S(x)\neq 1$, denoted $v'(x)$, is the ratio $\frac{ 2^{ |m| } }{ 2^{|w|}}$ where $|m|$ is the number of letters in the word $m$.
Note that this corresponds to the usual (right-)slope of the map $v:[0,1)\to [0,1)$ at the point $S(x)$. 
We extend the definition of $v'(x)$ at $x=1$ by defining $v'(1):=\lim_{x\to 1} v'(x)$.

We consider the $V$-orbits of $V\act\fC$.
Given any nonempty finite word $c\in\fs$ we consider the set of all $x\in\fC$ satisfying that $x$ is eventually periodic of period $c$: 
$$x=y\cdot c\cdot c\cdots=y\cdot c^\infty$$
for some $y\in\fs.$
We call this set the \textit{tail equivalence class} of $c$ and note that this is a $V$-orbit.
The subset $\Q_2\subset\fC$ is the tail equivalence class of the word $c=0$ of length one.
A point $x\in\fC$ is called \textit{rational} if it belongs to one tail equivalence class, i.e.~it is eventually periodic.
Observe that $x\in\fC$ is rational if and only if its projection in $[0,1]$ by $S$ is rational in the usual sense.
We denote the set of rational points of $\fC$ by $R$.

A theorem of Rubin implies the following proposition on the automorphism group of $V$ \cite{Rubin96}, see \cite[Section 3]{BCMNO19} for details.

\begin{proposition}\label{prop:NCV}
Consider $V$ as a subgroup of $\Homeo(\fC)$ the group of homeomorphisms of the Cantor space $\fC$. 
Let $\NCV$ be the normaliser subgroup of $V$ inside $\Homeo(\fC).$
We have that $$\NCV\to \Aut(V),\ \varphi\mapsto \ad_\varphi:v\mapsto \varphi v\varphi^{-1}$$
is an isomorphism.
\end{proposition}
We will freely identify $\Aut(V)$ with $\NCV.$
Note that the group $\NCV$ maps $V$-orbits to $V$-orbits (for the action $V\act\fC$) and in particular maps $\Qt$ onto a $V$-orbit. 
However, the image of $\Qt$ may be different from $\Qt$ which will make our study more technical. 
We will be constantly using the following fact (which is easy to prove):

If $\varphi\in\NCV$ and $I$ is a sdi, then $\varphi(I)$ is a finite union of sdi.
See \cite[Lemma 1.3]{Brothier20-1} for a proof.
We end this section by a proposition on slopes that we prove using fixed points.

\begin{proposition}\label{prop:slope} 
Let $R\subset\fC$ be the subset of rational points.
The following assertions are true:
\begin{enumerate}
\item A point $x\in\fC$ is in $R$ if and only if there exists $v\in V$ satisfying $v(x)=x$ and $v'(x)\neq 1.$
\item If $\varphi\in\NCV$, then $\varphi(R)=R.$
\item For any $x\in R$ and $\varphi\in\NCV$ there exists some prime words $m_x,m_{\varphi(x)}$ satisfying that $x$ and $\varphi(x)$ are in the tail equivalence class of $m_x$ and $m_{\varphi(x)}$ respectively.
Moreover, if $v\in V$ satisfies $v(x)=x, v'(x)\neq 1$, then 
$$\frac{\log_2((\varphi v \varphi^{-1})'(\varphi(x)))}{|m_{\varphi(x)}|} = \frac{\log_2(v'(x))}{|m_x|}$$
and $\log_2(v'(x))$ is a multiple of $|m_x|$.
\end{enumerate}
\end{proposition}

\begin{proof}

Proof of (1).

Denote by $S$ the set of $x\in\fC$ for which there exists $v\in V$ satisfying $v(x)=x$ and $v'(x)\neq 1.$
Consider $x\in R$. There exists a nonempty word $c\in\fs$ and a finite word $y$ such that $x=y\cdot c^\infty$ where $c^\infty$ denotes the infinite concatenation of $c$ with itself: $c^\infty= c\cdot c\cdot c\cdots.$
Consider $I:=\{ y\cdot z:\ z\in\is\}$ which is a sdi and consider the subinterval $J:=\{ y\cdot c\cdot z:\ z\in\is\}.$
Up to replacing $y$ by $y\cdot c$ we can always assume that $I$ is a proper subset of $\fC$.
Complete $I$ into a sdp $A_I$ and $J$ into a sdp $A_J$. 
Up to modifying them we can assume that $A_I$ and $A_J$ have the same number of sdi.
Let $v\in V$ such that $v$ is adapted to $A_I$ sending an sdi of $A_I$ onto a sdi of $A_J$ and such that $v(I)=J.$
We necessarily have that $v(y\cdot z) = y\cdot c\cdot z$ for all $z\in\is$ and in particular $v(y\cdot c^\infty)=y\cdot c^\infty$ that is $v(x)=x.$
Moreover, $v$ has slope $2^{-n}$ at $x$ where $n$ is the number of letters in $c.$
This proves that $R$ is contained in $S$.

Conversely, consider $x\in S$ and chose $v\in V$ satisfying $v(x)=x, v'(x)\neq 1$.
Up to considering $v^{-1}$ rather than $v$ we can assume that $v'(x)<1.$
This implies that there exists a sdi $I$ containing $x$ that is adapted to $x$ and satisfying that $v(I)\subset I.$
If $m_I$ is the word associated to $I$, then we have that $m_I$ is a proper prefix of $m_{v(I)}$ that is $m_{v(I)}=m_I\cdot c$ for a certain nontrivial word $c.$
Since $x\in I$ there exists $z\in \is$ satisfying $x=m_I\cdot z$ and by definition of the action of $v$ on $I$ we have that 
$$m_I\cdot z = v(m_I\cdot z) = m_{v(I)}\cdot z = m_I\cdot c\cdot z.$$
This implies that $x= m_I\cdot c^\infty$ and thus $x\in R$.

Proof of (2).

Consider $\varphi\in\NCV$ and $x\in R$.
There exists $v\in V$ satisfying $v(x)=x$ and $v'(x)\neq 1.$
Note that $w:=\varphi v\varphi^{-1}$ is in $V$ by definition and that $w(\varphi(x))= \varphi(x)$.
If $w'(\varphi(x))=1$, then there exists a sdi $J$ containing $\varphi(x)$ on which $w$ acts like the identity. 
This would imply that $v$ fixes points arbitrary closed to $x$ but different from $x$ implying that $v$ acts like the identity on a neighbourhood of $x$ and thus on a sdi contradicting that $v'(x)\neq 1$.
Therefore, $w'(\varphi(x))\neq 1$ and thus $\varphi(x)\in R$.
We have proved that $\varphi(R)\subset R$ and considering $\varphi^{-1}$ we deduce that $\varphi(R)=R$.

Proof of (3).

Consider $\varphi\in\NCV$, $x\in\fC$ and $v\in V$ satisfying $v(x)=x.$
We follow a similar proof to the one given in \cite[Proposition 1.5]{Brothier20-1} and obtain that 
$$w\mapsto \log_2(w'(x))$$
provides an injective group morphism from $V_x/V_x'$ to $\Z$ where $V_x=\{ w\in V:\ w(x)=x\}$ and $V_x'$ is its derived group. 
Therefore, there exists a unique natural number $k_x$ (possibly equal to zero) satisfying that $w\mapsto \log_2(w'(x))$ is an isomorphism $\ell_x$ from $V_x/V_x'$ onto $k_x\Z.$
Note that by the proof of above $k_x=0$ if and only if $x$ is not a rational point, i.e.~$x\notin R.$
Similarly, we have an isomorphism $\ell_{\varphi(x)}$ from $V_{\varphi(x)}/V_{\varphi(x)}'$ onto $k_{\varphi(x)}\Z$.
Note that $\ad_{\varphi}:V\to V, v\mapsto \varphi v\varphi^{-1}$ sends $V_x$ onto $V_{\varphi(x)}$ and $V_x'$ onto $V_{\varphi(x)}'$ and thus factorising into an isomorphism $\ov\ad_{\varphi}: V_x/V_x'\to V_{\varphi(x)}/ V_{\varphi(x)}'.$
We obtain that $$f:=\ell_{\varphi(x)}\cdot \ov \ad_\varphi\cdot \ell_x^{-1}: k_{x}\Z\to k_{\varphi(x)} \Z$$
is a group isomorphism.
Note that if $v(x)=x$, then $v'(x)<1$ if and only if there exists a sdi $I$ containing $x$ satisfying that $\lim_{n\to\infty}v^n(y)=x$ for all $y\in I$.
This characterisation shows that if $v(x)=x$ and $v'(x)<1$, then $\ad_\varphi(v)'(\varphi(x))<1.$
This implies that $f(k_{x} n) = k_{\varphi(x)} n$ for all $n\in \Z$.
Therefore, if $v\in V_x$ and $\bar v\in V_x/V_x'$ is its class, then 
$$f\circ \ell_x(\bar v) = f(\log_2(v'(x))) = \frac{k_{\varphi(x)}}{k_x}\cdot \log_2(v'(x)).$$
Since 
$$f\circ \ell_x(\bar v) = \ell_{\varphi(x)}(\bar{\ad_\varphi(v)}) = \log_2(\ad_\varphi(v)'(\varphi(x)))$$
we deduce that
$$k_{\varphi(x)} \cdot \log_2(v'(x)) = k_x \cdot \log_2(\ad_\varphi(v)'(\varphi(x)), \ \forall v\in V_x, x\in\fC.$$

We conclude the proof by computing $k_x$ and $k_{\varphi(x)}$ in terms of certain prime words $m_x$ and $m_{\varphi(x)}$ depending on $x$ and $\varphi(x)$.
Assume that $v\in V$ satisfying $vx=x$ and $v'(x)<1.$
There exists a sdi $I$ containing $x$ for which $v$ is adapted that is $v(m_I\cdot z) = m_{v(I)}\cdot z$ for all $z\in\is.$
Moreover, since $v'(x)<1$ we have that $m_{v(I)}= m_I\cdot a$ for a nontrivial word $a$.
Since $x\in R$ there exists a prime word $m_x$ such that $x$ is in the tail equivalence class of $m_x$.
Therefore, there exists $z\in \fs$ satisfying $x=m_I\cdot z\cdot m_x^\infty$ since $x\in I$.
We obtain that
$$v(x)=m_I\cdot a\cdot z\cdot m_x^\infty = x = m_I\cdot z\cdot m_x^\infty.$$
Hence, $a\cdot z\cdot m_x^\infty = z\cdot m_x^\infty$. 
We deduce that $a\cdot z = z\cdot m_x^p$ for a certain natural number $p\geq 1.$
This implies that $a=c^p$ where $c$ is a cyclic permutation of the prime word $m_x.$
In particular, $|a| = p\cdot |m_x|$ where $|a|$ is the number of letters of the word $a$.
This implies that $\log_2(v'(x)) = |a| = p \cdot |m_x|\in |m_x|\Z.$
Conversely, consider the sdi 
$$J=\{ y\cdot m_x\cdot u:\ u\in\is\}$$ containing $x=y\cdot m_x^\infty$.
Note that $y\cdot m_x$ is nontrivial since $m_x$ is nontrivial and thus $J$ is a proper subset of the Cantor space.
Therefore, we can construct $w\in V$ satisfying that $w(y\cdot m_x\cdot u) = y\cdot m_x\cdot m_x\cdot u$ for all $u\in\is.$
Note that $w(x)=x$ and $\log_2(w'(x)) = |m_x|.$
We obtain that the range of the morphism $w\in V_x\to \log_2(w'(x))\in\Z$ is equal to $|m_x|\Z$.
Therefore, $k_x=|m_x|$.
Similarly, since $\varphi(R)=R$ we have that $\varphi(x)\in R$ and thus there exists a prime word $m_{\varphi(x)}$ satisfying that $\varphi(x)$ belongs to the tail equivalence class of $m_{\varphi(x)}$ and that the constant $k_{\varphi(x)}$ is equal to $|m_{\varphi(x)}|$.
We obtain that 
 $$|m_{\varphi(x)}|\cdot \log_2(v'(x)) = |m_x|\cdot \log_2(\ad_\varphi(v)'(\varphi(x))), \ \forall v\in V_x, x\in R.$$
This finishes the proof of the proposition.
\end{proof}

\begin{remark}
Note that two prime words may define the same tail equivalence class, e.g.~$01$ and $10$.
In fact two prime words $c,d$ define the same tail equivalence class if and only if $d$ is a cyclic permutation of $c$, i.e.~if $d=d_1\cdots d_n$ with $d_i\in\{0,1\}, 1\leq i\leq n$, then there exists $0\leq k\leq n-1$ such that $c = d_{k+1} \cdots d_n\cdot d_1\cdots d_k.$

From the proof of the previous proposition it is easy to deduce a more precise characterisation of pairs $(v,x)\in V\times\fC$ satisfying $v(x)=x$ and $v'(x)\neq 1.$
Assume $v\in V$ and $x\in\fC$ such that $v(x)=x$ and $v'(x)<1$. 
Then, there exists a prime word $c$, a word $a$ and a natural number $n\geq 1$ satisfying:
\begin{itemize}
\item $x=a\cdot c^\infty$;
\item if $I_a$ is the sdi $\{a\cdot z:\ z\in\is\}$, then $I_a$ is adapted to $x$ and $v(a\cdot z) = a\cdot c^n\cdot z$ for all $z\in \is.$
\end{itemize}
Note that $c$ corresponds to a cyclic permutation of the word $m_x$ of the proposition.

Instead of considering derivations one can work with germs of functions which is an approach adopted by Brin \cite[Section 3]{Brin96-chameleon}, see also \cite{Bleak-Lanoue10}.
We can recover some of the proof of the last proposition from this analysis by considering the groupoid of germs $\cG(V)$ of $V$.
The groupoid $\cG(V)$ has for objects $\fC$ and morphisms from $x$ to $y$ the equivalence classes of pairs $(U,f)$ where $U\subset \fC$ is a neighbourhood of $x$ and $f:U\to \fC$ is the restriction of an element of $V$ sending $x$ to $y$ under the equivalence relation $\sim$ defined as $(U,f)\sim (U',f')$ if $U,U'$ are neighbourhoods of $x$ and $f,f'$ both restrict to the same function on $U\cap U'$. The automorphism group of the object $x\in\fC$ inside the groupoid $\cG(V)$ is morally the quotient group $V_x/V_x'$.
\end{remark}

\subsection{Discrete loop groups and description of fraction groups}\label{sec:loop}
Let $\cF$ be the category of binary forests with set of trees $\fT$. Write $I$ for the trivial tree and $Y$ for the tree with two leaves. 
Let $\Gr$ be the monoidal category of groups with monoidal structure given by direct sums.
A functor $\Phi:\cF\to \Gr$ provides a limit group $K$, a Jones action $V\act K$ and thus a semidirect product $K\rtimes V$.
We are interested in classifying the class of such semidirect products $K\rtimes V$.
Recall that $K\rtimes V$ admits a nice description as a fraction group and thus we may refer to $K\rtimes V$ as a fraction group.
We limit our study to covariant monoidal functors $\Phi:\cF\to\Gr$ that are in one-to-one correspondence with triples $(\Ga,\al_0,\al_1)$ where $\Ga$ is a group and $\al_i:\Ga\to\Ga,i=0,1$ endomorphisms.
The correspondence is given by $\Phi\mapsto (\Phi(1), \Phi(Y))$ where $1$ is the object $1$ of $\cF$ and $Y$ the tree with two leaves.
In a previous article, we studied the class of fraction groups arising from triples $(\Ga,\al_0,\al_1)$ where $\al_1=\varep_\Ga$ is the trivial endomorphism, i.e.~$\al_1(g)=e_\Ga$ for all $g\in \Ga$ where $e_\Ga$ is the neutral element of $\Ga.$
In this article we are interested in the case where both $\al_0,\al_1$ are nontrivial and shortly we will restrict to the case where $\al_0,\al_1$ are automorphisms.

\subsubsection{General description of the limit group}
Fix a triple $(\Ga,\al_0,\al_1)$ with $\Ga$ a group and $\al_0,\al_1$ some endomorphisms of $\Ga.$
Let $\Phi:\cF\to\Gr$ be the associated monoidal covariant functor satisfying $\Phi(1)=\Ga$ and $\Phi(Y)=(\al_0,\al_1)\in \Hom(\Ga,\Ga\oplus\Ga).$
From this data we construct the directed system of groups 
$$(\Ga_t , \ \iota_{s,t}:\ s,t\in\fT, s\geq t)$$
defined as follows.
We define $$\Ga_t=\{(g,t):\ g=(g_\ell)_ {\ell\in\Leaf(t)} \in \Ga^{\Leaf(t)} \}$$ 
a copy of the group of maps from the leaves $\Leaf(t)$ of the tree $t$ to the group $\Ga$ for $t\in\fT$.
We may drop the $t$ and write $g$, $g_t$ or $(g_\ell)_{\ell\in\Leaf(t)}$ for $(g,t)$ if the context is clear.
We define the maps such as:
$$\iota_{ft,t}:\Ga_t\to \Ga_{ft}, \ (g,t)\mapsto (\Phi(f)(g),f\circ t),\ t\in\fT, f\in\Hom(\cF), g\in \Ga^{\Leaf(t)}$$
where $f$ is a forest composable with $t$, i.e.~the number of roots of $f$ is equal to the number of leaves of $t$.
We obtain a limit group $K:=\varinjlim_{t\in\fT}\Ga_t$ and a Jones action $\pi:V\act K.$

We now describe the elements of $K$ using equivalence classes of functions.
To do that it is convenient to define $\al_m$ for a finite word $m=m_1\cdots m_k\in\fs$ which is 
$$\al_m:= \al_{m_k}\circ\cdots \al_{m_1}.$$
Note the inversion of order between the letters of the word $m$ and the order of composition of the endomorphisms.
Elements of $\Ga_t$ for $t\in\fT$ can be identified with certain maps from $\fC$ to $\Ga.$
Indeed, consider a tree $t\in\fT$ with associated sdp $(I_t^\ell:\ell\in\Leaf(t))$.
Define $$\kappa_t:\Ga_t\to \{ \fC\to \Ga \} , \ \kappa_t(g)(x) = g_\ell \text{ if } x\in I_t^\ell.$$
Observe that $\kappa_t$ is a group isomorphism form $\Ga_t$ onto the maps $f:\fC\to \Ga$ that are constant on each $I_t^\ell, \ell\in\Leaf(t).$
Consider another tree $s$ larger than $t$.
There exists a forest $f$ satisfying $s=f\circ t.$
In the composition $f\circ t$ we attach to each leaf $\ell$ of $t$ a tree $f_\ell$ so that $f=(f_\ell)_{\ell\in\Leaf(t)}$ is the horizontal concatenation of the trees $f_\ell, \ell\in \Leaf(t).$
Given $\ell \in \Leaf(t)$ and $p$, a leaf of the tree $f_\ell$, we write $m_p=e_1\cdots e_k$ the finite sequence of $0,1$ (possibly empty) corresponding to the (geodesic) path from the root of $f_\ell$ to the leaf $p$ and where $e_i$ is $0$ (resp. $1$) if the $i$th edge starting from the root is a left edge (resp. a right edge).
If $I_p$ is the sdi corresponding to the leaf $p\in\Leaf(f_\ell)$, then observe that
$$\kappa_{f\circ t}(\Phi(f)(g_t))(x)=\al_{m_p}(g_\ell) = \al_{e_k}\circ\cdots\circ\al_{e_1}(g_\ell), \ \forall x\in I_p,\ g_t=(g_\ell)_{\ell\in\Leaf(t)}\in\Ga_t.$$
Hence, if $g=g_t\in\Ga_t$, then $\kappa_{f\circ t}(\Phi(f)(g_t))$ has its support contained in the support of $\kappa_t(g_t)$ and takes values of the form $\al_m(g_\ell)$ with $\ell\in\Leaf(t)$ and $m$ some finite words in $0,1.$
\tbb{Given $g\in K$ we can find a large enough tree $t$ such that $g$ is the equivalence class of some $(g_t,t)\in\Ga_t.$}
If $(g_{f\circ t},f\circ t)\in\Ga_{f\circ t}$ is another representative of the class of $g$, then by definition of the direct system we have $\Phi(f)(g_t)=g_{f\circ t}$.
The element $g$ is then described by a family of continuous maps $(\kappa_s(g_s), s\geq t)$ from $\fC$ to $\Ga$ all taking finitely many values.

\subsubsection{Support}
The description of elements of $K$ as equivalence classes of maps suggest a notion of support.
Consider $g\in K$.
Choose a tree $t$ large enough such that $g$ admits a representative $(g_t,t)\in \Ga_t$.
Denote by $\supp(\kappa_t(g_t))$ the support of the map $\kappa_t(g_t):\fC\to\Ga.$
Choose a tree $s$ such that $s\geq t.$
There exists a forest $f$ composable with $t$ such that $s=f\circ t$ and $(g_s,s):=(\Phi(f)(g_t),f\circ t)$ is another representative of $g$.
By definition of $\Phi(f)$ and the maps $\kappa_t,\kappa_s$ we obtain that $\supp(\kappa_t(g_t))\supset \supp(\kappa_s(g_s))$.
We define the support of $g$ as the intersection:
$$\supp(g):=\bigcap_{s\in\fT:\ s\geq t} \supp(\kappa_{s}(g_s)).$$
Note that this intersection does not depend on the choice of the representative $(g_t,t)$ of $g$.
In particular, if $\al_0,\al_1$ are injective (this is the main case of our study), then $\supp(g)=\supp(\kappa_t(g_t))$ for any choice of representative $(g_t,t)$ of $g$.

\subsubsection{General description of the Jones action}

We now briefly describe the Jones action 
$$\pi:V\to \Aut(K), v\mapsto \pi_v.$$
We will later provide a more practical description of this action in the case where both $\al_0$ and $\al_1$ are automorphisms.
Although, it is useful to present the general case since we will be using it in the proof of the next proposition and in Section \ref{sec:isom}.
Consider $g\in K$ and $v\in V$.
By definition of the group $V$ there exists two ordered sdp $\ov I:= (I_k:\ 1\leq k\leq n), \ov J:= (J_k:\ 1\leq k\leq n)$ and a permutation $\sigma$ satisfying that:
$$v(m_{I_k}\cdot x)=m_{J_{\sigma(k) } } \cdot x$$ for all $1\leq k\leq n$ and $x\in \is.$
Hence, $v$ sends the $k$th sdi of $\ov I$ to the the $\sigma(k)$th sdi of $\ov J$ for all $1\leq k\leq n.$
Let $t,s$ be the trees associated to the sdp $\ov I, \ov J$ respectively.
Since $K$ is the direct limit of the directed system of groups $(\Ga_r, r\in\fT)$ one can find a tree $r$ and a representative $(g_{r},r)\in \Ga_{r}$ of $g$. 
Here, $g_{r}:\Leaf(r)\to \Ga$ is a map from the set of leaves of $r$ to the group $\Ga.$
If $r\neq t$, then we can always find some forests $p,q$ satisfying $pt=qr.$
We have that $(\Phi(q)(g_r), qr)$ is a representative of $g$ and $v$ is described by the pair of trees $(pt, ps).$
Therefore, we can always assume that $r=t$ up to choosing large enough trees.
The element $v$ sends a sdi of $\ov I$ to a sdi of $\ov J$ which defines a bijection $b:\Leaf(t)\to \Leaf(s).$
Consider the map $g_r\circ b^{-1}: \Leaf(s)\to \Ga$ which is an element of $\Ga_s.$
The Jones action is described by the following formula:
$$\pi_v(g):= [ (g_r\circ b^{-1},s)]$$
where $[(g_r\circ b^{-1},s)]$ is the class of $(g_r\circ b^{-1},s)$ inside the directed limit group $K$.
The Jones action is somehow induced by the spatial action $V\act \fC$.
Note that the pair of endomorphisms $(\al_0,\al_1)$ only appeared when we considered the representative $(\Phi(q)(g_r), qr)$ of $g$. 
They will become more apparent when we will describe the Jones action in the special case where  $\al_0,\al_1$ are automorphisms.

We now describe a key class of examples of Jones action and fraction groups.
Consider a group $\Ga$ and the pair of morphisms $(\al_0,\al_1)=(\id_\Ga,\id_\Ga)$.
This defines a covariant monoidal functor $\Phi:\cF\to\Gr$ satisfying that 
$$\Phi(1)=\Ga \text{ and } \Phi(Y):\Ga\to\Ga\oplus\Ga, g\mapsto (g,g).$$
Let $K\rtimes V$ be the associated fraction group.
If $g\in K$, then there exists a large enough tree $t\in\fT$ and an element $g_t=(g_\ell)_{\ell\in\Leaf(t)}$ in $\Ga_t$ that is a representative of $g$.
The tree $t$ provides a sdp $(I_t^\ell:\ell\in\Leaf(t))$ of $\fC$. 
We associate to $g_t$ the map 
$$\kappa_t(g_t):\fC\to\Ga, x\mapsto g_\ell \text{ if } x\in I_t^\ell.$$
Observe that $\kappa_t(g_t)$ does not depend on the choice of $t$ since $\al_0=\al_1=\id_\Ga$, i.e.~if $g_s\in\Ga_s$ with $s\in\fT$ is another representative of $g$, then $\kappa_s(g_s)=\kappa_t(g_t).$
We find that $K$ is isomorphic to the group of all maps $f:\fC\to\Ga$ satisfying that there exists a sdp $(I_1,\cdots,I_n)$ so that $f$ is constant on each $I_k$, $1\leq k\leq n.$
Under this identification, the Jones action $\pi:V\act K$ is then the classical spatial action induced by $V\act \fC$:
$$\pi_v(f)(x):=f(v^{-1}x), v\in V, f\in K, x\in\fC.$$
We will qualify these Jones actions and fractions groups as \textit{untwisted}.

Note that if we were working in the case of the previous article \cite{Brothier20-1}: that is $\Ga$ is a group, $\al_0$ is any endomorphism and $\al_1=\varep_\Ga$ is trivial, then by refining the tree $t$ into $f\circ t$ we obtain that the support of $a_{f\circ t}$ is getting smaller obtaining at the limit a discrete support contained in $\{r_t^\ell:\ell\in\Leaf(t)\}$ where $r_t^\ell$ is the first point of the sdi  $I_t^\ell, \ell\in\Leaf(t)$.
The case $\al_0=\al_1=\id_\Ga$ provides that $K$ corresponds to continuous functions with finite range supported on regions (clopen) rather than on a finite union of points of the space $\fC$.
It is better adapted to build field theories. Moreover, the group $K$ is reminiscent of the loop group of a Lie group (the group of smooth maps from the circle to a fixed Lie group). Recall that loop groups were used to construct conformal field theories by Wassermann \cite{Wassermann98}.

Before ending this general presentation we explain why we can always restrict our analysis to triples $(\Ga,\al_0,\al_1)$ satisfying that $g\mapsto (\al_0(g),\al_1(g))$ is injective. 
The argument is due to Tanushevski \cite[Corollary 3.11]{Tanushevski16}.
We provide a short proof for the convenience of the reader.

\begin{proposition}\label{prop:Tanushevski}
Consider a triple $(\Ga,\al_0,\al_1)$ with $\Ga$ a group and $\al_0,\al_1\in\End(\Ga).$
There exists another triple $(\ov\Ga,\ov\al_0,\ov\al_1)$ such that $\ov\Ga$ is a group, $\ov\al_0,\ov\al_1\in\End(\ov\Ga)$, the morphism $\ov g\in\ov\Ga\mapsto (\ov\al_0(\ov g),\ov\al_1(\ov g))$ is injective and satisfying that the fraction groups associated to $(\Ga,\al_0,\al_1)$ and 
$(\ov\Ga,\ov\al_0,\ov\al_1)$ are isomorphic.
\end{proposition}
\begin{proof}
Fix $(\Ga,\al_0,\al_1)$ as above and consider its associated monoidal functor $\Phi:\cF\to\Gr$ and Jones action $\pi:V\act K$.
For any tree $t\in\fT$ consider the kernel of $\Phi(t)$ which is a normal subgroup of $\Ga.$
Let $N$ be the union of the kernels: $N:=\cup_{t\in\fT} \ker(\Phi(t)).$
Note that if $t\leq s$, then $\ker(\Phi(t))\subset \ker(\Phi(s))$ for $t,s\in\fT$. This easily implies that $N$ is a normal subgroup of $\Ga$ and moreover is closed under taking $\al_0$ and $\al_1.$
This implies that $\al_0,\al_1$ factorise into endomorphisms $\ov\al_0,\ov\al_1$ of $\ov\Ga:=\Ga/N$.
Let $\ov \pi:V\act \ov K$ be the Jones action associated to the triple $(\ov\Ga,\ov\al_0,\ov\al_1)$ where $\ov K$ is the direct limit of the directed system of groups $( \ov\Ga_t:=(\Ga/N)_t:\ t\in\fT)$.
The quotient map $q:\Ga\to\ov\Ga$ induces a family of quotient maps $\Ga_t\to\ov\Ga_t$ providing a group morphism $K\to\ov K$.
It is not hard to prove that this later morphism is a $V$-equivariant isomorphism for the two Jones actions.
\end{proof}

\begin{remark}This last proposition tells us that, if we consider fraction groups obtained from \textit{monoidal} functors $\Phi:\cF\to\Gr$, then we can always consider that $\Phi(Y):\Ga\to\Ga\oplus\Ga$ is injective (which is equivalent to saying in the monoidal context that $\Phi(f)$ is injective for all forests $f$).
Here, the fact that $\Phi$ is monoidal is crucial. It is not excluded that interesting examples of fraction groups will only appear from non-monoidal functors $\Phi:\cF\to\Gr$ satisfying that $\Phi(f)$ is non-injective for certain forests $f$.
\end{remark}

\subsubsection{Restriction to pairs of automorphisms}

In order to obtain an easy description of the fraction group $K\rtimes V$ in terms of maps and spatial action of $V$ we restrict the choice of triples $(\Ga,\al_0,\al_1)$ to the one where both $\al_0$ and $\al_1$ are automorphisms. 
{\bf Hence, from now on $\al_0,\al_1$ are automorphisms.}
We are going to describe $G$ in a similar way as done above as a certain subgroup of a twisted \textit{unrestricted} permutational wreath product. 
We start by introducing some terminology and notations:

\begin{definition}[Discrete loop group]
Consider a group $\Ga$ and a map $f:\fC\to \Ga.$
We say that $f$ is locally constant if there exists a sdp $(I_1,\cdots,I_n)$ of $\fC$ such that $f$ restricted to $I_j$ is constant for any $1\leq j\leq n.$
The collection of all locally constant maps from $\fC$ to $\Ga$ forms a group $L\Ga$ with pointwise multiplication that we call the \textit{discrete loop group} of $\Ga.$
\end{definition}

\begin{remark}
Consider a group $\Ga$ equipped with the discrete topology.
We have that the discrete loop group of $\Ga$ is equal to the group $C(\fC,\Ga)$ of continuous functions from $\fC$ to $\Ga$.
\end{remark}
\begin{proof}Indeed, consider $a\in L\Ga.$ If $g\in \Ga$, then its inverse image by $a$ is either empty or equal to a finite union of sdi which is open. Therefore, $a:\fC\to\Ga$ is continuous implying that $L\Ga\subset C(\fC,\Ga).$
Conversely, consider $f\in C(\fC,\Ga)$. Let $x\in \fC$ and $g:=f(x)$. 
Since $f$ is continuous we have that $U_x:=f^{-1}(\{g\})$ is an open neighbourhood of $x$ and thus contains a sdi $I_x$ so that $x\in I_x.$
We obtain an open covering $(I_x:\ x\in \fC)$ and by compactness there exists a finite subcovering $(I_x:\ x\in X)$ where $X\subset \fC$ is finite. 
For each $x\in X$ define $J_x:= I_x\setminus (\cup_{y\in X, y\neq x} I_y)$ and note that $(J_x:\ x\in X)$ is a sdp of $\fC$ satisfying that $f$ is constant on each $J_x, x\in X.$ 
Therefore, $f\in L\Ga$ and $C(\fC,\Ga)\subset L\Ga.$
\end{proof}

By definition $L\Ga$ is a subgroup of the product $\prod_\fC \Ga.$
Since $\Qt\subset \fC$ is a dense subset we have that all sdi have a nontrivial intersection with $\Qt.$
This implies that the restriction map $r:\prod_\fC\Ga\to \prod_\Qt\Ga$ is injective when restricted to $L\Ga.$
It will be sometimes more convenient to consider $r(L\Ga)\subset \prod_\Qt\Ga$ rather than $L\Ga\subset \prod_\fC\Ga.$
We will often identify $L\Ga$ with $r(L\Ga).$

\begin{notation}\label{nota:tau}
Recall that if $m=l_1\cdots l_n\in\fs$ is a finite word in $0,1$ of length $n$, then we write $\al_m:=\al_{l_n}\circ\cdots\circ\al_{l_1}$ (note the inversion of the order) the composed automorphism.

Consider a sdi $I$ and its associated word $m_I$, i.e.~$I=\{m_I\cdot z:\ z\in\{0,1\}^\N\}.$
We write $\al_I:= \al_{m_I}.$
For example, if $I$ corresponds to $[0,1/2)$ and $J=[1/4,1/2)$, then $\al_I=\al_0$ and $\al_J=\al_1\circ\al_0.$

If $t$ is a tree, then any of its leaves $\ell$ defines a sdi $I^\ell_t.$
We write $\al_\ell$ for the automorphism $\al_{I^\ell_t}$ if the context is clear.

Consider now $v\in V$. There exists two (not necessarily ordered) sdp $(I_k:\ 1\leq k\leq n)$ and $(J_k:\ 1\leq k\leq n)$ such that $v$ is adapted to the first satisfying $v(I_k)=J_k$ for all $1\leq k\leq n.$
If $x\in I_k$ for some $1\leq k\leq n$, we write 
$$\tau_{v,x}:= \al_{v(I_k)}^{-1}\al_{I_k}=\al_{J_k}^{-1}\al_{I_k}.$$
Observe that the definition of $\tau_{v,x}$ does not depend on the choice of the sdp and defines a map
$$\tau_{v,\cdot}:\fC\to \Aut(\Ga), x\mapsto \tau_{v,x}$$
that is locally constant.
\end{notation}

We are now ready to provide an explicit description of the Jones action.

\begin{proposition}\label{prop:twistedLoop}
Let $(\Ga,\al_0,\al_1)$ be a triple such that $\Ga$ is a group and $\al_0,\al_1\in\Aut(\Ga).$
Consider the associated directed system of groups $(\Ga_t: t\in\fT)$, Jones action $V\act K$ and fraction group $K\rtimes V$.
For any tree $t\in\fT$ we have a morphism of groups
$$\kappa_t:\Ga_t\to L\Ga , \ g=(g_\ell)_{\ell\in\Leaf(t)}\mapsto \kappa_t(g):x\in\fC\mapsto\al_\ell^{-1}(g_\ell) \text{ if } x\in I_t^\ell.$$
This defines a directed system of morphisms whose limit $\kappa:K\to L\Ga$ is an isomorphism.

The Jones action $V\act K$ is conjugated by $\kappa$ to the following action on $L\Ga$:
$$\pi:V\to \Aut(L\Ga), \pi_v(a)(vx) = \tau_{v,x}(a(x)),\ v\in V, a\in L\Ga, x\in \fC.$$
We identify the Jones action $V\act K$ with $\pi:V\act L\Ga$ and fraction group $K\rtimes V$ with the semidirect product $L\Ga\rtimes V$ obtained from $\pi$.
\end{proposition}

\begin{proof}
Consider $(\Ga,\al_0,\al_1)$ as in the proposition.
It is clear that the formula $\kappa_t$ for $t\in\fT$ defines a group morphism from $\Ga_t$ to $\prod_{\fC}\Ga$. 
If $g\in\Ga_t$, then $\kappa_t(g)$ is constant on each sdi appearing in the sdp associated to $t$.
Therefore, $\kappa_t(g)$ is locally constant and thus belongs to $L\Ga.$
If $\kappa_t(g)=e_{L\Ga}$ (where $e_{L\Ga}$ is the neutral element of $L\Ga$), then $\al_\ell^{-1}(g_\ell)=e_\Ga$ for all $\ell\in\Leaf(t)$ implying that $g_\ell=e_\Ga$ for all leaves $\ell\in \Leaf(t)$ and thus $g=e_{\Ga_t}.$
Hence, $(\kappa_t: t\in\fT)$ is a family of injective group morphisms valued in $L\Ga.$

We now show that the family $(\kappa_t:\ t\in\fT)$ defines a morphism from $K:=\varinjlim_{t\in\fT}\Ga_t$ to $L\Ga.$
By definition of the directed system of the groups $\Ga_t, t\in\fT$ it is sufficient to check that for all $t\in\fT$, all forests $f\in\Hom(\cF)$ composable with $t$ and all $g=(g_\ell)_{\ell\in\Leaf(t)}\in\Ga_t$ we have that 
$$\kappa_t(g) = \kappa_{f\circ t}(\Phi(f)(g))$$
where $\Phi:\cF\to\Gr$ is the functor induced by the triple $(\Ga,\al_0,\al_1).$
If $t$ has $n$ leaves, then the forest $f$ has $n$ roots and thus is the horizontal concatenation of $n$ trees. 
Denote these trees by $f_1,\cdots,f_n$ when ordered from left to right.
If $\ell$ is the $k$th leaf of $t$, then it is lined up with $f_k$ in the composition $f\circ t.$ 
Write in that case $f_\ell:=f_k$ so that $f_\ell$ is attached to the leaf $\ell$ of $t$ inside the composition $f\circ t.$
Choose one leaf $u$ of the tree $f_\ell$ and write $p_u=b_1\cdots b_k\in\fs$ the path from the root of $f_\ell$ to the leaf $u$.
Observe that $\Phi(f)(g)=h$ with $h:\Leaf(f)\to \Ga$ satisfying that $h_u=\al_{b_k}\circ \cdots \circ \al_{b_1}(g_\ell)$.
If $I^u_{f\circ t}$ is the sdi corresponding the leaf $u$ of $f\circ t$, then we obtain that 
$$\kappa_{f\circ t}(\Phi(f)(g))(x) = \al_u^{-1} (h_u) = \al_u^{-1}( \al_{b_k}\circ \cdots \circ \al_{b_1}(g_\ell)), \ \forall x\in I^u_{f\circ t}.$$
By definition, $\al_u = \al_{b_k}\circ\cdots\circ\al_{b_1}\circ \al_\ell$ implying that $\kappa_{f\circ t}(\Phi(f)(g))(x) = \al_\ell^{-1}(g_\ell)$ for all $x\in I^u_{f\circ t}$.
Note that $I^u_{f\circ t}$ is a subinterval of $I^\ell_t$ implying that $\kappa_t(g)$ and $\kappa_{f\circ t}(\Phi(f)(g))$ coincide on $I^u_{f\circ t}.$
We have proved that 
$$\kappa_{f\circ t}\circ \Phi(f) = \kappa_t \text{ for all } t\in\fT, f\in\Hom(\cF).$$
Therefore, we can define the direct limit of the family of morphism $(\kappa_t:\ t\in\fT)$ giving a group morphism $$\kappa:K\to L\Ga.$$
Moreover, $\kappa$ is injective since each of $\kappa_t,t\in\fT,$ and $\Phi(f), f\in\Hom(\cF)$ are.

Consider $a\in L\Ga.$
There exits a sdp $(I_1,\cdots,I_n)$ such that $a$ is constant on each $I_k, 1\leq k\leq n.$
Let $t$ be the tree associated to this sdp.
Note that the range of $\kappa_t$ is equal to all $b\in L\Ga$ that are constant on each $I_k,1\leq k\leq n$.
In particular, we obtain that $a$ belongs to the range of $\kappa_t$ and thus $\kappa$ is surjective.
We have proved that $\kappa$ is an isomorphism from $K$ to $L\Ga.$

Consider the Jones action $\pi:V\act K$ and $g\in K, v\in V$.
There exists two sdp $(I_1,\cdots,I_n)$ and $(J_1,\cdots,J_n)$ such that $v$ is adapted to the first and sends $I_k$ onto $J_k$ for all $1\leq k\leq n.$
There exists a tree $t\in\fT$ such that $g$ is the equivalence class of an element of $\Ga_t$.
Up to refining $t$ and the sdp of above we can assume that the sdp of $t$ is equal to $(I_1,\cdots,I_n).$
Write $\ell_k$ the leaf of $t$ associated to the sdi $I_k$ for  $1\leq k\leq n$.
Similarly consider the tree $s\in\fT$ associated to the sdp $(J_1,\cdots,J_n)$ and denote by $r_k$ the leaf of $s$ associated to $J_k$ for $1\leq k\leq n.$
Observe that $$\kappa(g)(x) = \al_{I_k}^{-1}(g_{\ell_k}), \ \forall x\in I_k, 1\leq k\leq n,$$
where $(g_{\ell_k}:\ 1\leq k\leq n)$ is the representative of $g$ in $\Ga_t.$
By definition of the Jones action we have that the automorphism $\pi_v(g)$ has for representative $h=(h_{r_k}:\ 1\leq k\leq n)\in \Ga_s$ where $h_{r_k} = g_{\ell_k}$ for all $1\leq k\leq n.$
Therefore, $$\kappa[\pi_v(g)](y) = \al_{J_k}^{-1}(h_{r_k}) = \al_{J_k}^{-1}(g_{\ell_k}), \ \forall y \in J_k, 1\leq k\leq n.$$
In particular, if we fix $1\leq k\leq n$ and $x\in I_k$ we have that $vx\in J_k=v(I_k)$ and thus:
$$\kappa[\pi_v(g)](vx) = \al_{J_k}^{-1}(g_{\ell_k}) = \al_{J_k}^{-1} \circ \al_{I_k} (\kappa(g)(x)) = \al_{v(I_k)}^{-1}\circ \al_{I_k}(\kappa(g)(x)).$$
This implies that $\kappa[\pi_v(g)](vx) = \tau_{v,x}(\kappa(g)(x))$ for all $g\in K, v\in V, x\in\fC.$
This finishes the proof of the proposition.
\end{proof}

Note that the Jones action $\pi:V\act L\Ga$ extends in the obvious way into an action on $\prod_{\fC}\Ga$ giving a twisted permutational wreath product $\prod_\fC\Ga\rtimes V$.

\begin{remark}
Consider a nontrivial group $\Ga$ and the map $$h:\oplus_\Qt\Ga\to \prod_\Qt\Ga, \ h(a)(x) = \prod_{y\leq x} a(y)$$ 
where $\prod_{y\leq x} a(y):= a(x_1)\cdots a(x_n)$ when $0\leq x_1< \cdots < x_n\leq x$ and $\supp(a)\cap [0,y]\subset\{x_1,\cdots,x_n\}.$
This is a well-defined map since all the products considered are finite.
Moreover, one can show that $h$ is a group isomorphism from $\oplus_\Qt\Ga$ onto $L\Ga$. 
However, this morphism is not compatible with the action $V\act\Qt.$
We will see in Section \ref{sec:isom} that $\oplus_{\Q_2}\Ga\rtimes V$ and $L\Ga\rtimes V$ are never isomorphic so the latter isomorphism is not $V$-equivariant whatever are the Jones actions.
\end{remark}

\subsection{Previous results on fraction groups}

We will refer and use the following results proved in \cite{Brothier20-1}.

\begin{proposition}\cite[Proposition 2.4]{Brothier20-1}\label{prop:chara}
Let $\Phi:\cF\to\Gr$ be a covariant monoidal functor with associated Jones action $V\act K$ and semidirect product $G:=K\rtimes V$.

The subgroup $K$ is the unique maximal normal subgroup satisfying the following \textit{decomposability property:}
\begin{itemize}
\item $K$ can be decomposed as a direct sum of two groups $K=A\oplus B$;
\item $K = \cN_G(A) = \cN_G(B)$, where $\cN_G(A):=\{ g\in G:\ gAg^{-1}=A\}.$
\end{itemize} 

In particular, if $G:=K\rtimes V$ and $\ti G:=\ti K\rtimes V$ are two fraction groups constructed from covariant monoidal functors from $\cF$ to $\Gr$, then any isomorphism $\theta:G\to\ti G$ satisfies $\theta(K)=\ti K.$
\end{proposition}

The following result is a direct consequence of Propositions 3.3 and 3.5 of the article \cite{Brothier20-1}. 
The first cited proposition shows that we can always assume that $\al$ is an automorphism. The second cited proposition gives the explicit wreath product description of the fraction group.

\begin{proposition}\label{prop:WP}
Consider a group $\Ga$ and an endomorphism $\al\in\End(\Ga).$
Let $G=K\rtimes V$ be the fraction group associated to the triple $(\Ga,\al,\varep_\Ga)$ where $\varep_\Ga:g\in\Ga\mapsto e_\Ga$ denotes the trivial endomorphism.

There exists a group $\ti\Ga$ and an automorphism $\ti\al\in\Aut(\ti\Ga)$ satisfying that $G$ is isomorphic via a $V$-equivariant map to the fraction group $\ti G=\ti K\rtimes V$ associated to the triple $(\ti\Ga,\ti\al,\varep_{\ti\Ga}).$
Moreover, $\ti\Ga$ is isomorphic to the restricted twisted permutational wreath product $\oplus_{\Qt}\ti\Ga\rtimes V$ where the action $V\act \oplus_{\Qt}\ti \Ga$ is the following:
$$(v\cdot a)(vx) : = {\ti\al}^{\log_2(v'(x))}( v(x) ) ,\ v\in V, a\in \oplus_\Qt \ti\Ga, x\in\Qt.$$
\end{proposition}

\begin{proposition}\cite[Proposition 4.7]{Brothier20-1}\label{prop:cocycle}
Let $\La$ be an Abelian group and consider the set of cocycles:
$$\Coc=\Coc(V\act \prod_{\Q_2}\La) := \{ c:V\to \prod_{\Q_2}\La:\ c_{vw} = c_v \cdot c_w^v, \ \forall v,w\in V\}.$$
This set is an Abelian group under the pointwise product 
$$(c\cdot d)_v(x):= c_v(x)\cdot d_v(x), c,d\in\Coc, v\in V, x\in\Q_2.$$
For any $\zeta\in\La$ define
$$s(\zeta)_v(x):= \zeta^{\log_2(v'(v^{-1}x))}, v\in V,x\in\Q_2$$ 
which is a cocycle.
For any $c\in\Coc$ there exists a pair $(\zeta,f)\in \La\times \prod_{\Q_2}\La$ satisfying that $c_v = s(\zeta_v)\cdot f (f^v)^{-1}, v\in V$.
Moreover, the pair $(\zeta,f)$ is unique up to multiplying $f$ by a constant map. 
\end{proposition}

\section{Isomorphic fraction groups}

The aim of this section is to provide a partial classification of the class of fraction groups/semidirect products $L\Ga\rtimes V$ constructed from triples $(\Ga,\al_0,\al_1)$ where $\Ga$ is any group and $\al_0,\al_1$ are \textit{automorphisms} of $\Ga.$

\subsection{Obvious isomorphisms between fraction groups}

We start by constructing obvious isomorphisms between fraction groups.

\begin{lemma}
Consider two groups $\Ga,\ti\Ga$, an isomorphism $\beta\in\Isom(\Ga,\ti\Ga)$ and two automorphisms $\al_0,\al_1\in\Aut(\Ga)$. 
Define the automorphisms $\ti\al_i:= \beta\al_i \beta^{-1}$ for $i=0,1$ of the group $\ti\Ga.$
We consider the two fraction groups $L\Ga\rtimes V$ and $L\ti\Ga\rtimes V$ associated to the triples $(\Ga,\al_0,\al_1)$ and $(\ti \Ga, \ti\al_0,\ti\al_1)$ respectively.
The following formula defines an isomorphism of groups:
$$\theta: L\Ga\rtimes V \to L\ti\Ga\rtimes V , \ \theta(av) = \kappa(a) v , \ a\in L\Ga, v\in V
$$
where $\kappa\in \Isom(L\Ga,L\ti\Ga)$ is defined as 
$$\kappa(a)(x) := \beta(a(x)), \ a\in L\Ga, x\in \fC.$$
\end{lemma}

\begin{proof}
Consider $\beta\in \Isom(\Ga,\ti\Ga)$ and define the product of isomorphisms
$$\kappa := \prod_{\fC} \beta:\prod_{\fC}\Ga\to \prod_{\fC}\ti\Ga,\ \kappa(a)(x):=\beta(a(x)), a\in \prod_\fC\Ga, x\in\fC.$$
This is an isomorphism that maps $L\Ga$ onto $L\ti\Ga$ and thus restrict to an isomorphism $\kappa\in\Isom(L\Ga,L\ti\Ga).$
Let us check that $\kappa$ is $V$-equivariant for the Jones actions $\pi:V\act L\Ga$ and $\ti\pi:V\act L\ti\Ga$.
For any $v\in V$ we consider the locally constant maps $\tau_{v,\cdot}:\fC\to \Aut(\Ga)$ and $\ti\tau_{v,\cdot}:\fC\to\Aut(\ti\Ga)$ defined in Notation \ref{nota:tau} that are associated to $(\al_0,\al_1)$ and $(\ti\al_0,\ti\al_1)$ respectively.
Observe that if $p=p_1\cdots p_n$ is a finite word and $\al_p:=\al_{p_n}\circ\cdots\circ\al_{p_1}$ is the associated automorphism, then $\beta\circ \al_p=\ti\al_p\circ\beta$.
In particular, $\beta\circ\tau_{v,x}=\ti\tau_{v,x}\circ\beta$ for all $v\in V, x\in\fC$.
Observe that 
$$\kappa(\pi_v(a)) (vx)  = \beta( \pi_v(a)(vx) ) = \beta( \tau_{v,x}(a(x) ) )  = \ti\tau_{v,x} ( \beta(a(x)) )= \ti\pi_v(\kappa(a))(x),$$
for all $v\in V,x\in\fC, a\in L\Ga.$
Since $\kappa$ is $V$-equivariant it extends into an isomorphism from $L\Ga\rtimes V$ onto $L\ti\Ga\rtimes V$ via the formula
$av\mapsto \kappa(a)v$ for $a\in L\Ga, v\in V$.
\end{proof}

\begin{lemma}\label{lem:isomaut}
Consider a group $\Ga$, some elements $h_i\in\Ga$ and automorphisms $\al_i\in\Aut(\Ga)$ for $i=0,1.$
Put $\ti\al_i:= \ad(h_i)\circ \al_i,i=0,1$ and consider the two fraction groups $G$ and $\ti G$ with associated Jones actions $\pi:V\act L\Ga$ and $\ti\pi:V\act L\Ga$ that are constructed from the triples $(\Ga,\al_0,\al_1)$ and $(\Ga,\ti\al_0,\ti\al_1)$ respectively.
These two fraction groups are isomorphic.
\end{lemma}

\begin{proof}
For any $n\geq 1$ and any finite sequence $q=q_1\cdots q_n\in \{0,1\}^n$ we define 
$$h_q^\al := h_{q_n}\cdot \al_{q_n}(h_{q_{n-1}})\cdot \al_{q_{n-1} q_n  }(h_{q_{n-2} } ) \cdots \al_{q_2 \cdots q_n }(h_{q_1}).$$
Recall, that we adopted the convention: 
$$\al_{p_1\cdots p_k} = \al_{p_k}\circ\cdots\circ \al_{p_1} \text{ for } k\geq 1, p_1,\cdots,p_k\in \{0,1\}.$$
Observe that 
$$\ti\al_q = \ad(h^\al_q) \circ \al_q.$$
Moreover, if $q,m$ are finite words, then $$h^\al_{q\cdot m} = h^\al_m\cdot \al_m(h^\al_q).$$
For any $v\in V$ and $x\in\fC$ consider a sdi $I$ containing $x$ and adapted to $v$.
Put 
$$c_v(vx):= \al_{ v(I)  }^{-1} ( (h_{m_{v(I)} }^\al)^{-1} h_{m_I}^\al),$$
where $m_I\in\fs$ is the word associated to $I$.
This formula does not depend on the choice of $I$ and thus $c_v$ is a well-defined map from $\fC$ to $\Ga.$
Note that if a sdi $I$ is adapted to $v$, then the map $x\mapsto c_v(vx)$ is constant on $I$.
Therefore, $c_v$ is in $L\Ga.$
Moreover, the map $c:V\to L\Ga, v\mapsto c_v$ satisfies the cocycle identity:
$$c_{v}c_w^v=c_{vw} \text{ for all } v,w\in V,$$
where $c_w^v(x):=c_w(v^{-1}x), x\in\fC.$
A direct computation shows that the map
$$\theta: \ti G\to G ,\ \theta(av) = a \cdot c_v \cdot v, \  a\in L\Ga, v\in V$$
defines a group isomorphism with inverse 
$$bw\mapsto b\cdot c_w^{-1}\cdot w, \ b\in L\Ga, w\in V.$$
\end{proof}

\begin{lemma}
Consider a group $\Ga$ and a pair of automorphisms $(\al_0,\al_1)$ of $\Ga.$
Let $G,\ti G$ be the fraction groups associated to the triples $(\Ga,\al_0,\al_1)$ and $(\Ga,\al_1,\al_0)$ respectively. 
These two fraction groups are isomorphic.
\end{lemma}

\begin{proof}
Consider the permutation $\sigma$ of $\{0,1\}$ of order $2$ and defines the map $\varphi:\fC\to\fC$ such that $\varphi(x)(i)= \sigma(x(i))$ for $x\in\{0,1\}^\N$ and $i\in \N.$
This is a homeomorphism of $\fC$ that normalises $V$ inside $\Homeo(\fC).$
Define the map 
$$\theta: G\to \ti G,\ \theta(av):= (a\circ\varphi^{-1})\cdot \ad_\varphi(v), a\in L\Ga, v\in V$$
where $\ad_\varphi(v):=\varphi v\varphi^{-1}.$
One can check that $\theta$ is an isomorphism of groups.
\end{proof}

From these obvious isomorphisms we obtain the following results:
\begin{proposition}\label{prop:conjug}
Consider two groups $\Ga,\ti\Ga$ together with two pairs of automorphisms $(\al_0,\al_1)\in\Aut(\Ga)^2,(\ti\al_0,\ti\al_1)\in\Aut(\ti\Ga)^2$ and consider the associated fraction groups $G:=L\Ga\rtimes V, \ti G:=L\ti\Ga\rtimes V$ respectively.

If there exists an isomorphism $\beta\in\Isom(\Ga,\ti\Ga)$, a permutation $\sigma$ of $\{0,1\}$ and two elements $h_i\in \ti\Ga,i=0,1$ satisfying
$$\ti\al_{\sigma(i)} = \ad(h_i)\circ \beta\al_i\beta^{-1} \text{ for } i=0,1,$$
then $G$ is isomorphic to $\ti G.$
\end{proposition}

\begin{remark}\label{rem:exotic-isom}
It is tempting to believe that Proposition \ref{prop:conjug} provides a necessary condition for having two fraction groups isomorphic.
So far we have been unable to prove if this is the case or not. 
The existence of exotic $\varphi\in\NCV$ not sending $\Q_2$ onto itself nor onto the other copy of the dyadic rationals inside $\fC$ may imply that there exists some exotic isomorphisms between fraction groups not fulfilling the sufficient condition of the proposition.
For instance, consider a group $\Ga$, a pair of its automorphisms $(\al_0,\al_1)$ and a certain exotic homeomorphism $\varphi\in\NCV.$
Consider some finite words $p,q\in \fs$ with associated automorphisms $(\al_p,\al_q).$
Let $G,\ti G$ be the fraction groups associated to $(\Ga,\al_0,\al_1)$ and $(\Ga,\al_p,\al_q)$ respectively.
Could we have that the formula 
$$av\mapsto a^\varphi.\ad_\varphi(v),\ a\in L\Ga, v\in V$$
defines an isomorphism between $G$ and $\ti G$ even though we are not fulfilling the assumption of the last proposition?
\end{remark}

\subsection{Morphisms into centers}\label{sec:morph-center}

In this section we consider morphisms from a fraction group to the center of another fraction group. 
We show that we can multiply such a morphism together with an isomorphism and obtain a new isomorphism.
We will use this fact in the next section for decomposing isomorphisms between fraction groups.

\tbb{Throughout this subsection} we consider two groups $\Ga,\ti\Ga$, two pairs of automorphisms $(\al_0,\al_1)\in \Aut(\Ga)^2, (\ti\al_0,\ti\al_1)\in \Aut(\ti\Ga)^2$, the associated Jones actions $\pi:V\act L\Ga, \ti\pi:V\act L\ti\Ga$ and the associated fraction groups $G:=L\Ga\rtimes V, \ti G:=L\ti\Ga\rtimes V$.
As before, we write $\tau:V\times\fC\to\Aut(\Ga), (v,x)\mapsto \tau_{v,x}$ for the map satisfying
$$\pi_v(a(vx)) = \tau_{v,x}(a(x)), v\in V, x\in\fC, a\in L\Ga$$
of Proposition \ref{prop:twistedLoop} and similarly we use the symbol $\ti\tau$ for describing $\ti\pi.$
We introduce some notations and terminologies that will be used in this section and in the proof of Theorem \ref{theo:support}.

Let $\cI$ be the set of all finite union of sdi. If $I\in\cI$, then we write $I^c:=\fC\setminus I$ for the complement of $I$ inside $\fC$.
Recall that if $\varphi\in\NCV$ and $I\in\cI$, then $\varphi(I)\in\cI.$
If $I\in\cI$ and $g\in \Ga$, then we write $g_I$ for the function supported in $I$ that is constant in $I$ equal to $g$.
If $I\in\cI$ and $H\subset \Ga$ is a subgroup, then we consider the subgroups of $L\Ga$:
$$\tbb{L_I H = L_I(H):=\{ a\in LH:\ \supp(a)\subset I\}} \text{ and } D_IH=D_I(H):=\{ g_I:\ g\in H\}$$
and the subgroup of $\Ga$:
$$H^\al:=\{g\in H:\ g=\al_0(g)=\al_1(g)\}.$$
We say that an element of $\Ga^\al$ is $\al$-invariant.
\tbb{If $\Lambda$ is a group, then we write $e_\Lambda$ for its neutral element.}
We write $Z(\Ga)$ for the center of $\Ga$ and say that an element of $Z(\Ga)$ is central.
We use similar notations and terminologies for subgroups of $\ti\Ga$ and $L\ti\Ga$.
If $I\in\cI$ we write 
$$\Stab_V(I):=\{v\in V:\ v(I)=I\}$$ 
for the stabilizer subgroup of $I$ inside $V$ and 
$$\Fix_V(I):=\{v\in V:\ v(x)=x, \forall x\in I\}.$$ 
Observe that $\Stab_V(I)=\Stab_V(I^c)$ and $\Stab_V(I)$ is the subgroup of $V$ generated by $\Fix_V(I)$ and $\Fix_V(I^c).$
The next lemma describes certain centralizer subgroups.

\begin{lemma}\label{lem:centralizer}
For all $I\in\cI$ we have:
\begin{align*}
\{ a\in L\Ga:\ av = va,\ \forall v\in \Stab_V(I) \} & =D_I (\Ga^\al) \cdot D_{I^c} (\Ga^\al) \text{ and }\\
\{a\in L\Ga:\ av = va, \forall v\in \Fix_V(I) \} & = L_I\Ga\cdot D_{I^c}(\Ga^\al).
\end{align*}
Moreover, the center $Z(G)$ of $G$ is equal to $D_\fC(Z(\Ga)^\al)$: the subgroup of constant maps $a:\fC\to Z(\Ga)^\al.$
\end{lemma}

\begin{proof}
Fix $I\in\cI$.
Note that if the second equality is true for $I$ and $I^c$, then 
$$\{a\in L\Ga:\ av = va, \forall v\in \Fix_V(I)\cup \Fix_V(I^c) \} = D_I (\Ga^\al) \cdot D_{I^c} (\Ga^\al).$$
Since $\Stab_V(I)$ is generated by $\Fix_V(I)$ and $\Fix_V(I^c)$ we then obtain the first equality.
It is thus sufficient to prove the second equality.

Consider $a\in L_I\Ga\cdot D_{I^c}(\Ga^\al)$ and $v\in \Fix_V(I)$.
There exists $b\in L_I\Ga, g\in \Ga^\al$ such that $a= b\cdot g_{I^c}.$
Observe that $\pi_v(g_{I^c})(vx) = \tau_{v,x}(g_{I^c}(x))$ for $x\in \fC$.
If $x\in I$, then $\pi_v(g_{I^c})(vx) = \tau_{v,x}(e_\Ga) = e_\Ga.$
If $x\in I^c$, then $\pi_v(g_{I^c})(vx) = \tau_{v,x}(g) = g$ since $\tau_{v,x}$ is in the group generated by $\al_0,\al_1$ and $g$ is $\al$-invariant.
We deduce that $g_{I^c}$ commutes with $v$.
Now, since $v\in \Fix_V(I)$ it acts like the identity on the open set $I$ implying that $\tau_{v,x}=\id_\Ga$ for all $x\in I$.
Therefore, if $x\in I$, then $\pi_v(b)(x) = \pi_v(b)(vx) = \tau_{v,x}(b(x))=b(x)$.
If $x\notin I$, then $\pi_v(b)(vx) = \tau_{v,x}(e_\Ga)=e_\Ga$ since $b$ is supported in $I$.
We deduce that $b$ commutes with $v$ and thus 
$$ \{a\in L\Ga:\ av = va, \forall v\in \Fix_V(I) \} \supset L_I\Ga\cdot D_{I^c}(\Ga^\al).$$

Conversely, consider $a$ which commutes with all $v\in \Fix_V(I).$ 
If $I=\fC$, then $a$ trivially belongs to $L_\fC \Ga= L\Ga$ and we are done.
Hence, assume that $I$ is strictly contained inside $\fC$.
Consider a finite word $p$ in $0,1$ and the periodic sequence $x_p:=p^\infty\in\fC$.
We can find a $p$ such that $x_p$ is in $I^c$.
Indeed, the set of all $x_p$ for $p$ any finite word is dense in the Cantor space $\fC$ and $I^c$ is a nontrivial open subset of $\fC$.
For $d\geq 1$ large enough we have that $a$ is constant on the sdi $J$ defined such that $m_J=p^d$, see Notation \ref{nota:tau}.
Note that $x_p\in J.$
Choose $v\in V$ such that $v\in\Fix_V(I)$, $v$ is adapted to $J$ and $m_{v(J)}=p^{d+1}$.
In particular, $v(x_p)=x_p$ and $v(J)\subset J.$
Moreover, $\al_{v(J)}^{-1}\al_J=\al_p^{-1}.$
We obtain that
$$(vav^{-1})(vx_p) = \pi_v(a)(vx_p)=\tau_{v,x_p}( a(x_p) )=\al_{v(J)}^{-1}\al_J(a(x_p)) = \al_p^{-1}(a(x_p)).$$
Since $a$ commutes with $v$ and $vx_p=x_p$ we obtain that $a(x_p)=\al_p(a(x_p)).$
Consider now another $w\in \Fix_V(I)$ that is adapted to $J$ and satisfying that $m_{w(J)}=p^d \cdot 1\cdot p^d.$
Note that $w(J)\subset J$ since $m_J$ is a prefix of $m_{w(J)}.$
Observe that $\al_{w(J)}^{-1}\al_{J} = \al_p^{-d}\al_1^{-1}.$
We obtain that
$$(waw^{-1})(wx_p) =\pi_w(a)(wx_p)=\tau_{w,x_p}(a(x_p))=\al_p^{-d}\al_1^{-1}(a(x_p)).$$
Since $a$ commutes with $w$ and $a$ is constant on $J$ we obtain that $(waw^{-1})(wx_p)=a(x_p)$ and thus $\al_p^d(a(x_p)) = \al_1^{-1}(a(x_p))$.
Since $a(x_p)$ is $\al_p$-invariant we obtain that $\al_1(a(x_p))=a(x_p).$
A similar argument provides that $\al_0(a(x_p))=a(x_p).$
We have proved that $a(x_p)\in \Ga^\al.$
Since the set of $\{x_q:\ q \text{ a finite word }\}$ is dense in $I^c$ and $a$ is locally constant we obtain that $a(x)\in \Ga^\al$ for all $x\in I^c$.
%
Let us show that $a$ is constant on $I^c$.
Consider $x,y\in I^c\cap \Q_2$. There exists $v\in \Fix_V(I)$ satisfying $vx=y.$
Hence, 
$$a(y)=a(vx)=(vav^{-1})(vx) = \pi_v(a)(vx) = \tau_{v,x}(a(x)) = a(x)$$ since $a(x)\in \Ga^\al.$
Therefore, $a$ is constant on $I^c\cap \Q_2$ and thus constant on $I^c$ since $a$ is locally constant and since $I^c\cap \Q_2\subset I^c$ is a dense subset.
Hence, the restriction of $a$ to $I^c$ belongs to $D_{I^c}(\Ga^\al).$
This implies that $a\in L_I\Ga\cdot D_{I^c}(\Ga^\al).$
This proves the second equality of the lemma and as explained earlier this implies the first equality of the lemma.

Take $I=\fC$ and observe that $\Stab_V(\fC)=V.$
The first equality of the lemma implies that the commutant of $V$ inside $G$ is equal to $D_\fC(\Ga^\al).$
We deduce that 
$$Z(G)= D_\fC(\Ga^\al) \cap Z(L\Ga) = D_\fC(\Ga^\al) \cap LZ(\Ga) = D_\fC(Z(\Ga)^\al).$$
\end{proof}

We now prove a proposition regarding multiplications of isomorphisms and morphisms valued in centers.

\begin{proposition}\label{prop:center}
If $\beta:G\to\ti G$ is an isomorphism and $\ga:G\to Z(\ti G)$ is a morphism, then $\beta\ga:G\to \ti G, g\mapsto \beta(g)\cdot \ga(g)$ is an isomorphism.
Moreover, $$Z(G)\subset L(\Ga^\al)\rtimes V\subset G'\subset \ker(\ga)$$ where $G'=[G,G]$ is the derived subgroup of $G$ and $\ker(\ga)$ the kernel of $\ga.$
\end{proposition}
\begin{proof}
Consider $\beta,\ga$ as above.

We start by proving that $L(\Ga^\al)\rtimes V\subset G'$.
\tbb{Note that $V$ is simple and non-abelian} implying that $V=V'$ and thus $V\subset G'$ \cite[Theorem 6.9]{Cannon-Floyd-Parry96}.
Consider $g\in\Ga^\al$ and $I$ a proper sdi of $\fC$.
Choose $v\in V, J\in\cI$ satisfying that $v(J)=I\cup J$ and $I\cap J=\emptyset.$
Observe that 
$$[v,g_J] = vg_Jv^{-1} g_J^{-1} = g_{I\cup J} g_J^{-1} = g_I$$
implying that $g_I\in G'.$
We deduce that $L(\Ga^\al)\subset G'$ since $L(\Ga^\al)$ is generated by the subset $\{g_I:\ I \text{ sdi }, g\in \Ga^\al\}.$
Therefore, $L(\Ga^\al)\rtimes V\subset G'$ since this group is generated by $L(\Ga^\al)$ and $V$.

By Lemma \ref{lem:centralizer}, we have that $Z(G)=D_\fC(Z(\Ga)^\al)$ and thus $Z(G)\subset L(\Ga^\al).$
Since $Z(\ti G)$ is Abelian we have that $G'\subset \ker(\ga)$ implying that $$Z(G)\subset L(\Ga^\al)\rtimes V\subset G'\subset \ker(\ga).$$

Let us show that $\beta\ga$ is an isomorphism.
Since $\ga$ is valued in $Z(\ti G)$ it is obvious that $\beta\ga$ is multiplicative and thus is a group morphism.
Let us prove that $\beta\ga$ is injective.
Consider $g\in G$ so that $\beta\ga(g)=e_{\ti G}$ that is $\beta(g) = \ga(g)^{-1}.$
We deduce that $\beta(g)\in Z(\ti G)$ and thus $g\in Z(G)$ since $\beta$ restricts to an isomorphism from $Z(G)$ onto $Z(\ti G).$
We have proven that $Z(G)\subset \ker(\ga)$.
Therefore, $\ga(g)=e_{\ti G}$ and thus $\beta(g)=\ga(g)^{-1} = e_{\ti G}$. 
This implies that $g=e_G$ since $\beta$ is injective.
This proves that $\beta\ga$ is injective.

Fix $\ti g\in \ti G$ and consider $g:= \beta^{-1}( \ti g \cdot \ga( \beta^{-1} ( \ti g^{-1} ) ) )\in G.$
Observe that
\begin{align*}
\beta\ga(g) & = \beta(g)\cdot \ga(g)\\
& =  \ti g \cdot \ga( \beta^{-1} ( \ti g^{-1} ) ) \cdot \ga( \beta^{-1}(  \ti g \cdot \ga( \beta^{-1} ( \ti g^{-1} ) ) ) )\\
& = \ti g \cdot \ga( \beta^{-1} ( \ti g^{-1} ) ) \cdot \ga( \beta^{-1}(  \ti g )) \cdot \ga ( \beta^{-1}(\ga( \beta^{-1} ( \ti g^{-1} ) )))\\
& = \ti g \cdot \ga ( \beta^{-1}(\ga( \beta^{-1} ( \ti g^{-1} ) ))).\\
\end{align*}
Now, $\ga( \beta^{-1} ( \ti g^{-1} ) )\in Z(\ti G)$ implying that $\beta^{-1}(\ga( \beta^{-1} ( \ti g^{-1} ) ))\in Z(G)$. 
We have proven that $Z(G)\subset \ker(\ga)$ and thus $\ga ( \beta^{-1}(\ga( \beta^{-1} ( \ti g^{-1} ) )))=e_{\ti G}.$
We deduce that $\beta\ga(g)=\ti g$.
Hence, $\beta\ga$ is surjective.
We have proven that $\beta\ga$ is an isomorphism.
\end{proof}

\subsection{Classification of fraction groups}

We now prove the main theorem of this section. 
It is the most technical result of the paper which directly leads to a partial classification of the fraction groups considered and is a key result for describing all automorphisms of untwisted fraction groups.
We keep the same notations introduced in Section \ref{sec:morph-center}.
The difficulty resides in proving that there exists a decomposition such as the one given in Formula \ref{eq:decompo} satisfying (1). 
This will be done via a series of claims.
If we do not require (1), then this formula is a direct consequence of the fact that $\theta(L\Ga)=L\ti\Ga$ \tbb{(see Proposition \ref{prop:chara}).}
Moreover, note that (2) and (3) are easy consequence of (1).

\begin{theorem}\label{theo:support}
Consider two groups with a pair of automorphisms $(\Ga,\al_0,\al_1)$ and $(\ti\Ga,\ti\al_0,\ti\al_1)$ with associated fraction groups $G=L\Ga\rtimes V$ and $\ti G=L\ti\Ga\rtimes V$ respectively.
Assume we have an isomorphism $\theta:G\to\ti G.$

There exists a morphism $\zeta:G\to Z(\ti G)$, a homeomorphism $\varphi\in \NCV$, an isomorphism $\kappa^0:L\Ga\to L\ti\Ga$ and a cocycle $c:V\to L\ti\Ga$ such that 
\begin{equation}\label{eq:decompo}\theta(av) = \zeta(a)\cdot\kappa^0(a) \cdot c_v \cdot \ad_\varphi(v) \text{ for all } a\in L\Ga, v\in V.\end{equation}

Moreover, we have the following properties:
\begin{enumerate}
\item The isomorphism $\kappa^0$ is spatial in the following sense:
$$\supp(\kappa^0(a)) = \varphi(\supp(a)) \text{ for all } a\in L\Ga;$$
\item There exists a family of isomorphisms $(\kappa^0_x:\ x\in\fC)$ from $\Ga$ to $\ti\Ga$ satisfying that 
$$\kappa^0(a)(\varphi(x)) = \kappa^0_x(a(x)) \text{ for all } a\in L\Ga, x\in \fC;$$
\item For any $x\in\fC$ and $v\in V$ we have the following equality:
$$ \kappa^0_{vx} \circ \tau_{v,x}= \ad(c_v(\varphi(vx))) \circ \ti\tau_{\ad_\varphi(v),\varphi(x)} \circ \kappa^0_x.$$
\end{enumerate}
\end{theorem}

\begin{proof}
Consider $G,\Ga,\al_0,\al_1, \ti G, \ti \Ga, \ti\al_0,\ti\al_1, \theta$ as above and the notations introduced in Section \ref{sec:morph-center}.
Recall that by definition a sdi is always nonempty. 
For convenience we will further assume in this proof that a sdi is always different to $\fC$ and any element of $\cI$ is different from $\fC$ and $\emptyset$.

Proof of the decomposition \eqref{eq:decompo} satisfying Property (1):

Proposition \ref{prop:chara} implies that $\theta(L\Ga)=L\ti\Ga$ and thus there exists an isomorphism $\kappa\in\Isom(L\Ga,L\ti\Ga)$, $\phi=\ad_\varphi\in\Aut(V)$ (by Proposition \ref{prop:NCV}) and $c:V\to L\ti\Ga$ such that 
$$\theta(av) = \kappa(a) \cdot c_v \cdot \phi_v \text{ for all } a\in L\Ga, v\in V.$$
The notation $\phi=\ad_\varphi$ means that $\varphi\in \NCV$ is the unique homeomorphism of the Cantor space $\fC$ satisfying that $\phi_v=\varphi v\varphi^{-1}$ for all $v\in V$.
We want to show that if $a\in L\Ga$, then there exists $b\in L\ti\Ga$ and $h\in Z(\ti\Ga)^{\ti\al}$ such that $\kappa(a)= b\cdot h_\fC$ where $\supp(b)=\varphi(\supp(a))$ and $h_\fC$ is the constant function equal to $h$ everywhere.
From this we will be able to further decompose $\theta$ by writing $\kappa$ as $\kappa^0\zeta$ for suitable $\kappa^0$ and $\zeta.$

We now consider a certain subgroup of $\ti\Ga$ which measures the default of $\theta$ to be spatial.
Define 
$$M_{I,x}:= \{ \kappa(a)(\varphi(x)):\ a\in L_I\Ga\} \text{ for } I\in\cI, x\in I^c.$$
Our aim is to show that $M_{I,x}\subset Z(\ti\Ga)^{\ti\al}$ for all choice of $I\in\cI, x\in I^c.$ 
We start by an easy fact.

{\bf Claim 1: For all $I\in\cI, x\in I^c$ the subset $M_{I,x}\subset \ti\Ga$ is a normal subgroup.}

Fix $I\in\cI$ and $x\in I^c.$
Consider the map $\kappa_x:L\Ga\to \ti\Ga, a\mapsto \kappa(a)(\varphi(x))$.
It is a group morphism and $\kappa_x(L_I\Ga)=M_{I,x}$ implying that $M_{I,x}$ is a subgroup of $\ti\Ga.$
If $h\in\ti\Ga$, then there exists $a\in L\Ga$ so that $\kappa(a)=h_\fC$.
In particular,
$$ h M_{I,x} h^{-1} = \kappa_x( a L_I\Ga a^{-1}) = \kappa_x(L_I\Ga) = M_{I,x}$$
since $L_I\Ga\subset L\Ga$ is a normal subgroup.
This proves the claim.

{\bf Claim 2: For all $I,J\in\cI$ and $x \notin I\cup J$, we have $M_{I,x}=M_{J,x}$.}

Consider $I,J\in\cI$ and $x\in \fC$ so that $x\notin I\cup J$.
There exists $v\in V$ so that $v$ is adapted to $I$, $v(I)=J$ and $v(x)=x, v'(x)=1.$ This implies that $\phi_v(\varphi(x))=\varphi(x), \phi_v'(\varphi(x))=1$ and thus $\ti\tau_{\phi_v,\varphi(x)}=\id_{\ti\Ga}.$
Now observe that
$$M_{J,x} = \kappa_x( L_J\Ga) = \kappa_x( v L_I\Ga v^{-1}).$$

If $a\in L_I\Ga$, then
\begin{align*}
\kappa_x(v a v^{-1}) & = \kappa(vav^{-1})(\varphi(x)) = \kappa(vav^{-1})(\varphi(vx))\\
& = [\theta(v) \kappa(a) \theta(v)^{-1} ] (\varphi(vx)) = [c_v \phi_v \kappa(a) \phi_v^{-1} c_v^{-1}](\varphi(vx))\\
& = \ad(c_v(\varphi(vx)))\circ \ti\tau_{\phi_v,\varphi(x)}(\kappa(a)(\varphi(x)))\\
& = \ad(c_v(\varphi(vx)))( \kappa(a)(\varphi(x))) \in \ad(c_v(\varphi(vx)))( M_{I,x}) = M_{I,x}
\end{align*}
since $M_{I,x}\subset \ti\Ga$ is a normal subgroup.
We deduce that $M_{J,x}\subset M_{I,x}$ and by reversing the roles of $I$ and $J$ we deduce that $M_{J,x}\supset M_{I,x}$ implying the equality $M_{J,x}=M_{I,x}$.

{\bf Claim 3: We have that $M_{I,x}$ is a subgroup of $Z(\ti\Ga)$ for all $I\in\cI, x\in I^c.$}

Fix $I\in\cI$ and $x\notin I$.
We start by proving that $M_{I,x}$ is an Abelian group.
Consider $g,h\in M_{I,x}$ and two disjoint sdi $I_0,I_1$ that are contained in $I$. In particular, $x\notin I_0$ and $x\notin I_1.$
Claim 2 implies that there exists $a\in L_{I_0}\Ga,b\in L_{I_1}\Ga$ so that $g=\kappa(a)(\varphi(x)), h=\kappa(b)(\varphi(x)).$
\tbb{Since $a$ and $b$ have disjoint support they mutually commute and thus so do $g$ and $h$.}
We have proven that $M_{I,x}$ is Abelian.

Consider $g\in M_{I,x}$ and $h\in \ti\Ga.$
There exists $a\in L_I\Ga$ and $b\in L\Ga$ so that $g=\kappa(a)(\varphi(x)), h = \kappa(b)(\varphi(x)).$
Decompose $b$ as $b=b_I\cdot b_{I^c}$ where $b_I\in L_I\Ga, b_{I^c}\in L_{I^c}\Ga.$
We have seen that $\kappa(a)(\varphi(x))$ and $\kappa(b_I)(\varphi(x))$ mutually commute.
Now, $ab_I$ commutes with $b_{I^c}$ since they have disjoint support and thus so do $\kappa(ab_I)(\varphi(x))$ and $\kappa(b_{I^c})(\varphi(x))$.
We deduce that $\kappa(a)(\varphi(x))$ and $\kappa(b)(\varphi(x))$ mutually commute and thus $gh=hg.$
This proves the claim.

{\bf Claim 4: We have that $\kappa(L_I\Ga)\subset L_{\varphi(I)}\ti\Ga \cdot D_{\fC}(Z(\ti\Ga)^{\ti\al})$ for all $I\in\cI.$ 
In particular, $M_{I,x}$ is a subgroup of $Z(\ti\Ga)^{\ti\al}$ for all $I\in\cI, x\notin I$.}

Fix $I\in\cI, a\in L_I\Ga$ and $v\in \Fix_V(I).$
Note that $a$ commutes with $v$ by Lemma \ref{lem:centralizer}.
Let us show that $\kappa(a)$ commutes with $\Fix_V(\varphi(I)).$
Consider $x\in\fC$.
We have that:
\begin{align*}
\kappa(a)(\varphi(vx)) & = \kappa(\pi_v(a))(\varphi(vx))\\
& = [\theta(v) \kappa(a) \theta(v)^{-1} ] (\varphi(vx) ) \\
& = \ad(c_v(\varphi(vx))) \circ \ti\tau_{\phi_v,\varphi(x)} (\kappa(a)(\varphi(x))).
\end{align*}
Now, if $x\in I$, then $\ti\tau_{\phi_v,\varphi(x)}=\id_{\ti\Ga}$ since $\phi_v$ acts like the identity on a neighborhood of $\varphi(x).$
Hence, $\kappa(a)(\varphi(vx)) =\ad(c_v(\varphi(vx)))( \kappa(a)(\varphi(x)))$ and since $vx=x$ we deduce that $\kappa(a)(\varphi(vx))$ commutes with $c_v(\varphi(vx))$ for all $x\in I$.
If $x\notin I$, then $\kappa(a)(\varphi(vx))$ is central by Claim 3 implying that $c_v(\varphi(vx))$ commutes with it. 
\tbb{This proves that $c_v$ and $\kappa(a)$ commute.}
Since $av=va$ we have that $\theta(av)=\theta(va)$ and thus $\kappa(a)$ and $\theta(v)$ commute.
\tbb{Now, $\theta(v)=c_v\cdot \phi_v$ and $\kappa(a)$ commute} with both $\theta(v)$ and $c_v$.
Therefore, $\kappa(a)$ commutes with $\phi_v$.
We have proven that $\kappa(a)$ commutes with $\phi(\Fix_V(I))$ implying that it commutes with $\Fix_V(\varphi(I))$ since $\Fix_V(\varphi(I))=\phi(\Fix_V(I)).$

By applying Lemma \ref{lem:centralizer} to $\Fix_V(\varphi(I))$ we deduce that $\kappa(a)\in L_{\varphi(I)}\ti\Ga\cdot D_{\varphi(I)^c} \ti\Ga^{\ti\al}$ that is $\kappa$ is constant on the complement of $\varphi(I)$ and takes a value in $\ti\Ga^{\ti\al}$.
By Claim 3 we know that $\kappa(a)(\varphi(x))$ is in the center of $\ti\Ga$ if $x\notin I$.
We deduce that $\kappa(a)\in L_{\varphi(I)}\ti\Ga\cdot D_{\varphi(I)^c} Z(\ti\Ga)^{\ti\al}.$
Since $L_{\varphi(I)}\ti\Ga\cdot D_{\varphi(I)^c} Z(\ti\Ga)^{\ti\al}=L_{\varphi(I)}\ti\Ga\cdot D_{\fC} Z(\ti\Ga)^{\ti\al}$ the claim is proven.

By Lemma \ref{lem:centralizer} we have $Z(\ti G) = D_{\fC} Z(\ti\Ga)^{\ti\al}$ and thus we have proven:
$$\kappa(L_I\Ga)) \subset L_{\varphi(I)}\ti\Ga\cdot Z(\ti G) \text{ for all } I\in\cI.$$
We will now define the desirable morphism $\zeta:G\to Z(\ti G)$.

{\bf Claim 5: There exists a unique group morphism 
$$\zeta:G\to Z(\ti G)$$
satisfying that for all $a\in L\Ga$ we have a decomposition
$\kappa(a) = b\cdot \zeta(a)$ where $b\in L\ti\Ga, \supp(b)\subset \varphi(\supp(a))$.}

Fix $I\in\cI$.
Claim 4 proved that $\kappa(L_I\Ga)\subset L_{\varphi(I)}\ti\Ga\cdot Z(\ti G).$
By convention $I\neq\fC$ implying that $L_{\varphi(I)}\ti\Ga\cap Z(\ti G)=\{e_{\ti G} \}$ since all elements of $Z(\ti G) = D_{\fC}Z(\ti\Ga)^{\ti\al}$ have either full or trivial support and no elements of $L_{\varphi(I)}\ti\Ga$ have full support.
Moreover, these two subgroups mutually commute and thus $L_{\varphi(I)}\ti\Ga\cdot Z(\ti G)$ is isomorphic to the direct product of groups $L_{\varphi(I)}\ti\Ga\oplus Z(\ti G).$
Therefore, there exists unique morphisms $\kappa^I:L_I\Ga\to L_{\varphi(I)}\ti\Ga, \zeta^I:L_I\Ga\to Z(\ti G)$ satisfying that 
$$\kappa(a) = \kappa^I(a)\cdot \zeta^I(a) \text{ for all } a\in L_I\Ga.$$

Consider another $J\in \cI$ and assume that $I\cap J\neq\emptyset.$
If $a\in L_{I\cap J}\Ga$, then 
$$\kappa(a) = \kappa^I(a)\cdot \zeta^I(a) = \kappa^J(a)\cdot \zeta^J(a).$$
Evaluating this equality at $\varphi(x)$ with $x\notin I\cap J$ we deduce that $\zeta^I(a)(\varphi(x))=\zeta^J(a)(\varphi(x))$.
Since $\zeta^{I}(a)$ and $\zeta^J(a)$ are constant functions we deduce that $\zeta^{I}(a)=\zeta^J(a).$

We now define $\zeta^\fC$.
Let $I_0,I_1$ be the first and second of half of $\fC$ respectively.
For all $a\in L\Ga$ we can can decompose uniquely $a$ as $a=a_0\cdot a_1$ where $a_k\in L_{I_k}\Ga$ for $k=0,1.$
We put:
$$\zeta^\fC(a):= \zeta^{I_0}(a_0)\cdot \zeta^{I_1}(a_1) \text{ for all } a\in L\Ga.$$
\tbb{By following a similar argument to that above} and considering $I\cap I_0$ and $I\cap I_1$ we deduce that $\zeta^\fC(a)=\zeta^I(a)$ for all $I\in \cI$ and $a\in L_I\Ga.$
Since each $\zeta^I, I\in\cI$ are morphisms valued in an Abelian group we deduce that $\zeta^\fC$ is a morphism.

Consider $v\in V$ adapted to $I$ so that $v(I)=K$ for a certain $K\in\cI$.
If $a\in L_I\Ga$, then $vav^{-1}\in L_K\Ga$ and thus:
$$\kappa(vav^{-1}) =c_v \phi_v \kappa(a) \phi_v^{-1} c_v^{-1} = c_v \phi_v \kappa^I(a)\cdot \zeta^I(a) \phi_v^{-1} c_v^{-1} = c_v \phi_v \kappa^I(a) \phi_v^{-1} c_v^{-1}\cdot \zeta^I(a).$$
Since $c_v \phi_v \kappa^I(a) \phi_v^{-1} c_v^{-1}$ is supported in $\varphi(K)$ we deduce that $\zeta^I(a) = \zeta^K(vav^{-1}).$

We obtain that $\zeta^\fC(a) = \zeta^\fC(vav^{-1}).$
This allows us to extend the morphism $\zeta^\fC:L\Ga\to Z(\ti G)$ into a morphism:
$$\zeta:G\to Z(\ti G)$$
so that $\zeta(v)=e_{\ti G}$ for all $v\in V$ and $\zeta(a)=\zeta^\fC(a)$ for all $a\in L\Ga.$

By construction the morphism $\zeta$ satisfies the properties of the claim.

{\bf End of the proof of Formula \eqref{eq:decompo} satisfying Property (1):}

\tbb{Define $\zeta^{\dag}$ to be the morphism:
$$\zeta^{\dag}: G\to Z(\ti G), g\mapsto \zeta(g)^{-1}.$$}
This is indeed a morphism since $Z(\ti G)$ is Abelian and $\zeta$ is a morphism.
\tbb{Consider $\theta^0:= \theta\zeta^{\dag}$} which is an isomorphism by Proposition \ref{prop:center}.
Note that $\theta=\theta^0\zeta$ and $\theta^0(av) = \kappa^0(a)\cdot c_v\cdot \phi_v$ for all $a\in L\Ga, v\in V$ where $\kappa^0$ is of the form:
$$\kappa^0:L\Ga\to L\ti \Ga, a\mapsto \kappa(a)\cdot \zeta(a)^{-1}.$$
We have that $\kappa^0$ is an isomorphism since $\theta^0$ is and moreover $\supp(\kappa^0(a))\subset \varphi(\supp(a))$ for all $a\in L\Ga$ by construction of $\zeta$. 

Let us show that $\supp(\kappa^0(a))=\varphi(\supp(a))$ for all $a\in L\Ga.$
Assume this is not the case.
There exists a nonempty finite union of sdi $I$ (possibly equal to $\fC$) and $a\in L\Ga$ with support $I$ so that $\supp(\kappa^0(a))=\varphi(J)$ where $J\subset I$ is strictly contained inside $I$.
Consider now the inverse $(\theta^0)^{-1}$ of $\theta^0$ and $a=(\theta^0)^{-1}(\kappa^0(a)).$
Applying the same reasoning that we have done for $\theta$ to $(\theta^0)^{-1}$ we obtain that 
$$a=(\theta^0)^{-1}(\kappa^0(a)) = b\cdot h_{\fC}$$
where $b$ is supported in $J$ and $h\in Z(\Ga)^\al.$
Since $J$ is a proper subset of $I$ and $a$ has support $I$ we must have $h\neq e_\Ga.$
We now reapply $\kappa^0$ and obtain:
$$\kappa^0(a) = \kappa^0(b)\cdot \kappa^0(h_\fC).$$
Since $b$ is supported in $J$ we have that $\supp(\kappa^0(b))\subset \varphi(J)$. 
Now, $h_\fC\in Z(G).$
Since $\theta^0$ is an isomorphism it maps $Z(G)$ onto $Z(\ti G)$. 
Moreover, $Z(\ti G) = D_\fC(Z(\ti\Ga)^{\ti\al})$. 
Therefore, $\theta^0(h_\fC)=\kappa^0(h_\fC) = k_\fC$ for a certain $k\in Z(\ti\Ga)^{\ti\al}$. 
Since $h$ is nontrivial so is $k$ implying that $\varphi(J)^c$ is contained in the support of $\kappa^0(a)$.
This contradicts the assumption $\supp(\kappa^0(a))=\varphi(J)$.
We have proven that $\supp(\kappa^0(a))=\varphi(\supp(a))$ for all $a\in L\Ga.$
This concludes the proof of the first part of the theorem.

Up to considering $\theta^0$ rather than $\theta$ we now assume that $\zeta$ is trivial and thus $\kappa^0=\kappa$. This will make our notations slightly lighter.

Proof of (2):

Let us prove that $\kappa$ can be decomposed as a product of isomorphisms $\prod_{x\in\fC}\kappa_x.$
Define the map 
$$\kappa_{x,I}:\Ga\to\ti\Ga, g\mapsto \kappa(g_I)(\varphi(x)) \text{ where } x\in\fC, I\in\cI.$$
By the first part of the proof we know that, if $x\notin I$, then $\kappa_{x,I}(g)=e_{\ti\Ga}$ for all $g\in \Ga$ since $\supp(\kappa(g_I))\subset\varphi(I).$
This implies that if $x\in I$, then 
$$\kappa_{x,\fC}(g) = \kappa(g_\fC)(\varphi(x)) = \kappa(g_I\cdot g_{\fC\setminus I})(\varphi(x)) = \kappa_{x,I}(g)\cdot \kappa_{x,\fC\setminus I}(g) = \kappa_{x,I}(g)$$
for all $g\in \Ga$.
We deduce that $\kappa_{x,\fC}=\kappa_{x,I}$ for all $x\in \fC$ and $I\in\cI$ satisfying $x\in I$.
We now write $\kappa_x:=\kappa_{x,\fC}$ for $x\in\fC$ which is a group morphism since $\kappa$ is.
Consider $a\in L\Ga$ and note that since $a$ is locally constant we can find a sdp $(I_1,\cdots,I_n)$ and some elements $g^1,\cdots,g^n\in\Ga$ satisfying that $a= g^1_{I_1}\cdots g^n_{I_n}$.
Given $x\in \fC$, there exists a unique $1\leq j\leq n$ such that $x\in I_j$ and thus:
$$\kappa(a)(\varphi(x)) = \kappa_x(g^1)\cdots \kappa_x(g^n) = \kappa_x(g^j) = \kappa_x(a).$$
Consider $x\in\fC$ and let us show that $\kappa_x$ is an isomorphism from $\Ga$ to $\ti\Ga.$
If $\ti g\in \ti\Ga$, then since $\kappa$ is surjective there exists $a\in L\Ga$ such that $\kappa(a)=\ti g_\fC$ and in particular $\kappa(a)(\varphi(x))=\kappa_x(a(x)) = \ti g$ implying that $\kappa_x$ is surjective.
If $\kappa_x$ is not injective, then there exists $g\in \Ga,g\neq e_\Ga$ such that $\kappa_x(g)=e_{\ti\Ga}.$
Since $\kappa(g_\fC)$ is locally constant we can find a sdi $I$ such that $\kappa_y(g)=e_{\ti\Ga}$ for all $y\in I$ implying that $\kappa(g_I)=e_{L\ti\Ga}$, a contradiction since $\kappa$ is injective.

Proof of (3):

Consider $x\in\fC, v\in V, g\in\Ga$ and a sdi $I$ containing $x$ so that $v$ is adapted to $I$.
Recall that $vg_Iv^{-1} = \pi_v(g_I)= [\al_{v(I)}^{-1} \al_I(g)]_{v(I)}$ the function supported in $v(I)$ and equal to $\al_{v(I)}^{-1} \al_I(g)$ on $v(I)$.
Observe that
$$
\kappa([\al_{v(I)}^{-1} \al_I(g)]_{v(I)} ) = \theta(vg_Iv^{-1})  = \theta(v)\theta(g_I)\theta(v)^{-1} = c_v\cdot \phi_v \cdot \kappa(g_I) \cdot \phi_v^{-1}\cdot c_v^{-1}.$$
Evaluated at $\varphi(vx)$ we obtain that 
$$\kappa([\al_{v(I)}^{-1} \al_I(g)]_{v(I)})(\varphi(vx)) = \ad(c_v(\varphi(vx)))[ \ti\tau_{\phi_v,\varphi(x)}(\kappa(g_I)(\varphi(x))) ].$$
Using that $\al_{v(I)}^{-1}\al_I=\tau_{v,x}$ and the decomposition $\prod_{y\in\fC}\kappa_y$ of $\kappa$ we obtain the equality:
$$\kappa_{vx}\circ\tau_{v,x} = \ad(c_v(\varphi(vx))) \circ \ti\tau_{\phi_v,\varphi(x)} \circ \kappa_x.$$
\end{proof}

\begin{remark}\label{rem:support}
Consider an isomorphism $\theta:G\to \ti G$.
We have proven in Proposition \ref{prop:center} that $L(\Ga^\al)\rtimes V\subset G'$ and thus belongs to the kernel of the map $\zeta:G\to Z(\ti G)$ of the theorem.
In particular, in the untwisted case we have by definition that $\al_0=\al_1=\id_\Ga$ implying that $\Ga=\Ga^\al$ and thus $\zeta$ is trivial. 
Hence, we automatically have that $\supp(\theta(a))=\varphi(\supp(a))$ for all $a\in L\Ga$ when $G$ is an untwisted fraction group.
This will be particularly relevant in the last section of this article where we will be considering $\Aut(G)$ for $G$ an untwisted fraction group.
\\
Interestingly, we have the same result in the opposite case: when there are few $\al$-invariant elements.
Indeed, assume that $Z(\Ga)^\al$ is trivial. Since this is isomorphic to $Z(G)$ by Lemma \ref{lem:centralizer} we deduce that $Z(G)$ is trivial and thus so is $Z(\ti G)$ since $G$ is isomorphic to $\ti G$.
Therefore, any morphism $\zeta:G\to Z(\ti G)$ is trivial when there are no central $\al$-invariant elements of $\Ga$. 
Hence, Theorem \ref{theo:support} implies that all isomorphisms from $G$ to $\ti G$ are spatial.
\end{remark}

\begin{example}
Here is an example of a triple $(\Ga,\al_0,\al_1)$ with associated fraction group $G:=L\Ga\rtimes V$ and a nontrivial morphism $\zeta: G\to Z(G).$ Using $\zeta$ we will then construct a nontrivial non-spatial automorphism of $G$.
\\
Consider the additive group $\Ga:=\Z[1/2]\times\Z[1/2]$ the automorphisms $\al=\al_0=\al_1$ so that 
$$\al(t,r):= (t, \frac{r}{2}), \ t,r\in \Z[1/2].$$
If $I$ is a sdi, then we write $m_I$ for its associated finite word, $|m_I|$ for the number of letters in this word and $\Leb(I)$ for the Lebesgue measure of $I$.
Note that $\Leb(I)= 2^{-|m_I|}$ for all sdi $I$.
Given $t,r\in \Z[1/2]$ and a sdi $I$ consider$(t,r)_I \in L\Ga$ the element supported in $I$ taking the value $(t,r)$ in $I$ and define 
$$\zeta( (t,r)_I):= \Leb(I) \cdot r.$$
This extends uniquely into a group morphism from $L\Ga$ to $\Z[1/2]$.
The formula is the following: 
consider $a\in L\Ga$ so that $a= (t_1,r_1)_{I_1}\cdots (t_n,r_n)_{I_n}$ for a sdp $(I_1,\cdots,I_n)$ and some $t_j,r_j\in \Z[1/2], 1\leq j\leq n.$ 
We have
$$\zeta( a) =\sum_{k=1}^n Leb(I_k) \cdot r_k.$$
Observe that $\zeta$ is $V$-invariant. Indeed consider $a\in L\Ga$ and $v\in V$.
Up to decomposing $a$ over a partition adapted to $v$ and $a$ we can assume that $a=(t,r)_I$ where $t,r\in\Z[1/2]$ and $I$ is a sdi adapted to $v.$
We have that 
$$v(t,r)_I v^{-1} = [\al_{v(I)}^{-1} \al_I (t,r)]_{v(I)} = [\al^{|m_I| - |m_{v(I)}|} (t,r)]_{v(I)} = [2^{|m_{v(I)}|-|m_I|} \cdot (t,r)]_{v(I)}.$$
We deduce the following:
\begin{align*}
\zeta( v (t,r)_I v^{-1} ) & = \zeta( [2^{|m_{v(I)}|-|m_I|}\cdot  (t,r)]_{v(I)} ) =\Leb(v(I)) \cdot 2^{|m_{v(I)}| - |m_I|} \cdot  r \\
& = \Leb(I)\cdot  r = \zeta( (t,r)_I)
\end{align*}
proving that $\zeta$ is $V$-invariant.
Therefore, $\zeta$ extends into a group morphism 
$$\zeta:G\to \Z[1/2], av\mapsto \zeta(a).$$
By identifying $\Z[1/2]$ with $Z(G)$ we obtain a nontrivial morphism
$$\zeta:G\to Z(G).$$
Therefore, $G$ admits automorphisms that are not spatial such as the following:
$$G\ni g\mapsto g\cdot \zeta(g) \in G.$$
\end{example}

\begin{corollary}\label{cor:isom}
Consider two groups $\Ga,\ti\Ga$ and two pairs of automorphisms $\al_0,\al_1\in\Aut(\Ga), \ti\al_0,\ti\al_1\in\Aut(\ti\Ga).$
Denote by $G=L\Ga\rtimes V$ and $\ti G=L\ti\Ga\rtimes V$ the fraction groups associated to $(\Ga,\al_0,\al_1)$ and $(\ti\Ga,\ti\al_0,\ti\al_1)$ respectively.
The following assertions are true:
\begin{enumerate}
\item If $G\simeq \ti G$, then $\Ga\simeq\ti\Ga.$ 
\item Assume that $\ti\al_0,\ti\al_1$ are inner automorphisms. We have that $G\simeq\ti G$ if and only if $\Ga\simeq\ti\Ga$ and $\al_0,\al_1$ are inner automorphisms.
\end{enumerate}
\end{corollary}
\begin{proof}
Proof of (1).
This is a direct consequence of Theorem \ref{theo:support}.2.

Proof of (2).
Assume that $\ti\al_0$ and $\ti\al_1$ are inner automorphisms and suppose that there exists an isomorphism $\theta: G\to \ti G$.
Up to multiplying $\theta$ by \tbb{the morphism $\zeta^{\dag}$} of Theorem \ref{theo:support} we may assume that 
$$\theta(av)=\kappa(a)\cdot c_v\cdot \ad_\varphi(v), \ a\in L\Ga, v\in V$$
so that $\supp(\kappa(a))=\varphi(\supp(a))$ for all $a\in L\Ga.$
\tbb{Hence, $\kappa$ decomposes as in the second and third item of that theorem.}
Therefore,
$$\kappa_{vx}\circ \tau_{v,x} = \ad(c_v(\varphi(vx)))\circ \ti\tau_{\ad_\varphi(v),\varphi(x)}\circ \kappa_x$$
for all $v\in V, x\in\fC$.
In particular, $\Ga$ and $\ti\Ga$ are isomorphic since $\kappa_x\in\Isom(\Ga,\ti\Ga)$ for all $x\in\fC.$
Since $\ti\al_0,\ti\al_1$ are inner automorphisms so is $\ti\tau_{\ad_\varphi(v),\varphi(x)}.$
Therefore, given any $v\in V,x\in\fC$ there exists $h_{v,x}\in\ti\Ga$ so that
$$\kappa_{vx}\circ\tau_{v,x} = \ad(h_{v,x})\circ \kappa_x.$$
In particular, consider $x_0=0^\infty\in\fC$ the infinite sequence of $0$, $I$ the first fourth of $\fC$ and $v\in V$ adapted to $I$ such that $v(I)$ is the first half of $\fC$.
We have that $m_{I}=00, m_{v(I)}=0$ and thus $\al_{v(I)}^{-1}\al_I=\al_0.$
Moreover, $vx_0=x_0.$
The equation of above gives:
$$\kappa_{x_0}\circ \al_0=\ad(h_{v,x_0})\circ \kappa_{x_0}.$$
Therefore, $\al_0$ is inner.
A similar argument applied to the infinite sequence of $1$ provides that $\al_1$ is inner.
Conversely, assume that $\Ga\simeq\ti\Ga$ and all the automorphisms $\al_0,\al_1,\ti\al_0,\ti\al_1$ are inner. 
Lemma \ref{lem:isomaut} implies that $G$ is isomorphic to $G(\Ga,\id_\Ga,\id_\Ga)$ the group associated to $(\Ga,\id_\Ga,\id_\Ga)$. 
The same lemma applied to $(\ti \Ga,\ti\al_0,\ti\al_1)$ implies that $\ti G$ is isomorphic to $G(\ti\Ga,\id_{\ti\Ga},\id_{\ti\Ga}).$ 
Since $\Ga\simeq \ti\Ga$ we have that $G(\Ga,\id_\Ga,\id_\Ga)\simeq G(\ti\Ga,\id_{\ti\Ga},\id_{\ti\Ga})$ that is $G\simeq \ti G.$
\end{proof}

\subsection{Classification of a class of co-context-free groups}\label{sec:coCF}

Recall that a group is co-word-free (in short $co\mathcal{CF}$) or is a co-context-free group if its co-word problem with respect to a finite generating subset is a context-free language.
A conjecture of Lehnert, combined with a theorem of Bleak, Matucci and Neunh\"offer, states that  a group is $co\mathcal{CF}$ if and only if it is finitely generated and embeds inside Thompson group $V$ \cite{Lehnert08-thesis,BleakMatucciNeunhoffer16}.
Candidates for counterexamples of this conjecture have been proposed in \cite{BZFGHM18}.
They have been initially constructed using cloning systems of Witzel and Zaremsky \cite{Witzel-Zaremsky18,Zaremsky18-clone}.

Here is a description of this class of groups using our framework.
Consider a finite group $\Ga$ and an automorphism $\al_1\in\Aut(\Ga).$
Let $G:=G_{\Ga,\al_1}=L\Ga\rtimes V$ be the fraction group associated to the triple $(\Ga, \id_\Ga,\al_1).$
The class of groups $G_{\Ga,\al_1}$ of above is exactly the class considered in \cite{BZFGHM18}.
It has been proved that they are all $co\mathcal{CF}$ groups and conjectured that some of them do not embed inside $V$ (which would disprove the conjecture of Lehnert).

We start by giving a concrete and practical description of these groups.
Fix a finite group $\Ga$ and one of its automorphism $\al_1.$
To follow the previous notation write $\al_0$ for the identity automorphism $\id_\Ga.$
Consider the Jones action $\pi:V\act L\Ga$ associated to the triple $(\Ga,\al_0,\al_1)=(\Ga,\id_\Ga,\al_1).$
As already mentioned the group $G=G_{\Ga,\al_1}$ is isomorphic to the semidirect product $L\Ga\rtimes V$ constructed from the Jones action $\pi.$
Recall that $L\Ga$ is the group of continuous functions from the Cantor space $\fC=\is$ to the group $\Ga$ when $\Ga$ is equipped with the discrete topology.
The description of the Jones action given in Proposition \ref{prop:twistedLoop} gives that 
$$\pi_v(a)(vx) =\tau_{v,x}(a(x))=\al_1^{N_{v,x}}(a(x)), x\in\fC,v\in V, a\in L\Ga$$ where $N_{v,x}\in\Z$ is an integer depending on $v$ and $x$.
Define 
$$f:\fs\to \Z, (x_i)_{i\geq 1}\mapsto \sum_{i\geq 1} x_i$$ and observe that if $I$ is a sdi with associated word $m_I$ (i.e.~$I$ is the set of words with prefix $m_I$), then $\al_I=\al_1^{f(m_I)}.$
Therefore, if $x\in\fC,v\in V$ and $I$ is a sdi adapted to $v$ and containing $x$ we have that 
$$\tau_{v,x}= \al_{v(I)}^{-1} \al_I = \al_1^{-f(m_{v(I)})} \al_1^{f(m_I)}=\al_1^{  f(m_I)-f(m_{v(I)})}.$$
Hence, $\tau_{v,x}=\al_1^{N_{v,x}}$ where $N_{v,x}=f(m_I)-f(m_{v(I)})$ and $x,v,I$ are as above.

We present a partial classification of the class of these groups. 
This is deduced from Theorem \ref{theo:support}.
\begin{corollary}\label{cor:coCF}
Consider some finite groups $\Ga,\ti\Ga$, some automorphisms $\al_1\in\Aut(\Ga),\ti\al_1\in\Aut(\ti\Ga)$, and the groups $G,\ti G$ associated to the triples $(\Ga,\id_\Ga,\al_1), (\ti\Ga,\id_{\ti\Ga},\ti\al_1)$ respectively.
If $G$ is isomorphic to $\ti G$, then there exists an isomorphism $\beta:\Ga\to\ti\Ga$.
Moreover, there exists $h\in\Ga,\ti h\in\ti \Ga$ and $n,\ti n\in\Z$ such that
$$\al_1=\ad(h)\circ \beta^{-1} \ti\al_1^{\ti n}\beta \text{ and } \ti\al_1=\ad(\ti h)\circ \beta \al_1^{n}\beta^{-1}.$$
In other words, the map $$\ga\in\Aut(\Ga)\mapsto \beta\ga\beta^{-1}\in\Aut(\ti\Ga)$$ realises an isomorphism from the subgroup generated by $\al_1$ in $\Out(\Ga):=\Aut(\Ga)/\Inn(\Ga)$ onto the subgroup generated by $\ti\al_1$ inside $\Out(\ti\Ga).$
\end{corollary}
Note that we do not know if in the conclusion of the the corollary we can assume that $n=\ti n=1.$
\begin{proof}
Consider $\Ga,\ti\Ga,\al_1,\ti\al_1,G,\ti G$ as above and assume that $\theta:G\to\ti G$ is an isomorphism. 
Choose this isomorphism in such a way that the $\zeta$ of Theorem \ref{theo:support} is trivial. (\tbb{Take for instance $\theta\zeta^{\dag}$} rather than $\theta$.)
By Theorem \ref{theo:support} there exists a family of isomorphisms $(\kappa_x,x\in\fC)$ from $\Ga$ to $\ti\Ga$, a cocycle $c:V\to L\ti\Ga$, and a homeomorphism $\varphi\in\NCV$ satisfying that 
$$\theta(av) = \kappa(a)\cdot c_v\cdot \phi_v, \text{ for all } a\in L\Ga,v\in V$$
where $\kappa=\prod_{x\in\fC}\kappa_x$ and $\phi_v:=\ad_\varphi(v).$
Moreover, we have the equality:
$$\kappa_{vx}\circ\tau_{v,x}=\ad(c_v(\varphi(vx)))\circ \ti\tau_{\phi_v,\varphi(x)}\circ \kappa_x \text{ for all } v\in V, x\in\fC.$$
Observe that since $\Ga$ is finite we have that $x\mapsto \kappa_x$ is locally constant.
Indeed, for any $g\in \Ga,x\in\fC$ we have that $\kappa(g_\fC)(\varphi(x))=\kappa_x(g)$. Since $\kappa(g_\fC)$ is locally constant so is $x\mapsto \kappa_x(g).$
Hence, for any $g\in \Ga$ there exists a sdp $P_g$ adapted to $\kappa(g_\fC)$, i.e.~$\kappa(g_\fC)$ is constant on each sdi of the sdp $P_g$.
Since $\Ga$ is finite we can find a sdp $P$ that is thinner than all $P_g,g\in \Ga.$
We deduce that $x\mapsto \kappa_x$ is constant on each sdi of $P$.
Consider any sdi $I$ in the sdp $P$ and write $\beta$ for the isomorphism $\kappa_x$ where $x\in I$.
Consider the sdi $J$ satisfying that $m_J:= m_I\cdot 11.$
There exists $v\in V$ that is adapted to $J$ and satisfies that $m_{v(J)}=m_I\cdot 1.$
Note that $J,v(J)\subset I$ and $\al_{v(J)}^{-1}\al_J=\al_1.$
This implies the equality:
$$\beta\circ \al_1 = \ad(c_v(\varphi(vx)))\circ \ti\tau_{\phi_v,\varphi(x)}\circ \beta \text{ for all } x\in J.$$
By definition of $\ti\tau$ we have that $\ti\tau_{\phi_v,\varphi(x)}$ is a power of $\ti\al_1.$
We deduce that
$$\al_1=\ad(h)\circ \beta^{-1} \ti\al_1^{\ti n}\beta$$ for $h:=\beta^{-1}(c_v(\varphi(vx)))\in\Ga$ and a certain $\ti n\in\Z.$
A similar construction to that above proves that there exists $y\in I$ and $v\in V$ such that $v(y)\in I$ and $\ti\tau_{\phi_v,\varphi(y)}=\ti\al_1$ giving us the second equality of the corollary.
\end{proof}

\subsection{Two disjoint classes of groups}\label{sec:isom}

In this section we compare the class of groups studied in this article and in the previous one \cite{Brothier20-1}.
Groups of the previous article were constructed from triples $(\Ga,\al,\varep_\Ga)$ where $\al\in\End(\Ga)$ and where $\varep_\Ga:g\in\Ga\mapsto e_\Ga$ is the endomorphism mapping all elements to the neutral element of $\Ga.$
They are all isomorphic to some twisted permutational restricted wreath products $\oplus_{\Q_2}\Ga\rtimes V$.
The comparison is motivated by the question of whether or not certain fraction groups of this present article have the Haagerup property.
Indeed, we previously proved that if $\al$ is injective and $\Ga$ has the Haagerup property (as a discrete group), then so does $\oplus_{\Q_2}\Ga\rtimes V$ \cite{Brothier19WP}.
However, if $\La$ is a group with the Haagerup property, then we do not know if $L\La\rtimes V$ has the Haagerup property not even in the untwisted case. 
We were wondering if one could embed $L\La\rtimes V$ into $\oplus_{\Q_2}\Ga\rtimes V$ in a reasonable way and for a suitable $\Ga$. 
This would allow us to deduce analytical properties of $L\La\rtimes V$ by studying $\oplus_{\Q_2}\Ga\rtimes V$.
The next theorem shows that there are no nice embeddings nor isomorphisms between the two classes.

\begin{theorem}\label{theo:noisom}
Consider two groups $\Ga,\La$, an endomorphism $\al\in\End(\Ga)$ and two \textit{injective} endomorphisms $\beta_0,\beta_1\in \End(\La)$.
Consider the fraction groups $K\rtimes V$ and $L\rtimes V$ built via the triples $(\Ga,\al,\varep_\Ga)$ and $(\La,\beta_0,\beta_1)$ respectively. 

\begin{enumerate}
\item If $K$ or $L$ is nontrivial, then there are no isomorphisms between $K\rtimes V$ and $L\rtimes V$.
\item If $L$ is nontrivial, then there are no injective $V$-equivariant morphisms from $L$ to $K$.
\item If $K$ is nontrivial, then there are no $V$-equivariant morphisms from $K$ to $L$.
\end{enumerate}
\end{theorem}

\begin{proof}
Consider $\Ga,\La,\al,\beta_0,\beta_1, K\rtimes V,L\rtimes V$ satisfying the hypothesis of the theorem.
By Proposition \ref{prop:WP} we can assume that $\al$ is an automorphism and that $K\rtimes V$ is isomorphic to a twisted permutational restricted wreath product $\oplus_{\Q_2}\Ga\rtimes V$. 

Denote by $\sigma:V\to \Aut(K),v\mapsto \sigma_v$ and $\pi:V\to\Aut(L),v\mapsto \pi_v$ the two Jones actions.

Consider $a\in K$ and let $C_V(a)$ be the subgroup of $v\in V$ commuting with $a$ inside $K\rtimes V$.
Note that $$C_V(a)=\{ v\in V:\ \sigma_v(a)=a\}, \ a\in K.$$
By definition of the Jones action we have that if $v\in V$ satisfies $v(x)=x, v'(x)=1$ for all $x\in\supp(a)$, then $v$ commutes with $a$.
We obtain that $C_V(a)$ contains the subgroup 
$$W_{\supp(a)}:=\{v\in V:\ v(x)=x, v'(x)=1,\ \forall x\in\supp(a)\}.$$
Consider now $b\in L$ and similarly define its relative commutant 
$$C_V(b):=\{v\in V:\ vb=bv\} = \{v\in V:\ \pi_v(b)=b\}.$$ 
If $v\in C_V(b)$, then by definition of the Jones action we have that $v$ stabilises the support of $b$ and thus 
$$C_V(b)\subset \Stab_V(\supp(b)).$$

We are going to show that $C_V(a)$ is much larger than $C_V(b)$ in general when $a\in K, b\in L$ implying that there are few morphisms between $K\rtimes V$ and $L\rtimes V$.

{\bf Claim: If $F\subset \Q_2$ is a finite subset, $I\subset \fC$ is a nonempty finite union of sdi and $W_F\subset \Stab_V(I)$, then $I=\fC.$}

Proof of the claim: Consider a finite subset $F\subset \Q_2$ and $I\subset \fC$ a nonempty finite union of sdi and assume that $I\neq \fC$ and $W_F\subset \Stab_V(I).$
Since $F$ is finite and $I, \fC\setminus I$ are both nonempty we can find two sdi $A,B$ satisfying that $A\cap F=B\cap F=\emptyset$ and $A\subset I, B\subset \fC\setminus I$.
There exists $v\in V$ satisfying $v(A)=B, v(B)=A$ and $v$ acts as the identity on $\fC\setminus (A\cup B).$
Observe that $v\in W_F$ but $v\notin\Stab_V(I)$, a contradiction.

Proof of (1).

Assume that $K$ or $L$ is nontrivial and that there exists an isomorphism $\theta:K\rtimes V\to L\rtimes V.$
By Proposition \ref{prop:chara}, we have that $\theta(K)=L.$
Therefore, both $K$ and $L$ are nontrivial.
Fix $a\in K$ that is nontrivial.
We obtain that $C_V(a)$ is isomorphic to $C_V(\theta(a))$.
Note that $C_V(a)$ contains $W_{\supp(a)}$ where $\supp(a)$ is finite and that $C_V(\theta(a))$ is contained in the stabiliser subgroup of $\supp(\theta(a))$ which is a finite union of sdi.
Since $a\neq e_K$ and $\theta$ is injective we have that $\theta(a)\neq e_L$ and thus its support is nonempty.
The claim implies that $\supp(\theta(a))=\fC$ and thus for all $a\neq e_K.$
This contradicts the fact that $\theta$ is surjective.

Proof of (2).

Assume there exists an injective morphism $\theta:L\to K$ that is $V$-equivariant and that $L$ is nontrivial.
There exists $b\in L$ supported on a (nonempty) sdi $I$ different from $\fC$ since $L$ is nontrivial.
Note that 
\begin{equation}\label{eq:CVtheta}
C_V(\theta(b)) = \{v\in V: \ \sigma_v(\theta(b))=\theta(b)\} = \{v\in V: \ \theta(\pi_v(b))=\theta(b)\}.
\end{equation}
Therefore, $C_V(\theta(b)) = C_V(b)$ by injectivity of $\theta.$
The argument of above implies that $W_{\supp(\theta(b))}\subset \Stab_V(I)$ which contradicts the claim since $\emptyset\neq I \neq \fC.$

Proof of (3).

Assume there exists a $V$-equivariant morphism $\theta:K\to L$ and further assume that $K$ is nontrivial.
Fix $a\in K, a\neq e_K$. 
By Equation \eqref{eq:CVtheta} we have that $C_V(a)\subset C_V(\theta(a))$.
Therefore, 
\begin{equation}\label{eq:WCV}W_{\supp(a)}\subset C_V(\theta(a)).\end{equation}

Since $\theta(a)\in L$ and $L$ is the direct limit of the $\La_t,t\in\fT$ there exists a large enough tree $t_0$ such that $\theta(a)$ is the equivalence class of a certain $(h_{t_0},t_0)\in\La_{t_0}$ with coordinates $(h_\ell)_{\ell\in \Leaf(t_0)}.$

We are going to show that $h_\ell=\beta_0(h_\ell)=\beta_1(h_\ell)$ for all $\ell\in\Leaf(t_0).$ 
Write $I_\ell$ the sdi associated to a leaf $\ell\in\Leaf(t_0)$.
Following the notation of Section \ref{sec:loop} consider the locally constant map 
$$\kappa_{t_0}(h_{t_0}):\fC\to \La, x_\ell\in I_\ell\mapsto h_\ell, \ \ell\in \Leaf(t_0).$$

Consider a thinner tree $t\geq t_0$ whose associated sdp is obtained by subdividing each $I_\ell, \ell\in\Leaf(t_0)$ into $2^n$ subintervals all of equal length for a fixed $n\geq 1.$
In terms of trees we have that $t=f\circ t_0$ where $f$ is the forest where each of its tree has $2^n$ leaves all at distance $n$ from the root.
We choose $n\geq 1$ large enough so that for each $\ell\in\Leaf(t_0)$ there exists a subinterval $J_\ell\subset I_\ell$ in the sdp of $t$ satisfying $J_\ell\cap\supp(a)=\emptyset.$
This is possible since $\supp(a)$ is finite.
Note that if $x_\ell\in J_\ell, \ell\in\Leaf(t_0)$, then 
$$\kappa_t(h_t)(x_\ell) = \beta_{m_\ell}(\kappa_{t_0}(h_{t_0})(x_\ell))=\beta_{m_\ell}(h_\ell)$$ 
for a certain finite word $m_\ell$ of length $n$.\\
Fix $\ell\in\Leaf(t_0).$
We start by proving that $\beta_{m_\ell}(h_\ell)=h_\ell$.
Note that by definition we have $m_{J_\ell}=m_{I_\ell}\cdot m_\ell$ where $m_{J_\ell}$ is the finite word associated to the sdi $J_\ell$.
Since $J_\ell\cap \supp(a)=\emptyset$ we can find $v\in V$ adapted to $J_\ell$ such that $v(J_\ell)$ is the sdi satisfying $m_{v(J_\ell)} = m_{J_\ell}\cdot m_\ell$
that is:
$$m_{v(J_\ell)} = m_{I_\ell}\cdot m_\ell\cdot m_\ell.$$
Moreover, $v$ can be chosen inside $W_{\supp(a)}$. 
Indeed, it is sufficient to considering a $v$ whose restriction to any sdi of $t$ intersecting $\supp(a)$ is the identity, defining $v$ on $J_\ell$ as explained above and then choosing any appropriate piecewise linear map on the remaining pieces of $\fC$.
Note that Equation \eqref{eq:WCV} implies that $v$ commutes with $\theta(a).$
Given any $x_\ell\in v(J_\ell)$ we have that $\kappa_{t_0}(h_{t_0})(x_\ell)=h_\ell$ and $\kappa_t(h_t)(x_\ell)= \beta_{m_\ell}(h_\ell)$. 
By definition of the directed system of groups $(\La_s,s\in\fT)$ we have that if we evaluate a representative of $\theta(a)$ at the point $x_\ell$ for any tree having the sdi $v(J_\ell)$ in its associated sdp we obtain 
$\beta_{m_\ell}\circ \beta_{m_\ell}(h_\ell).$
Now, by the definition of the Jones action, if we evaluate the representative of $\pi_v(\theta(a))$ at the point $x_\ell$ for any tree having the sdi $v(J_\ell)$ in its associated sdp we obtain 
$\beta_{m_\ell}(h_\ell).$
Since $v$ commutes with $\theta(a)$ we obtain the equality: $$\beta_{m_\ell}(h_\ell)=\beta_{m_\ell}\circ\beta_{m_\ell}(h_\ell).$$
Since $\beta_0$ and $\beta_1$ are injective so is $\beta_{m_\ell}$ and thus $h_\ell=\beta_{m_\ell}(h_\ell).$\\
We now prove that $\beta_0(h_\ell)=\beta_1(h_\ell)=h_\ell.$
Consider now $v\in W_{\supp(a)}$ adapted to $J_\ell$ such that $v(J_\ell)$ is the first half of $J_\ell.$
Fix $x_\ell\in v(J_\ell).$
The same evaluation process of representatives of $\theta(a)$ and $\pi_v(\theta(a))$ at $x_\ell$ provides that $\beta_0\circ \beta_{m_\ell}(h_\ell)=\beta_{m_\ell}(h_\ell)$.
Since $\beta_{m_\ell}(h_\ell)=h_\ell$ we obtain that $\beta_0(h_\ell)=h_\ell.$
A similar proof provides that $\beta_1(h_\ell)=h_\ell.$\\
Denote by $H:\fC\to\La$ the map equal to $\kappa_{t_0}(h_{t_0}).$
We have proved that $\beta_0(H(x))=\beta_1(H(x))=H(x)$ for all $x\in\fC$. 
This implies that $H = \kappa_s(h_s)$ for all $s\geq t_0$ where $(h_s,s)$ is the representative of $\theta(a)$ in $\La_s.$
It is easy to deduce that $\pi_v(H)(x)=H(v^{-1}x)$ for all $v\in V$ and $x\in \fC$.
If $\{A_1,\cdots,A_n\}$ is the partition of $\fC$ for which $\theta(a)$ is constant on each $A_i$ and takes different values in $A_i$ and $A_j$ if $i\neq j$, then 
$$C_{V}(\theta(a))=\{v\in V:\ v(A_i)=A_i, \ \forall 1\leq i\leq n\}.$$
A similar argument used in the claim shows that $C_V(\theta(a))$ never contains $W_{\supp(a)}$ unless $n=1$.
This implies that $\theta(a)$ commutes with the whole group $V$ and so does $a$.
Since $V$ acts transitively on $\Q_2$ this implies that $a$ is supported on $\Q_2$ or nowhere, a contradiction since $a\neq e_K$ and $a$ is finitely supported.
\end{proof}

\begin{remark}
Given a triple $(\Ga,\al_0,\al_1)$ with $\al_0,\al_1$ any endomorphisms it could happen that the associated direct limit group $K:=\varinjlim_{t\in\fT} \Ga_t$ is trivial even if $\Ga$ is not. 
Take for instance any group $\Ga$ and $\al_0=\al_1=\varep_\Ga$ the trivial endomorphism.
In fact, $K$ is trivial if and only if there exists $n\geq 1$ such that $\al_w=\varep_\Ga$ for all finite word $w\in\fs$ with more than $n$ letters.

Consider now the notations and assumptions of the last theorem. 
We obtain that $K$ is trivial if and only if $\al_0^n=\varep_\Ga$ for $n\geq 1$ large enough.
Since $\beta_0$ and $\beta_1$ are injective so is $\beta_w$ for all word $w$.
Therefore, $L$ is trivial if and only if $\La$ is trivial.
\end{remark}

\section{Automorphism groups of untwisted fraction groups}
For the whole section we fix a group $\Ga$, consider its discrete loop group $L\Ga$ of locally constant maps and the semidirect product $G=L\Ga\rtimes V$ obtained from the \textit{untwisted} Jones action:
$$\pi:V\to\Aut(L\Ga), \pi_v(a)(x):= a(v^{-1}x), \ v\in V, a\in L\Ga, v\in V.$$
Note that $G$ is the fraction group obtained from the monoidal functor:
$$\Phi:\cF\to \Gr, \ \Phi(1)=\Ga,\ \Phi(Y)(g) = (g,g) ,\ g\in \Ga$$
as explained in Section \ref{sec:loop}.

The aim of this section is to provide a clear description of the automorphism group $\Aut(G)$ of $G$.
We consider four kinds of automorphisms that we call \textit{elementary}. 
We will later show that they generate $\Aut(G).$
Moreover, we will explain how these elementary automorphisms interact with each other giving a semidirect product structure (up to a small quotient) of $\Aut(G).$

\subsection{Elementary automorphisms}\label{sec:elementary}

We now separately define four kinds of \textit{elementary} automorphisms of our fraction group $G$.

\subsubsection{Action of the normaliser of $V$ inside the homeomorphism group of the Cantor space.}
Consider the group of homeomorphisms of the Cantor space $\fC$ that normalise $V$:
$$\NCV=\{\varphi\in \Homeo(\fC):\ \varphi V\varphi^{-1} =V\}.$$
Recall from Section \ref{sec:slopes} that if $I$ is a sdi and $\varphi\in\NCV$, then $\varphi(I)$ is a finite union of sdi.
This implies that if $a\in L\Ga$, then $a^\varphi:=a\circ\varphi^{-1}$ is in $L\Ga$ defining an action of $\NCV$ on $L\Ga.$
This action extends to the semidirect product $L\Ga\rtimes V$ as follows:
$$\varphi\cdot av:= a^\varphi\cdot \ad_\varphi(v) = a^\varphi\cdot \varphi v\varphi^{-1}, \ \varphi\in\NCV, a\in L\Ga, v\in V.$$
Since the action $\NCV\act V$ is faithful (this is a well-known fact, see for instance \cite{BCMNO19}) so is the action $\NCV\act L\Ga\rtimes V.$

\begin{remark}
Note that we will sometime work with $\Qt$ rather than $\fC$ in this section. 
We will then often use the identifications of $L\Ga$ as a subgroup of $\prod_\Qt\Ga$ and as a subgroup of $\prod_\fC\Ga.$
For instance, if $a\in L\Ga, \varphi\in\NCV$ and one wants to consider only maps in $\prod_\Qt\Ga$, then the element $a^\varphi$ is understood as the restriction to $\Qt$ of the map $x\in \fC\mapsto \ti a(\varphi^{-1}(x))$ where $\ti a:\fC\to\Ga$ is the unique extension of $a$ into a locally constant map on $\fC$.
More generally, we will identify a locally constant map defined on $\Qt$ with its unique extension into a locally constant map defined on $\fC$.
\end{remark}

\subsubsection{Action of the automorphism group of $\Ga$}
Given $\beta\in\Aut(\Ga)$ we consider the diagonal automorphism 
$$\ov\beta:\prod_\Qt\Ga\to\prod_\Qt\Ga, \ \ov\beta(a)(x) = \beta(a(x)), \ a\in \prod_\Qt\Ga, x\in \Qt.$$
It is easy to see that $\ov\beta(L\Ga)=L\Ga$ and moreover this automorphism is $V$-equivariant inducing an automorphism of $L\Ga\rtimes V$.
This defines an action by automorphisms $\Aut(\Ga)\act L\Ga\rtimes V$ which is clearly faithful.

As mentioned earlier we will identify $\ov\beta$ with the locally constant map: $x\in\fC\mapsto \beta\in\Aut(\Ga).$

\subsubsection{Action of the normaliser of $G$ inside the group of all maps}

Write $\ov K:=\prod_{\Q_2}\Ga$ and the (full or unrestricted) wreath product $\ov K\rtimes V$.
Identify $L\Ga$ as a subgroup of $\ov K$ and $G$ as a subgroup of $\ov K\rtimes V.$
We write $$N_{\ov K}(G)=\{f\in \ov K :\ fGf^{-1} = G\}$$ the group of maps $f:\Q_2\to \Ga$ that normalise $G$ inside $\ov K\rtimes V.$
This clearly defines an action written $\ad$ of $N_{\ov K}(G)$ on $G$. 

Note that $\ker(\ad)$ is the set of constant maps valued in the center $Z(\Ga)$. 
We write $Z(\Ga)$ for $\ker(\ad)$ unless the context is not clear.
Indeed, if $f\in Z(\Ga)$, then 
$$\ad(f)(av) = favf^{-1} = fa(f^v)^{-1} v = fa f^{-1} v = ff^{-1} av= av$$
for all $a\in L\Ga, v\in V.$
Conversely, let $f\in N_{\ov K}(G)$ satisfying $\ad(f)=\id_G$, where $\id_G$ is the identity of $G$.
Consider $g\in\Ga$ and $\ov g$ the constant map equal to $g$ everywhere.
We have that $\ov g=\ad(f)(\ov g)= f \ov g  f^{-1}$ implying that $f$ is valued in $Z(\Ga).$
Moreover, $v=\ad(f)(v) = f(f^v)^{-1}v$ and thus $f=f^v$ for all $v\in V$ implying that $f$ is constant since $V\act \Q_2$ is transitive.
We continue to write $\ad$ the factorised action: 
$$\ad:N_{\ov K}(G)/Z(\Ga)\act G.$$
Note that $N_{\ov K}(G)$ is in general strictly larger than $L\Ga$ and can contain elements taking infinitely many values as illustrated by the following example:

\begin{example}
Consider $\Ga=\Z$ and define $f:\Q_2\to \Z, (x_n)_n\mapsto \sum_n x_n$ using the classical identification $\Q_2\subset \{0,1\}^\N$ and where $\Q_2$ is identified with the set of finitely supported sequences.
Consider $v\in V$ and an adpated sdp $(I_k , \ 1\leq k\leq n).$
By definition of $V\act \fC$ we have that $v(m_{I_k} \cdot y) = m_{v(I_k)} \cdot y$ for all $1\leq k\leq n$ and $y\in\fC$.

Observe that if $x=m_{I_k}\cdot y$ is in $I_k$ for $1\leq k\leq n$, then 
$$f(vx)f(x)^{-1} = f(m_{v(I_k)}\cdot y) f(m_{I_k}\cdot y)^{-1} = f(m_{v(I_k)})-f(m_{I_k}).$$
This implies that $f(f^v)^{-1}$ is locally constant for all $v\in V$. 
Since $\Z$ is Abelian we deduce that $f\in N_{\ov K}(G)$. 
This provides an example of an element $f\in N_{\ov K}(G)$ which takes infinitely many values.
\end{example}

\subsubsection{Action of the center of $\Ga$}

The action of the center is rather less obvious than the three other actions defined above.
Consider the map 
$$V\times \Qt\to \Z,\ (v,x)\mapsto  \ell_v(x):= \log_2(v'(v^{-1}x)).$$
Since elements of $V$ satisfy the product rule for derivation we obtain that $v\mapsto \ell_v$ satisfies the cocycle property:
$$\ell_{vw} = \ell_v + \ell_w^v,\ v,w\in V.$$
This implies that the formula $$\zeta\cdot (av):= \zeta^{\ell_v} av,\ \zeta\in Z(\Ga), a\in L\Ga, v\in V$$
defines an action 
$$F:Z(\Ga)\to \Aut(G), \ \zeta\mapsto F_\zeta: av\mapsto \zeta^{\ell_v}av, \ a\in L\Ga, v\in V.$$
This action is faithful. Indeed, consider $v\in V$ and $x\in \Q_2$ such that $vx=x$ and $v'(x)=2.$ Take for instance one of the classical generator of $F$ and $x=0$ ($x$ is the infinite sequence of $0$ in the Cantor space $\fC$ which corresponds to the usual zero of the real numbers).
We obtain that $F_\zeta(v)=\zeta^{\ell_v} v\in L\Ga$ and $\zeta^{\ell_v}(x):=\zeta^{\ell_v(x)} = \zeta$ for all $\zeta\in Z(\Ga).$
Therefore, the action is faithful.

We summarise our study in the proposition below.

\begin{proposition}\label{prop:elemaction}
The following formulas define faithful actions by automorphisms on $G$:
\begin{align*}
A&:\NCV\times \Aut(\Ga)\to \Aut(G),\ A_{\varphi,\beta} (av):= \ov\beta(a)^\varphi \cdot \ad_\varphi(v);\\
\ad &: N_{\ov K}(G)/Z(\Ga)\to \Aut(G), \ \ad(f)(av) = favf^{-1} = fa(f^v)^{-1} \cdot v;\\
F&:Z(\Ga)\to \Aut(G), \ F_\zeta(av) : = \zeta^{\ell_v} \cdot a \cdot v,
\end{align*}
where 
$$\ov\beta(a)(x):= \beta(a(x)), \ a^\varphi(x):=a(\varphi^{-1}x), \ \ad_\varphi(v) := \varphi v \varphi^{-1},$$
$$ \ell_v(x):=\log_2(v'(v^{-1}x)) \text{ and } \zeta^{\ell_v}:\Q_2\to Z(\Ga), x\mapsto \zeta^{\ell_v(x)}$$ for $$a\in L\Ga, v\in V, x\in\Q_2, \beta\in\Aut(\Ga), \varphi\in\NCV, f\in N_{\ov K}(G)/Z(\Ga), \zeta\in Z(\Ga).$$

Moreover, the actions $\ad$ and $F$ mutually commute.
\end{proposition}

\begin{proof}
There is nothing to prove except that the actions of $\NCV$ and $\Aut(\Ga)$ (resp. $N_{\ov K}(G)/Z(\Ga)$ and $Z(\Ga)$) mutually commute which is rather obvious.
\end{proof}

\subsection{Decomposition of the automorphism group}\label{sec:decompoAut}

Before proving the main theorem we define a useful function.
Consider $\varphi\in\NCV$ and the inverse image $\varphi^{-1}(\Q_2)$. Note that $\Q_2$ is by definition the set of all sequences of $\fC$ that are eventually equal to $0$. It is an orbit under the action of $V$ and by Rubin theorem we have that $\varphi^{-1}(\Q_2)$ is also such an orbit. There exists a prime word $c$ in $0,1$ satisfying that $x\in \varphi^{-1}(\Q_2)$ if and only if $x = y \cdot c^\infty$ for a certain word $y$.
We write $k_\varphi:=|c|$ the number of letters in the word $c$.

\begin{lemma}\label{lem:ga}
Consider $\varphi\in\NCV$.
There exists a unique map $$\ga_\varphi:\Q_2\to\Z, x\mapsto \ga_\varphi(x)$$
satisfying that if $v\in V$, then 
$$\ga_\varphi(v0)=\log_2((\varphi^{-1}v\varphi)'(\varphi^{-1}0)) - k_\varphi\cdot \log_2(v'(0)).$$
In particular, if $v,w\in V$ and $v0=w0$, then 
$$\log_2((\varphi^{-1}v\varphi)'(\varphi^{-1}0)) - k_\varphi\cdot \log_2(v'(0))=\log_2((\varphi^{-1}w\varphi)'(\varphi^{-1}0)) - k_\varphi\cdot \log_2(w'(0)).$$
Moreover, we have the equality:
$$\ga_\varphi(vx)-\ga_\varphi(x) = \log_2((\varphi^{-1}v\varphi)'(\varphi^{-1}x)) - k_\varphi\cdot \log_2(v'(x))$$ 
for all $\varphi\in\NCV, v\in V, x\in\Q_2.$
\end{lemma}

Note that here we use the symbol $0$ to denote two different objects.
We are using it as a letter when we write elements of $\fC$ as infinite sequences of $0$ and $1$.
Moreover, we are using this symbol to denote the usual zero of the real numbers corresponding to the infinite sequence of $0$ in $\fC$.

\begin{proof}
Consider $\varphi\in\NCV, v,w\in V$ and $x\in \Q_2$ satisfying that $vx=wx.$
Let $c$ be a prime word in $0,1$ satisfying that $\varphi^{-1}(\Q_2)$ is the class of $\fC/V$ containing $c^\infty$ and write $k_\varphi=|c|.$
Put $u:=w^{-1}v$ and observe that $ux=x.$
By Proposition \ref{prop:slope} we have that 
$$\log_2[(\varphi^{-1}u\varphi)'(\varphi^{-1}x)] = k_\varphi \cdot \log_2(u'(x)).$$
Recall that the chain rule applies to elements of $V$.
Therefore, 
$$v'(x) = (wu)'(x) = w'(ux)\cdot u'(x)= w'(x)\cdot u'(x).$$
This implies that
$$\log_2(v'(x))-\log_2(w'(x))=\log_2(u'(x)).$$
We deduce the following:
\begin{align*}
\log_2[(\varphi^{-1}v\varphi)'(\varphi^{-1}x)]  & = \log_2[(\varphi^{-1}wu\varphi)'(\varphi^{-1}x)] \\
& = \log_2[(\varphi^{-1}w\varphi)'(\varphi^{-1}x)]+\log_2[(\varphi^{-1}u\varphi)'(\varphi^{-1}x)] \\
& = \log_2[ (\varphi^{-1}w\varphi)'(\varphi^{-1}x)] +  k_\varphi \cdot \log_2(u'(x))\\
& =  \log_2[ (\varphi^{-1}w\varphi)'(\varphi^{-1}x)] +  k_\varphi ( \log_2(v'(x))-\log_2(w'(x)) ).
\end{align*}
We obtain the equality:
$$\log_2((\varphi^{-1}v\varphi)'(\varphi^{-1}x)) - k_\varphi\cdot \log_2(v'(x)) = \log_2((\varphi^{-1}w\varphi)'(\varphi^{-1}x)) - k_\varphi\cdot \log_2(w'(x)).$$
This proves the first statement when applied to $x=0.$

Consider now the map $\varphi\in\NCV$, $v\in V$ and $x\in\Q_2.$
There exists $w\in V$ satisfying $w0=x$ since $V\act\Q_2$ is transitive.
We have that 
\begin{align*}
\ga_{\varphi}(vx) & = \log_2((\varphi^{-1}vw\varphi)'(\varphi^{-1}0)) - k_\varphi\cdot \log_2((vw)'(0)) \\
& = \log_2((\varphi^{-1}v\varphi)'(\varphi^{-1}x)) + \log_2((\varphi^{-1}w\varphi)'(\varphi^{-1}0)) - k_\varphi\cdot \log_2(v'(x)) - k_\varphi\cdot \log_2(w'(0)) \\
& = \ga_\varphi(x) +  \log_2((\varphi^{-1}v\varphi)'(\varphi^{-1}x)) - k_\varphi\cdot \log_2(v'(x))
\end{align*}
proving the second statement.
\end{proof}

We are now ready to prove the main theorem of this section.

\begin{theorem}\label{theo:Aut(G)loop}
Let $\Ga$ be a group and $G:=L\Ga\rtimes V$ the associated untwisted fraction group. 
The formula
$$(\varphi,\beta)\cdot (\zeta, f):=(\beta(\zeta)^{k_\varphi} , \ov\beta(f)^\varphi\cdot \zeta^{\ga_\varphi})$$
defines an action by automorphisms of $\NCV\times\Aut(\Ga)$ on $Z(\Ga)\times N_{\ov K}(G)/Z(\Ga),$
where  $$\varphi\in\NCV, \beta\in\Aut(\Ga), \zeta\in Z(\Ga), f\in N_{\ov K}(G)/Z(\Ga).$$
Write $( Z(\Ga)\times N_{\ov K}(G)/Z(\Ga)) \rtimes (\NCV\times \Aut(\Ga))$ for the semidirect product constructed from the action of above.

The following map
\begin{align*}
\Xi: & ( Z(\Ga)\times N_{\ov K}(G)/Z(\Ga)) \rtimes (\NCV\times \Aut(\Ga))\to \Aut(G)\\
& (\zeta , f , \varphi , \beta)\mapsto F_\zeta\circ\ad(f)\circ A_{\varphi,\beta} 
\end{align*}
is a surjective group morphism with kernel
$$\ker(\Xi)=\{ (e_{Z(\Ga)} , \ov g , \id_V , \ad(g^{-1}) ):\ g\in \Ga\}$$
where $\ov g:\Q_2\to\Ga$ is the constant map equal to $g$ everywhere identified with an element of $N_{\ov K}(G)/Z(\Ga)$ and where $\ad(g^{-1})$ stands for the inner automorphism $h\mapsto g^{-1}hg$ belonging to $\Aut(\Ga).$
\end{theorem}

\begin{proof}
Let $\Ga$ be a group and let $G:= L\Ga\rtimes V$ be the (untwisted) fraction group associated to it. 
We start by proving that every automorphism of $G$ is the composition of some elementary ones as defined in Section \ref{sec:elementary}.

Fix $\theta\in \Aut(G)$.
Proposition \ref{prop:chara} implies that $\theta(L\Ga) = L\Ga$.
Therefore, there exists $\kappa\in\Aut(L\Ga), \varphi\in \NCV$ and a cocycle $c:V\to L\Ga$ satisfying:
$$\theta(av) = \kappa(a) \cdot c_v \cdot \ad_\varphi(v) \text{ for all } a\in L\Ga, v\in V.$$
Since $av\mapsto a^\varphi\cdot \ad_\varphi(v)$ is an elementary automorphism, up to composing with it, we can assume that $\varphi(x)=x$ for all $x\in\fC$.
We obtain that $\theta(av) = \kappa(a) \cdot c_v \cdot v$ for all $a\in L\Ga, v\in V.$

Note that $\theta$ is a spatial automorphism since $\Ga=\Ga^\al$ by Remark \ref{rem:support}. Therefore, the map $\zeta:G\to Z(G)$ of Theorem \ref{theo:support} is necessarily trivial.
Theorem \ref{theo:support} implies that $\kappa$ splits as a product of automorphisms $\prod_{x\in\fC}\kappa_x$ such that 
$$\kappa(a)(x) = \kappa_x(a(x)) \text{ for all } a\in L\Ga, x\in \fC$$
and where $\kappa_x\in \Aut(\Ga)$ for all $x\in \fC.$
Moreover, we have the following useful equation:
\begin{equation}\label{eq:kappacv}
\kappa_{vx} = \ad(c_v(vx))\circ \kappa_x \text{ for all } v\in V, x\in \fC.
\end{equation}

{\bf Claim: } The map $x\in\fC\mapsto \kappa_x\in\Aut(\Ga)$ is locally constant, i.e.~this map belongs to $L\Aut(\Ga).$

We will prove the claim using a compactness argument.
We start by observing that $c_v(x)\in Z(\Ga)$ if $x\in\fC,v\in V$ and $vx=x.$
Indeed, Equation \ref{eq:kappacv} implies that $\kappa_x = \ad(c_v(x)) \circ \kappa_x$ giving that $\ad(c_v(x))=\id_\Ga$. 
Therefore, $c_v(x)\in Z(\Ga).$
Consider now $x\in \Q_2$, $v\in V$ and $I$ a sdi starting at $x$ that is adapted to $v$ such that $v(I)$ is the first half of $I$.
Such a triple exists since $x\in \Q_2.$
Note that if $J\subset I$ is a sdi starting at $x$, then we have that $v$ is adapted to $J$ and $v(J)$ is the first half of $J$.
Since $c_v$ is locally constant, up to reducing $I$ to a smaller sdi starting at $x$, we can assume that $c_v$ is constant on $I$.
Since $vx=x$ we have that $c_v(x)\in Z(\Ga)$ and thus $c_v(y)=c_v(x)\in Z(\Ga)$ for all $y\in I$. 
Equation \ref{eq:kappacv} implies that 
\begin{equation}\label{eq:kappa}
\kappa_{vy} = \kappa_y \text{ for all } y \in I.
\end{equation}
Consider now a fixed $g\in \Ga$ and a fixed $y\in I$.
The function $\kappa(\ov g)$ is locally constant and thus there exits a sdi $I_g\subset I$ starting at $x$ on which $\kappa(\ov g)$ is constant.
Note that since $y\in I$ we have that the sequence $(v^n(y))_{n\geq 1}$ is contained in $I$ and is converging to $x$.
Hence, for $n$ large enough we have that $v^n(y)\in I_g$ implying that $\kappa_x(g) = \kappa_{v^n(y)}(g).$
Equation \eqref{eq:kappa} implies that $\kappa_{v^n(y)}(g)=\kappa_y(g)$.
We deduce that $\kappa_x(g) = \kappa_y(g).$
Since this is true for any $g\in \Ga$ and $y\in I$ we obtain that $y\mapsto \kappa_y$ is constant on $I$.
We have proved that for any $x\in \Q_2$ there exists a sdi $I_x$ starting at $x$ on which $y\mapsto \kappa_y$ is constant. 

Consider now the set $X\subset \fC$ of all $x\in\fC$ satisfying that there exists an open set $O_x$ containing $x$ such that $\kappa_x=\kappa_y$ for all $y\in O_x.$ 
We are going to show that $X=\fC$.
Note that by definition $X$ is open.
We have proved that $\Q_2\subset X$ implying that $X$ is nonempty.
Let us show that $X$ is $V$-invariant.
Consider $x\in X$ and $v\in V$.
By definition of $X$ there exists an open set $O_x$ containing $x$ such that $\kappa_y=\kappa_x$ for all $y\in O_x$.
The function $c_v(v\cdot)$ is locally constant implying that there exists a sdi $J_x$ containing $x$ such that $c_v(vy)=c_v(vx)$ for all $y\in J_x.$
The intersection $O:=O_x\cap J_x$ is open and contains $x$.
Moreover, if $y\in O$ we have by Equation \ref{eq:kappacv} the following:
$$\kappa_{vy} = \ad(c_v(vy))\circ \kappa_y = \ad(c_v(vx))\circ \kappa_y= \ad(c_v(vx))\circ \kappa_x = \kappa_{vx}.$$
Therefore, $z\mapsto \kappa_z$ is constant on $v(O)$ which is open and contains $vx$. 
We have proved that $vx\in X$ and thus $X$ is $V$-invariant.
Any $V$-orbit is dense in $\fC$ and thus intersects the nonempty open set $X$.
This proves that $X$ contains all the $V$-orbits and thus $X=\fC.$
Given $x\in \fC$ we write $I_x$ a sdi containing $x$ on which $z\mapsto \kappa_z$ is constant.
We have that $(I_x:\ x\in \fC)$ is an open covering of $\fC$ and by compactness there exists a finite subcovering. 
This implies that $y\mapsto\kappa_y$ is locally constant and proves the claim.

Via the claim we are now able to remove $\kappa$.
Equation \ref{eq:kappacv} implies that $\kappa_x=\kappa_y \mod \Inn(\Ga)$ if $Vx=Vy$ and $x,y\in\fC.$
Since $x\mapsto \kappa_x$ is locally constant and any orbit of $V$ is dense in $\fC$ we deduce that $\kappa_x=\kappa_y \mod \Inn(\Ga)$ for all $x,y\in \fC$.
Recall that given any $\beta\in\Aut(\Ga)$ we can produce an elementary automorphism of $G$ via the formula
$$av\mapsto \ov\beta(a)v, \text{ where } \ov\beta(a)(x):= \beta(a(x)), a\in L\Ga, v\in V, x\in\fC.$$
So up to composing with such an automorphism we can now assume that $\kappa_x\in \Inn(\Ga)$ for all $x\in \fC$.
We obtain that $\kappa=\ad(f)$ for a certain map $f:\fC\to\Ga.$ 
Moreover, $x\in\fC\mapsto \ad(f(x))\in\Inn(\Ga)$ is locally constant.
Therefore, $f=f_0\cdot f_1$ where $f_0\in L\Ga$ and $f_1\in \prod_\fC Z(\Ga)$.
We have that $\kappa=\ad(f_0\cdot f_1)=\ad(f_0)$ with $f_0\in L\Ga\subset N_{\ov K}(G)$.
Up to composing with the elementary automorphism $\ad(f_0)$, we can now assume that $\kappa=\id_{L\Ga}$ and thus 
$$\theta(av) = a \cdot c_v \cdot v \text{ for all } a\in L\Ga, v\in V.$$

It remains to remove the cocycle part of the automorphism.
We start by observing that $c_v(x)$ is in the center of $\Ga$ for all $v\in V, x\in\fC.$
Indeed, given $a\in L\Ga, v\in V$ we have
$$a^v=\kappa(a^v) = \theta(vav^{-1}) = \theta(v) \kappa(a) \theta(v)^{-1} = c_v \cdot v \cdot a \cdot v^{-1} \cdot c_v^{-1} = c_v\cdot a^v \cdot c_v^{-1}.$$
Therefore, $c_v$ commutes with any $a^v$ for $a\in L\Ga$ that is $c_v$ is in the center of $L\Ga$ for all $v\in V$.
Hence, $c_v(x)\in Z(\Ga)$ for all $x\in\fC, v\in V.$
We apply Proposition \ref{prop:cocycle} on cocycles valued in an Abelian group:
there exists $\zeta\in Z(\Ga)$ and $f:\Q_2\to Z(\Ga)$ satisfying 
$$c_v(x) = \zeta^{\log_2(v'(v^{-1}x))} \cdot f(x) f(v^{-1}x)^{-1},\ v\in V, x\in \Q_2.$$
We recognise that $\theta = F_\zeta\circ \ad(f)$ with $f$ necessarily in $N_{\ov K}(G).$
We have proved that $\Aut(G)$ is generated by the four kinds of automorphisms presented in Section \ref{sec:elementary}.

We now prove that the formula given in the theorem defines an action and thus a semidirect product.
In order to manipulate lighter notations we write $A_\varphi$ for $A_{\varphi,\id_\Ga}$ and $A_\beta$ for $A_{\id_V,\beta}$ where $\varphi\in\NCV$ and $\beta\in\Aut(\Ga).$
It is rather easy to see that the actions of $\NCV$ and $\Aut(\Ga)$ mutually commute and so do the actions of $Z(\Ga)$ and $N_{\ov K}(G).$
A routine computation shows that 
$$A_{\varphi,\beta} \ad(f) A_{\varphi,\beta}^{-1} = \ad(\ov\beta(f)^\varphi), \varphi\in\NCV,\beta\in\Aut(\Ga), f\in N_{\ov K}(G)$$ and that 
$$A_\beta F_\zeta A_\beta^{-1} = F_{\beta(\zeta)}, \zeta\in Z(\Ga).$$
The delicate point of the formula is the computation of $A_\varphi F_\zeta A_\varphi^{-1}.$
Fix $\varphi\in\NCV,\zeta\in Z(\Ga), a\in L\Ga, v\in V$ and observe that
$$(A_\varphi F_\zeta A_\varphi^{-1})(av)  = (A_\varphi F_\zeta) ( a\circ \varphi \cdot \varphi^{-1}v\varphi) = A_\varphi (\zeta^{\ell_{\varphi^{-1}v\varphi}} \cdot a\circ \varphi \cdot \varphi^{-1}v\varphi) = \zeta^{\ell_{\varphi^{-1}v\varphi}^\varphi} \cdot a v.$$
Note that $v\in V\mapsto \zeta^{\ell_{\varphi^{-1}v\varphi}^\varphi}$ is necessarily a cocycle since $A_\varphi F_\zeta A_\varphi^{-1}$ is multiplicative.
Moreover, it is valued in the Abelian group $\langle \zeta\rangle$ generated by $\zeta$ inside $\Ga.$
Proposition \ref{prop:cocycle} implies that there exists a pair $(\xi, h)\in \langle \zeta\rangle\times \prod_{\Q_2}\langle \zeta\rangle$ satisfying that $A_\varphi F_\zeta A_\varphi^{-1} = F_\xi\circ \ad(h)$ and moreover this pair is unique up to multiplying $h$ by a centrally valued constant map.
Note that since $A_\varphi F_\zeta A_\varphi^{-1}$ and $F_\xi$ are automorphisms of $G$ we necessarily have that $h$ normalises $G$.
All together we obtain that $\NCV\times \Aut(\Ga)$ normalises $Z(\Ga)\times N_{\ov K}(G)/Z(\Ga)$.
Consider the map $\ga_\varphi$ defined in Lemma \ref{lem:ga} and recall that $$\ga_\varphi(vx)-\ga_\varphi(x) = \log_2((\varphi^{-1}v\varphi)'(\varphi^{-1}x)- k_\varphi\cdot \log_2(v'(x)),v\in V,x\in\Q_2.$$
Therefore, $\ell_{\varphi^{-1}v\varphi}^\varphi - k_\varphi\cdot \ell_v = \ga_\varphi-\ga_\varphi^v$ and thus $$A_\varphi F_\zeta A_\varphi^{-1} = F_{\zeta^{k_\varphi}} \circ \ad(\zeta^{\ga_\varphi}).$$
This proves that the formula 
$$
(\varphi,\beta)\cdot (\zeta, f):=(\beta(\zeta)^{k_\varphi} , \ov\beta(f)^\varphi\cdot \zeta^{\ga_\varphi})$$
defines an action. 

We have proved that $\Xi$ is a well-defined surjective group morphism.
To conclude it remains to check that $\ker(\Xi)$ is equal to 
$$M:=\{(e_{Z(\Ga)},\ov g, \id_V,\ad(g^{-1})):\ g\in \Ga\}.$$
It is clear that $M$ is a subgroup of $\ker(\Xi).$
Consider $\theta=(\zeta,f,\varphi,\beta)\in \ker(\Xi).$
Observe that for all $v\in V$ we have:
$$v=\Xi(\theta)(v) = h \cdot \varphi v\varphi^{-1}$$
for a certain $h\in L\Ga.$
This implies that $\ad_\varphi(v)=v$ and thus $\varphi$ is trivial since $\NCV\act V$ is faithful.
Now, $e_G=\Xi(\theta)(v)v^{-1} = \zeta^{\ell_v} f(f^v)^{-1},v\in V$ and evaluating this function at $x\in\Q_2$ for a fixed $v\in V$ satisfying $vx=x$ and $v'(x)=2$ we obtain that $\zeta=e_{Z(\Ga)}.$
Consider $k\in \Ga$, we have $\ov k =\Xi(\theta)(\ov k) = f \ov{\beta(k)} f^{-1}.$
Therefore, $\beta= \ad(f(x)^{-1})$ for any choice of $x\in\Q_2$ implying that the map $x\in\Q_2\mapsto \ad(f(x))^{-1}\in\Inn(\Ga)$ is constant equal to $\ad(g^{-1})$ for a certain $g\in \Ga$. 
We obtain that $\theta= (e_{Z(\Ga)},\ov g,\id_V,\ad(g^{-1}))\in M$ and thus $\ker(\Xi)=M.$
\end{proof}

\appendix
\section{Comparison of Jones technology with cloning systems of Witzel-Zaremsky and with Tanushevski's construction}

\subsection{Cloning systems of Witzel-Zaremsky}

\subsection*{Zappa-Szep products and cloning systems}
We present here a different method for producing Thompson-like groups.
This method is based on Zappa-Szep products (or bicrossed products) of monoids. 
Brin constructed a braided version of the Thompson group known today at the braided Thompson group $BV$ (constructed independently by Dehornoy \cite{Dehornoy06}) using a bicrossed product of the monoid of forests and the infinite strand braid groups \cite{Brin07-BraidedThompson,Brin06-BV2}.
Witzel and Zaremsky provided an abstract reformulation of this construction. 
They obtained a very convenient formalism for constructing new groups having similar properties to Thompson groups \cite{Witzel-Zaremsky18,Zaremsky18-clone}.
This technology is very well adapted to study finiteness properties and indeed they constructed new interesting families of groups of type $F_n$ (i.e.~groups for which there exists a classifying space with finite $n$-squeleton) for $n\geq 1$. 
This is truly remarkable as this is a very difficult question to decide if a group satisfies any finiteness properties. 
Most of the usual operations on groups destroys finiteness properties. 
However, they discover that their machinery in many cases preserves such properties.

\subsubsection{Zappa-Szep products of a forest with a group.}

Consider $\cF_\infty$ the set of all forest having infinitely many roots and leaves indexed by the natural number such that all but finitely many of the trees are nontrivial. 
We equip this set with the operation of vertical concatenation making it a monoid. 
Consider now a group $\Ga_\infty$ and a Zappa-Szep product or bicrossed product $\cF_\infty\bowtie \Ga_\infty$ between these two monoids. 
This is a generalisation of the semidirect product where $\cF_\infty$ and $\Ga_\infty$ simultaneously act on each other.
Under mild conditions this monoid admits a calculus of fractions (it satisfies the Ore property and is cancellative) and thus provides a fraction group $\Frac(\cF_\infty\bowtie\Ga_\infty)$.
In that case we say that the bicrossed product $\cF_\infty\bowtie\Ga_\infty$ is a Brin-Zappa-Szep product.
Elements of $\Frac(\cF_\infty\bowtie\Ga_\infty)$ are (equivalence classes) of pairs of forests decorated by elements of $\Ga_\infty.$
Consider the subgroup $G(\cF_\infty\bowtie\Ga_\infty)$ made of pairs of trees rather than pairs of forests. 
This is the construction of Brin for obtaining the braided Thompson group $BV$. 
The group $\Ga_\infty$ is in this example the group $B_\infty$ of finitely supported braids (elements are braids with infinitely many strands indexed by the natural number and with only finitely many crossings). 
The Zappa-Szep product $\cF_\infty\bowtie B_\infty$ consists in ``braiding the forests'': we obtain a monoid where an element of it consists of a forest with on top of it a braid. The Zappa-Szep product explains how to compose these ``braided forests." 

\subsubsection*{Witzel-Zaremsky's cloning systems.}

Witzel and Zaremsky have defined a general theory for producing and studying the class of groups $G(\cF_\infty\bowtie\Ga_\infty)$ of above \cite{Witzel-Zaremsky18,Zaremsky18-clone}
Their theory can be interpreted as a general axiomatisation of Brin-Zappa-Szep products of $\cF_\infty$ with a group $\Ga_\infty$. 
To do this they define a so-called \textit{cloning system on a family of groups} that we now present.
Denote by $\cS=(S_n, j_{m,n}:S_n\to S_m:\ 1\leq m\leq n)$ the directed system of symmetric groups $S_n$ of $\{1,\cdots,n\}$ where $j_{m,n}$ is induced by the inclusions $\{1,\cdots,m\}\subset \{1,\cdots,n\}$ for $1\leq m\leq n.$
A cloning system is the following set of data:
\begin{enumerate}
\item a directed system of groups $$\cG:=(\Ga_n, \iota_{m,n}:\Ga_m\to \Ga_n:\ 1\leq m\leq n)$$
where the morphisms $\iota_{m,n}, 1\leq m\leq n$ are all injective;
\item a sequence of group morphisms $\cR=(\rho_n:\Ga_n\to S_n, n\geq 1)$ compatible with the directed systems $\cG$ and $\cS$;
\item a collection of so-called cloning maps $\cK=(\kappa_k^n:\Ga_n\to \Ga_{n+1},  1\leq k\leq n)$ that are injective maps (not necessarily group morphisms) compatible with the directed system $\cG$;
\item three conditions 
\begin{enumerate}
\item C1 (cloning a product) about cloning maps applied to a product of elements;
\item C2 (product of clonings) about compositions of cloning maps;
\item C3 (compatibility) about compatibilities between the maps of $\cR$ and the cloning maps.
\end{enumerate}
\end{enumerate}
Given such a cloning system one can form a Brin-Zappa-Szep product $\cF_\infty\bowtie \Ga_\infty$ where $\Ga_\infty$ is the direct limit group of the system $\cG$ that is: $\Ga_\infty=\varinjlim_n \Ga_n$.
Now we consider the fraction group $\Frac(\cF_\infty\bowtie \Ga_\infty)$ of the monoid $\cF_\infty\bowtie \Ga_\infty$ which consists in pairs of labelled \textit{forests}.
Finally, we consider $G(\cG,\cR,\cK)$ the subgroup of $\Frac(\cF_\infty\bowtie \Ga_\infty)$ which consists of pairs of labelled \textit{trees}.
This is the group considered by Witzel and Zaremsky.

\subsection*{Comparison between cloning systems and Jones technology}

We will now compare the two constructions of groups: cloning systems and Jones technology.
Recall that the initial motivations to create each framework are very different. 
Cloning systems were introduced to build groups that resemble Thompson group and better understand behavior of properties of groups such as finiteness properties.
Jones technology was introduced to build conformal field theories from subfactors. 
Jones technology is a machine for producing actions of groups (in fact actions of groupoids). 
The author is using Jones technology for producing groups that have remarkable descriptions as fraction groups.

Consider the category of forests $\cF$ and a functor $\Phi:\cF\to\cD$ where $\cD$ is a category.
The category $\cF$ admits a calculus of fractions and thus a fraction groupoid $\cG_\cF$ which consists of equivalence classes of pairs of forests $(f,g)$ so that $f$ and $g$ have the same number of leaves. 
Note that if we only consider pairs of \textit{trees} with same number of leaves we obtain Thompson group $F$. 
More generally,  if we consider pairs of forests with a fixed number of roots $r$ and the same number of leaves we obtain the Brown-Higmann-Thompson group $F_{2,r}$ (where $2$ refers to \textit{binary}).
Let us call $\cG_\cF$ the Thompson groupoid.
Jones technology constructs an action of the Thompson groupoid from any functor $\Phi:\cF\to\cD$.
By considering the restriction of this action to Thompson group $F$ we obtain what we have called in this article a Jones action.

\subsubsection*{First comparison: group actions and bicrossed-products}
We can already see two major differences between the two technologies.
On one side we construct a bicrossed product between the monoid of forests $\cF_\infty$ and a group $\Ga_\infty$.
On the other side we construct an action of the Thompson groupoid $\cG_\cF$. 
First, we have mutual actions of $\cF_\infty$ and $\Ga_\infty$ on each other rather than one groupoid acting on an object.
This is well illustrated by the construction of Brin of the braided Thompson group $BV$ where braids and forests are mutually acting on each other. 
Second, there are two limit processes in the Witzel-Zaremsky framework: a direct limit of groups $\Ga_\infty=\varinjlim\Ga_n$ and a monoid $\cF_\infty$ which can be interpret as a limit of the morphism spaces of the category of forests: $\cF_\infty=\varinjlim\cF(n,m).$
In Jones framework there is one limit process for constructing an action of Thompson group $F$ that is different from the one of above: a direct limit of a family of objects indexed by trees (or indexed by forests for an action of the Thompson groupoid). 
We will later come back more deeply on this.

We now specialize to functors $\Phi:\cF\to\Gr$ going to the category of groups.
Jones technology provides an action of the Thompson groupoid.
The restriction of this action to the Thompson group is an action $F\act K$ where $K$ is a group obtained by taking a limit of groups over the set of trees. Moreover, $F$ acts by group automorphisms and thus we can consider the semidirect product $K\rtimes F$. 
Consider now a group $G=G(\cG,\cR,\cK)$ constructed from a cloning system $(\cG,\cR,\cK)$.
{\bf We want to compare the groups $K\rtimes F$ and $G(\cG,\cR,\cK)$ and how they have been built.}

\subsubsection*{Second comparison: direct/inverse limit, semidirect product structure}
Note that here, if the functor $\Phi:\cF\to\Gr$ is covariant, then $K$ is a direct limit, while when the functor is contravariant, then we have that $K$ is an inverse limit. 
We will see that the construction using a cloning system shares some similarities with the construction of $K\rtimes F$ when the functor $\Phi:\cF\to\Gr$ is \textit{covariant}. 
However, we don't see any resemblance between the construction of a group using a cloning system and the construction of a group using a \textit{contravariant} functor.
Jones technology provides a semidirect product $K\rtimes F$ via the Jones action $F\act K.$
However, the group $G(\cG,\cR,\cK)$ obtained from a cloning system does not decompose into such a semidirect product in general, e.g.~the braided Thompson group $BV$.

\subsection*{Covariant functors and pure cloning systems}

\subsubsection*{Certain covariant functors provide pure cloning systems}

Consider now a covariant functor $\Phi:\cF\to\Gr$. 
Let us describe what is the amount of information that such a functor encodes. 
We will then relate this information with a cloning system.
For each $n\geq 1$ we have a group $\Phi(n).$
For each forest $f\in\cF(n,m), 1\leq n\leq m$ we have a group morphism $\Phi(f):\Phi(n)\to\Phi(m)$ so that $\Phi(f\circ g)=\Phi(f)\circ\Phi(g)$ if $g\in\cF(k,n), 1\leq k\leq n.$
We will now use the universal property of the category of forests for considering generators of the forests rather than all forests.
The forests are generated by the elementary ones: $f_{k,n}, 1\leq k\leq n$ the forests with $n$ roots, $n+1$ leaves with all its tree trivial except the $k$th one that has two leaves. 
These elementary forests satisfy the following relations:
$$f_{k,n+1}\circ f_{j,n} = f_{j+1, n+1}\circ f_{k,n}, 1\leq k<j\leq n.$$
The category $\cF$ is the universal category presented by these generators and relations.
Therefore, defining a covariant functor $\Phi:\cF\to\Gr$ is equivalent to defining a family of groups $(\Phi(n),\ n\geq 1)$ and a family of group morphism $(\Phi(f_{k,n}),\ 1\leq k\leq n)$ satisfying the relations
\begin{equation}\label{eq:Phi(f)}\Phi(f_{k,n+1})\circ \Phi(f_{j,n}) = \Phi(f_{j+1, n+1})\circ \Phi(f_{k,n}) , 1\leq k<j\leq n.\end{equation}
This is very reminiscent of a cloning system if we write $\Ga_n:=\Phi(n)$ and $\kappa_k^n:=\Phi(f_{k,n})$ for $1\leq k\leq n$.
A main difference being that in Jones framework, all the maps $\Phi(f_{k,n}), 1\leq k\leq n$ must be group morphisms (not necessarily injective) while they are any injective maps for cloning systems.
Assume that $\Phi(f_{k,n})$ is injective for all $1\leq k\leq n$ which is equivalent to say that $\Phi(f)$ is injective for all forest $f$.
In the cloning system framework there is a directed system $\cG$ associated to the groups $\Phi(n),n\geq 1$. In Jones framework there is no such thing in general.
Define $\rho_n:\Ga_n\to S_n$ to be the trivial map for all $n\geq 1$.
One can check that $(\Phi(n),n\geq 1)$ together with a directed structure, $\cK:=(\Phi(f_{k,n}), 1\leq k\leq n)$ and $(\rho_n,n\geq 1)$ defines a cloning system. 
This cloning system is called \textit{pure} because the maps $\rho_n, n\geq 1$ are all trivial.
Moreover, the group $G$ associated to this cloning system is isomorphic to the fraction group $K\rtimes F$ obtained from the functor $\Phi$.

\subsubsection*{Pure cloning systems provide covariant functors and more}

Conversely, consider a pure cloning system $(\cG,\cR,\cK)$.
Let us construct a (covariant) functor $\Phi:\cF\to\Gr.$
Condition C1 (see the axioms of a cloning system written above) is equivalent to have that $\cK$ is a collection of group morphisms. 
Condition C2 is equivalent to Equation \ref{eq:Phi(f)}. 
Therefore, the universality property of $\cF$ implies that there exists a unique covariant functor $\Phi:\cF\to\Gr$ satisfying
$$\Phi(n)=\Ga_n, \Phi(f_{k,n})=\kappa_k^n \text{ for all } 1\leq k\leq n.$$
Note that since the cloning maps are injective by assumption so are the morphisms $\Phi(f_{k,n})$ for all $1\leq k\leq n$ implying that $\Phi(f)$ is injective for all forest $f$.
Condition 3 is empty in this situation.
Consider now the Jones action $\pi:F\act K$ associated to the functor $\Phi$.
Note that since the functor is not necessarily monoidal we cannot extend in general the Jones action to an action of $V$.
We obtain a semidirect product $K\rtimes F$.
One can show that $K\rtimes F$ is isomorphic to the group $G(\cG,\cR,\cK)$ obtained from the cloning system $(\cG,\cR,\cK).$

In fact we can do more.
Note that the directed system $\cG$ produces a limit group $\Ga_\infty:=\varinjlim_n\Ga_n$ and write $\iota_n:\Ga_n\to\Ga_\infty, n\geq 1$ for the inclusion maps. 
Consider the monoid of forests $\cF_\infty$ and write $f_k\in \cF_\infty$ for the elementary forest having all its tree trivial except the $k$th one that has two leaves $k\geq 1.$
There is an obvious functor $j:\cF\to \cF_\infty$ that we now define.
The functor $j$ sends all object $n\geq 1$ of $\cF$ to the unique object of $\cF_\infty$ that we denote by $\infty$. 
The functor $j$ sends all finite forests $f$ of $\cF$ to the finitely supported forest $j(f)$ consisting of $f$ followed to its right by infinitely many trivial trees, i.e.~$j(f)=f\ot I^{\ot \infty}.$
Alternatively, using generators we can define $j:\cF\to\cF_\infty$ as the unique functor satisfying $j(n)=\infty$ and $j(f_{k,n})=f_k$ for all $1\leq k\leq n.$
We now define a functor $\ov\Phi:\cF_\infty\to \Gr$ satisfying $\ov\Phi(\infty)=\Ga_\infty$ and
$$\ov\Phi(f_k)(\iota_n(g)):= \iota_{n+1}(\Phi(f_{k,n})(g)), 1\leq k\leq n, g\in \Ga_n.$$ 
The functor $\ov\Phi$ is in some sense obtained by looking at the functor $\Phi$ at infinity. 
We say that $\ov\Phi$ is a \textit{lift} of the functor $\Phi.$
Without talking about functors and categories we have defined an action by group automorphisms of the monoid $\cF_\infty$ on the group $\Ga_\infty.$
We realize that the bicrossed product $\cF_\infty\bowtie \Ga_\infty$ induced by the cloning system is in fact a usual semidirect product $\cF_\infty\ltimes \Ga_\infty$. 
The action $\cF_\infty\act \Ga_\infty$ used to construct the semidirect product is precisely the one given by the functor $\ov\Phi.$

In conclusion, a pure cloning system gives an action $\cF_\infty\act \Ga_\infty$ and provides a group $G$.
Moreover, the pure cloning system defines a covariant functor $\Phi:\cF\to\Gr$ whose associated fraction group $K\rtimes F$ is isomorphic to $G$.
Conversely, a covariant functor $\Phi:\cF\to\Gr$ admitting a lift $\ov\Phi:\cF\to\Gr$ so that $\Phi(f)$ is injective for all forests provides a pure cloning system.
Hence, morally covariant functors $\Phi:\cF\to\Gr$ and pure cloning systems are very similar. 
They both produce a class of groups and these two classes contain a large common subclass.
Note that all these groups are built from $F$ rather than other Thompson groups since we did not incorporated any permutations or braids to the forests.

\subsection*{Incorporating permutations and braids}

One of the main difference between the two technologies resides in incorporating permutations or braids in the picture. 
It means that the group constructed will look like more of an extension of $F,T,V$.
The technology of cloning systems seems in general better adapted to this purpose than Jones actions and able to produce interesting groups containing $F,T,V$ and far from being semidirect products like in the Jones construction.



If $G$ is a group produced via a cloning system, then the properties of the cloning maps will decide if $G$ is closer to Thompson group $F,T,V$ or something in between (e.g.~$V^{mock}$ built via mock permutations) or beyond (e.g.~$BV$). 
If the cloning maps are not group morphisms, then necessarily the maps $\rho_n:\Ga_n\to S_n, n\geq 1$ are nontrivial. 
The data of the cloning systems is then more complex and produce groups closer to the larger Thompson groups $T,V,BV$, etc.
This is a key feature of this formalism that works very differently with respect to the one of Jones.

In Jones formalism, there are two ways to obtain actions of larger groups than $F$.
First, if one has a monoidal functor $\Phi:\cF\to\Gr$, then the Jones action $F\act K$ can be extended to $V\act K$ via permuting tensors.
Second, in the non-monoidal context we enlarge the initial category $\cF$.
This means that instead of considering the category of forests $\cF$ we will consider a larger category: for instance the category of affine forests $\mathcal{AF}$ which corresponds to forests together with cyclic permutations of their leaves.
Formally, the morphism spaces $\mathcal{AF}(n,m)$ are in bijection with $\cF(n,m)\times \Z/m\Z$ for $1\leq n\leq m.$
To apply Jones technology we consider functors $\Phi:\mathcal{AF}\to \Gr$ starting from $\mathcal{AF}.$ 
Such a functor provides a semidirect products $K\rtimes T$.
Now, replacing $\Z/m\Z,m\geq 1$ by the group of all permutations $S_m$, mock permutations $S_m^{mock}$, braid groups with $m$ strands $B_m$, etc.,  we obtain larger categories $\mathcal{VF},\mathcal{F}^{mock}, \mathcal{BV}$. 
A functor from one of these categories $\mathcal{AF},\mathcal{VF},\mathcal{F}^{mock}, \mathcal{BV}$ to the category of groups produces a semidirect product $K\rtimes H$ where $H$ is equal to $T,V,V^{mock},BV$ respectively.
Some cloning systems where the cloning maps are not group morphisms will sometime corresponds to a functor $\Phi:\cC\to\Gr$ where $\cC$ is one of these larger categories of forests. 
However, there are no longer a clear correspondence between the two frameworks.
The defect of the cloning maps of not being group morphisms is, in the Jones framework, hidden in the categorical structure of $\cC$; structure which is more complex than the categorical structure of $\cF$. 
Of course, the more $\cC$ is complex (in terms of its presentation), the more difficult it is to construct functors $\Phi:\cC\to\Gr$.

Transposing this functorial approach in cloning systems would correspond to replacing the monoid of forests $\cF_\infty$ by a larger monoid such as $\mathcal{VF}_\infty$ incorporating finitely supported permutations. 
We would then obtain a $V$-like group via the construction of a Brin-Zappa-Szep product $\mathcal{VF}_\infty\bowtie \Ga_\infty.$
The construction of Witzel and Zaremsky is not following this scheme but rather makes the group $\Ga_\infty$ and the cloning maps more complex but always keeping the same monoid of forests $\cF_\infty.$

\subsection*{Constructing the groups considered in this article with cloning systems}

We end this appendix by explaining how the groups $K\rtimes V$ considered in this article can be constructed using cloning systems. 
It is interesting to compare the constructions of $K\rtimes F$ and $K\rtimes V$ using cloning systems which we will do at the end of this analysis.
Consider a covariant monoidal functor $\Phi:\cF\to\Gr$ as considered in this present article. 
This functor is completely described by the triple $(\Ga,\al_0,\al_1)$ where $\Ga:=\Phi(1)$ is a group and $(\al_0,\al_1):=\Phi(Y)\in\Hom(\Ga,\Ga\oplus\Ga)$ is a group morphism.
This produces a Jones action $F\act K$ that extends to an action $V\act K$ using the monoidality of $\Phi.$
Since $\Phi$ is monoidal we can assume, using an argument due to Tanushevski, that $g\mapsto (\al_0(g),\al_1(g))$ is injective, see Proposition \ref{prop:Tanushevski}.
Define $\Ga_n:=\Ga^n\rtimes S_n$ where $S_n\act \Ga^n$ acts by permuting indices $n\geq 1.$
There is an obvious directed system structure on $(\Ga_n,n\geq 1)$ given by the inclusions $\{1,\cdots,m\}\subset \{1,\cdots,n\}$ for $1\leq m\leq n.$
One can see that the direct limit of this system is the restricted permutational wreath product $\Ga_\infty:= \oplus_{n\in \N}\Ga\rtimes S_\N$ where $S_\N$ is the group of finitely supported permutations on the set of nonzero natural numbers $\N$ acting by shifting indices. 
Define $\rho_n:\Ga_n\to S_n$ to be the quotient maps for $n\geq 1$.
Now for $1\leq k\leq n$ and $g\in \Ga^n$ we define $\kappa_k^n(g):=\Phi(f_{k,n})(g)\in \Ga^{n+1}$ and thus its $j$th entry $1\leq j\leq n+1$ is $$\kappa_k^n(g)_j=\begin{cases} g_j \text{ if } 1\leq j<k\\ \al_0(g_j) \text{ if } j=k\\ \al_1(g_j) \text{ if } j=k+1 \\ g_{j-1}$ if $k+2\leq j\leq n+1\end{cases}.$$
In particular, the cloning map restricted to direct sums of $\Ga$ are group morphisms.
However, they are not when applied to permutations. 
Given $\sigma\in S_n$ one can define $\kappa_k^n(\sigma)\in S_{n+1}$ in a diagrammatic way by 2-cabling the $k$th strand of $\sigma$, see \cite[Example 2.2]{Zaremsky18-clone} or \cite[Section 2]{Brothier19WP} for details.
Finally, define $\kappa_k^n(g\sigma):=\kappa_k^n(g)\kappa_k^n(\sigma)$ for $1\leq k\leq n$, $g\in\Ga^n, \sigma\in S_n$.
We have defined a cloning system and one can check that the associated group $G$ is isomorphic to $K\rtimes V$. 

Using cloning systems, the construction of $K\rtimes F$ is much easier than the one of $K\rtimes V$. 
To obtain $K\rtimes F$ we would consider the group $\Ga_n:=\Ga^n$ rather than $\Ga^n\rtimes S_n$, trivial maps $\rho_n:\Ga_n\to S_n$, and consider the restriction of the cloning maps of above to $\Ga^n$ for $n\geq 1.$ 
In particular, this cloning systems is pure as previously observed.

\subsection{Tanushevski's construction}

Tanushevski has independently developed a framework for constructing a large class of groups \cite{Tanushevski16,Tanushevski17}. 
He constructed in a very concrete way a large family of groups for which he successfully studied, among others, questions regarding the lattice of normal subgroups, generators, presentations and finiteness properties.
We are going to see that his construction can be completely reinterpreted using Jones technology.

\subsubsection*{Construction of Tanushevski}

Consider any group $\Ga$ and any group morphism $\al:\Ga\to\Ga\oplus\Ga, g\mapsto (\al_0(g),\al_1(g)).$
Using $(\Ga,\al)$, Tanushevski constructs a limit group $\sqcap_\al(\Ga)$ and an action $F\act \sqcap_\al(\Ga).$
Moreover, he notes that this action can be extended into $V\act \sqcap_\al(\Ga).$
He then considers the following three semidirect products 
$$\cF_\al(\Ga):= \sqcap_\al(\Ga)\rtimes F, \mathcal{T}_\al(\Ga):=\sqcap_\al(\Ga)\rtimes T \text{ and } \mathcal{V}_\al(\Ga):=\sqcap_\al(\Ga)\rtimes V.$$
He provides descriptions of elements of $\cF_\al(\Ga):= \sqcap_\al(\Ga)\rtimes F$ using pairs of trees whose leaves are labelled by elements of $\Ga.$

\subsubsection*{Comparison of the two constructions}

We now compare this construction and description using the technology of Jones.
Having a group $\Ga$ and a morphism $\al:\Ga\to\Ga\oplus\Ga$ is equivalent to have a covariant monoidal functor 
$\Phi:\cF\to\Gr$
from the category of forests to the category of groups equipped with direct sums for its monoidal structure.
The group $\sqcap_\al(\Ga)$ and the action $F\act \sqcap_\al(\Ga)$ are the direct limit group of Jones and the Jones action respectively that we have denoted in this article by $K$ and $\pi:F\act K$ respectively.
The extension of the action $F\act \sqcap_\al(\Ga)$ to an action $V\act \sqcap_\al(\Ga)$ of the Thompson group $V$ is the exact same extension that we have been using in this article.
The semidirect products $\cF_\al(\Ga)$ and the description of the elements as labelled pairs of trees is exactly the same than the one we gave in \cite{Brothier19WP}. 
%
It is remarkable to find the same construction of fraction groups in the work of Tanushevski. 
Hence, the construction and diagrammatic description of a group of Tanushevski is equivalent to considering a Jones action and the diagrammatic description of the author applied to a covariant monoidal functor $\Phi:\cF\to\Gr.$ 
As observed in the previous section, this is equivalent to consider pure cloning systems but requiring an additional property on the cloning which translates the monoidality of the functor $\Phi.$

\subsubsection*{Comparison of questions studied}
The construction of Tanushevski and the construction of groups presented in this article are rather identical and do not need any further comparison. 
We then present how different are the questions studied by Tanushevski and the author.
The two studies are luckily and interestingly almost completely disjoint and thus are complementary. 
The author has been considering analytical properties (such as the Haagerup property), classifications up to isomorphisms and description of automorphism groups for the groups $K\rtimes V$.
Tanushevski studied rather exclusively the class of groups $\sqcup_\al(\Ga)\rtimes F$ (written $K\rtimes F$ in our article) studying questions regarding finiteness properties and regarding the lattice of normal subgroups of $\sqcup_\al(\Ga)\rtimes F$ for instance.
Some of his work can easily be adapted to our larger group $K\rtimes V$ and for instance we used it to obtain Proposition \ref{prop:Tanushevski}.
He provides explicit presentations of $K\rtimes F$ that can then be completed into presentations of $K\rtimes T$ and $K\rtimes V$.
For example, given a finite presentation of a group $\Ga$ and an injective group morphism $\al=(\al_0,\al_1)$, thanks to Tanushevski, one can write an explicit finite presentation to $K\rtimes F$ and to $K\rtimes T$ and $K\rtimes V$.

Tanushevski provides a very strong results on finiteness properties.
Fix a natural number $n$ and a group $\Ga$.
He proved that the fraction group $K\rtimes F$ associated to $(\Ga,\id_\Ga,\id_\Ga)$ (in our language) is of type $F_n$ if and only if $\Ga$ is of type $F_n$. 
We do not know if this result extends to the larger semidirect products $K\rtimes V$ nor if this extends to twisted fraction groups, i.e.~to fractions constructed from triples $(\Ga,\al_0,\al_1)$ where $(\al_0,\al_1)\neq (\id_\Ga,\id_\Ga).$
Moreover, it would be interesting to know if some of the theorems we proved regarding classification of the groups $K\rtimes V$ and the description of their automorphism groups could be adapted to $K\rtimes F$ (or other semidirect products like $K\rtimes T, K\rtimes V^{mock}$, etc). 
Although, it is not that clear that there is a rigidity phenomena regarding isomorphisms when we consider semidirect products $K\rtimes F$.
Indeed, we are using in our proofs that the action $V\act \Q_2$ is transitive in a very strong sense. 
However, the action $F\act \Q_2$ is not even transitive. 
Moreover, the difference of complexity between the automorphism groups $\Aut(F)$ and $\Aut(V)$ suggests that the automorphism group of $K\rtimes F$ may differ greatly from the automorphism group of $K\rtimes V$ \cite{Brin96-chameleon,BCMNO19}.
This makes these questions even more interesting and intriguing.



\newcommand{\etalchar}[1]{$^{#1}$}

\end{document}